\documentclass{article} 
\usepackage{geometry}
\usepackage{amsmath}
\usepackage{amssymb}
\usepackage{amsthm}
\usepackage{graphicx}
\usepackage{subcaption}
\usepackage{float}
\usepackage{multirow}
\usepackage[normalem]{ulem}
\usepackage[utf8]{inputenc}
\usepackage{csquotes}
\usepackage[english]{babel}
\usepackage{setspace}
\usepackage{algorithm}
\usepackage{enumitem}
\usepackage{mathpazo}
\usepackage{mathtools}
\usepackage{bbm}
\usepackage{booktabs}
\usepackage{centernot}

\usepackage{cancel}

\usepackage{tikz}
\usetikzlibrary{shapes,decorations.pathreplacing,arrows,calc,arrows.meta,fit,positioning}
\tikzset{
    -Latex,auto,node distance =1 cm and 1 cm,semithick,
    state/.style ={circle, draw, minimum width = 1 cm, inner sep=0pt},
    point/.style = {circle, draw, inner sep=0.04cm,fill,node contents={}},
    bidirected/.style={Latex-Latex,dashed},
    el/.style = {inner sep=2pt, align=left, sloped}
}

\useunder{\uline}{\ul}{}
\usepackage[round]{natbib}
\usepackage[colorlinks=true, linkcolor=blue]{hyperref}

\newtheorem{theorem}{Theorem}[section]
\newtheorem{lemma}[theorem]{Lemma}
\newtheorem{corollary}[theorem]{Corollary}
\newtheorem{proposition}[theorem]{Proposition}

\newtheorem*{theorem*}{Theorem}
\newtheorem*{lemma*}{Lemma}
\newtheorem*{corollary*}{Corollary}
\newtheorem*{proposition*}{Proposition}
\newtheorem*{conjecture*}{Conjecture}

\theoremstyle{definition}
\newtheorem{definition}{Definition}
\newtheorem*{definition*}{Definition}
\theoremstyle{definition}
\newtheorem{example}{Example}
\theoremstyle{definition}
\newtheorem*{example*}{Example}
\theoremstyle{definition}

\theoremstyle{definition}
\newtheorem*{assumption*}{Assumption}
\theoremstyle{definition}

\theoremstyle{remark}
\newtheorem{remark}{Remark}[section]
\theoremstyle{remark}
\newtheorem*{remark*}{Remark}

\DeclareMathOperator{\cov}{cov}
\DeclareMathOperator{\var}{var}
 
\DeclareMathOperator*{\argmin}{arg\,min} 
\DeclareMathOperator{\Tr}{Tr}
\DeclareMathOperator{\supp}{supp}

\newcommand{\E}{\mathbb{E}}
\newcommand{\prob}{\mathbb{P}}

\newcommand{\AND}{\text{ and }}
\newcommand{\given}{\,|\,}
\newcommand{\indep}{\perp \!\!\! \perp }
\newcommand{\T}{^\top}

\DeclareMathOperator{\pa}{pa}
\DeclareMathOperator{\an}{an}
\DeclareMathOperator{\de}{de}
\DeclareMathOperator{\nd}{nd}

\DeclareMathOperator{\ch}{ch}

\newcommand{\bss}{\text{BSS}}

\newcommand{\klbss}{\text{KL-BSS}}
\newcommand{\compare}{\textsc{Compare}}

\newcommand{\dd}{d}
\newcommand{\sps}{s}
\newcommand{\ubsps}{\overline{s}}
\newcommand{\betam}{\beta_{\min}}
\newcommand{\betaspace}{\Theta}
\newcommand{\Sigmaspace}{\Omega}

\newcommand{\supps}{\mathcal{T}}
\newcommand{\suppspace}{\supps_{\dd,\sps}}
\newcommand{\suppspaceub}{\supps_{\dd}^{\ubsps}}
\newcommand{\penalt}{\tau}
\newcommand{\sigmam}{\sigma_{\min}}
\newcommand{\trusupp}{S_{*}} 
\newcommand{\estsupp}{\widehat{S}}
\newcommand{\mclass}{\mathcal{M}} 
\newcommand{\Sigmaspaceb}{\Sigmaspace_{\textup{B}}} 
\newcommand{\Sigmaspacek}{\Sigmaspace_{\textup{K}}} 
\newcommand{\eigvalb}{\lambda_{\textup{B}}} 
\newcommand{\eigvalk}{\lambda_{\textup{K}}}

\newcommand{\signal}{\Delta}
\newcommand{\signalt}{\widetilde{\Delta}}
\newcommand{\signalb}{\overline{\Delta}}
\newcommand{\signalone}{\Delta_1}
\newcommand{\signaltwo}{\Delta_2}
\newcommand{\signalltwo}{\widetilde{\Delta}_2}
\newcommand{\signalthree}{\Delta_3}
\newcommand{\regcoef}{\widetilde{\alpha}_\beta}
\newcommand{\ptregcoef}{\alpha_\beta}

\DeclareMathOperator{\Rad}{Rad}
\DeclareMathOperator{\Unif}{Unif}

\usepackage{authblk}

\begin{document}

\title{\klbss{}: Rethinking optimality for \\ neighbourhood selection in structural equation models}

\author[]{Ming Gao}
\author[]{Wai Ming Tai}
\author[]{Bryon Aragam}
\affil[]{\emph{University of Chicago}}

\date{\vspace{-8ex}}

\maketitle

{\let\thefootnote\relax\footnote{Contact: \texttt{\{minggao,waiming.tai,bryon\}@chicagobooth.edu}}}

\begin{abstract}
    We introduce a new method for neighbourhood selection in linear structural equation models that improves over classical methods such as best subset selection (\bss{}) and the Lasso. Our method, called \klbss{}, takes advantage of the existence of underlying structure in SEM---even when this structure is \emph{unknown}---and is easily implemented using existing solvers. Under weaker eigenvalue conditions compared to \bss{} and the Lasso, \klbss{} can provably recover the support of linear models with fewer samples. We establish both the pointwise and minimax sample complexity for support recovery, which \klbss{} obtains. Extensive experiments on both real and simulated data confirm the improvements offered by \klbss{}. While it is well-known that the Lasso encounters difficulties under structured dependencies, it is less well-known that even \bss{} runs into trouble as well, and can be substantially improved. These results have implications for structure learning in graphical models, which often relies on neighbourhood selection as a subroutine.
\end{abstract}

\section{Introduction}
\label{sec:intro}
Graphical models are commonly used for modeling complex systems with nontrivial dependence among the variables. They have been successful in machine learning, causal inference, and applications in scientific domains like medicine and genetics.
In practice, when the structure of a graphical model is unknown in advance, it needs to be inferred from the data.
A basic operation to learn the structure of a graphical model is the estimation of the neighbourhood of a given node. Under fairly general assumptions, this problem reduces to the familiar problem of variable selection, a.k.a. support recovery, and has been extensively studied as a prototypical model selection problem \citep[e.g.][see Section~\ref{sec:intro:related} for more discussion]{shibata1981optimal,nishii1984asymptotic,foster1994risk,shao1997,meinshausen2009lasso,wainwright2009information,ndaoud2020optimal,jin2014optimality,wang2010informationl,aksoylar2016sparse}.
Despite this long line of work, existing results are insufficient for understanding the nuances of neighbourhood selection
in graphical models with structured dependencies.
There is an exception for \emph{undirected}, Markov random fields, for which much is now known, including optimal estimators of the neighbourhood \citep[for an overview, see][]{drton2017structure,loh2018neighborhood}.
In the setting of \emph{directed}, structural equation models (SEM), however, although regression is widely used for neighbourhood selection in practice, a detailed understanding of the tradeoffs---both practical and theoretical---in neighbourhood selection (in particular, lower bounds on the risk), is missing.

This leads to a simple, fundamental question:
\begin{quote}
    \emph{Are existing support recovery techniques adequate for neighbourhood selection in SEM?}
\end{quote}
The obvious candidates are best subset selection \citep[\bss{}, see e.g.][]{miller2002subset} and the Lasso \citep{tibshirani1996regression}. \bss{} is known to be effective for support recovery with general random design matrices \citep{wainwright2009information}, whereas the Lasso requires nearly orthogonal designs \citep[i.e. the irrepresentability condition;][]{zhao2006model,wainwright2009sharp}. 
By imposing structural assumptions through the (directed) Markov property, SEM represent a potential middle ground between nearly orthogonal designs (required by the Lasso) and general design (allowed by \bss{}).
Thus, the question is whether the neighbourhood selection problem for SEM is essentially equivalent to general design regression or---if not---how to leverage \emph{unknown} structure to design better estimators of the structure itself.
Crucially, we emphasize that we \emph{will not} assume the structure of the SEM itself is known.

In this paper, we begin to answer this fundamental question and in doing so, propose a new method for
support recovery in SEM that illustrates the deficiencies with existing methods and shows how they can be improved. 
We quantify these deficiencies via an analysis of the minimax rate of support recovery in SEM, which is achieved by our approach, as well as its pointwise rates.
A major takeaway is that it is often \emph{easier} to recover the neighbourhood in linear SEM, even when its structure is unknown.
This shows that worst-case analyses
under general random designs are overly pessimistic, and the mere existence of structure can simplify the recovery problem. 
Our method, called \klbss{}, locates the hidden signature left by the unknown SEM structure, obtaining significant improvements in accuracy, and is easily implemented. See Figure~\ref{fig:demo}. In fact, in the worst-case, our method performs no worse than \bss{} on average. In the remainder of this section, we provide a brief overview of our approach, discuss our main technical contributions, and review related work.

\subsection{Overview}\label{sec:intro:overview}

\begin{figure}[t]
    \centering
    \begin{minipage}[b]{0.65\linewidth}
    \begin{tikzpicture}[scale=0.8, transform shape]
        
        \node (y) {};
        \node[state, fill=gray, fill opacity=0, text opacity=1] (x1) [right = of y, yshift=1cm] {$X_1$};
        \node[state, fill=gray, fill opacity=0, text opacity=1] (x2) [right = of y, yshift=-1cm] {$X_2$};
        \node[state, fill=gray, fill opacity=0.1, text opacity=1] (Y) [right = of x2, yshift=-1cm] {$Y$};
        \node[state, fill=gray, fill opacity=0.5, text opacity=1] (x3) [right = of x1, yshift=1cm] {$X_3$};
        \node[state, fill=gray, fill opacity=0.5, text opacity=1] (x4) [right = of x1, yshift=-1cm] {$X_4$};
        \node[state, fill=gray, fill opacity=0.5, text opacity=1] (x5) [right = of x4, yshift=1cm] {$X_5$};
        \node[state, fill=gray, fill opacity=0.5, text opacity=1] (x6) [right = of x4, yshift=-1cm] {$X_6$};
        \node[state, fill=gray, fill opacity=0.5, text opacity=1] (x7) [right = of x5, yshift=1cm] {$X_7$};
        \node[state, fill=gray, fill opacity=0.5, text opacity=1] (x8) [right = of x5, yshift=-1cm] {$X_8$};
        \node[state, fill=gray, fill opacity=0.5, text opacity=1] (x9) [right = of x8, yshift=1cm] {$X_9$};
        \node[state, fill=gray, fill opacity=0.5, text opacity=1] (x10) [right = of x8, yshift=-1cm] {$X_{10}$};

        \path (x1) edge [line width=.5mm] (Y);
        \path (x2) edge [line width=.5mm] (Y);
        \path (x1) edge (x3);
        \path (x1) edge (x4);
        \path (x2) edge (x4);
        \path (x3) edge (x7);
        \path (x4) edge (x5);
        \path (x4) edge (x6);
        \path (x5) edge (x7);
        \path (x5) edge (x8);
        \path (x5) edge (x9);
        \path (x6) edge (x8);
        \path (x6) edge (x10);
        \path (x7) edge (x9);
        
        \node [draw=black,dashed,fit=(x1) (x2), inner sep=0.2cm, ultra thick] (S1) {};
        \node [draw=black,dashed,fit=(x3) (x4), inner sep=0.2cm] (S2) {};
        \node [draw=black,dashed,fit=(x5) (x6), inner sep=0.2cm] (S3) {};
        \node [draw=black,dashed,fit=(x7) (x8), inner sep=0.2cm] (S4) {};
        \node [draw=black,dashed,fit=(x9) (x10), inner sep=0.2cm] (S5) {};
        \draw [decorate, decoration = {brace, amplitude=5pt}, -] ([yshift=.1cm] S1.north west) -- ([yshift=.1cm] S1.north east) node[midway, yshift=.2cm]{\Large $S_1 (\text{truth})$};
        \draw [decorate, decoration = {brace, amplitude=5pt}, -] ([yshift=.1cm] S2.north west) -- ([yshift=.1cm] S2.north east) node[midway, yshift=.2cm]{\Large $S_2$};
        \draw [decorate, decoration = {brace, amplitude=5pt}, -] ([yshift=.4cm] S3.north west) -- ([yshift=.4cm] S3.north east) node[midway, yshift=.2cm]{\Large $S_{3}$};
        \draw [decorate, decoration = {brace, amplitude=5pt}, -] ([yshift=.1cm] S4.north west) -- ([yshift=.1cm] S4.north east) node[midway, yshift=.2cm]{\Large $S_4$};
        \draw [decorate, decoration = {brace, amplitude=5pt}, -] ([yshift=.1cm] S5.north west) -- ([yshift=.1cm] S5.north east) node[midway, yshift=.2cm]{\Large $S_{5}$};

        \node[below = of y, yshift=-1.7cm] (Ebeta)  {\Large $\E\widehat{\beta}(S_j)=$};
        \node[right = of Ebeta, xshift=-1cm] (beta1) {$\begin{pmatrix}1 \\ 1\end{pmatrix}$};
        \node[right = of beta1, xshift=-0.1cm] (beta2) {$\begin{pmatrix}0.2 \\ 0.6\end{pmatrix}$};
        \node[right = of beta2, xshift=-0.3cm] (beta3) {$\begin{pmatrix}0.28 \\ 0.28\end{pmatrix}$};
        \node[right = of beta3, xshift=-0.3cm] (beta4) {$\begin{pmatrix}0.16 \\ 0.16\end{pmatrix}$};
        \node[right = of beta4, xshift=-0.4cm] (beta5) {$\begin{pmatrix}0.11 \\ 0.1\end{pmatrix}$};

    \end{tikzpicture}
    \end{minipage}%
    \begin{minipage}[b]{0.35\linewidth}
        \includegraphics[width=1.\linewidth]{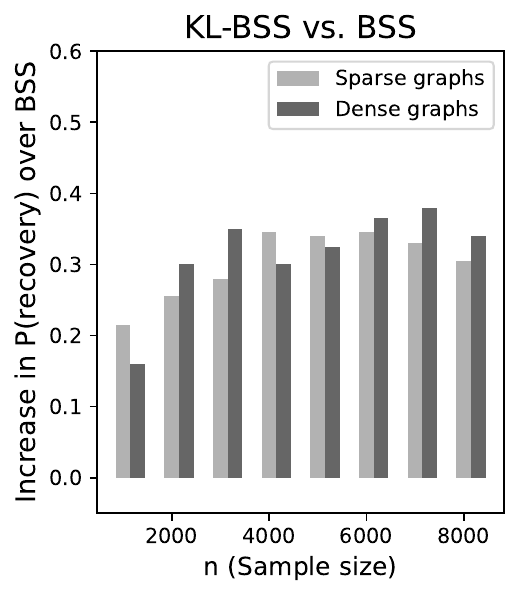}
    \end{minipage}
    \caption{
    Overview of SEM and improvement of \klbss{}.
    (Left) An example SEM over $d=11$ nodes. The target variable is $Y$, the neighbourhood of $Y$ is $S_1=\{X_1,X_2\}$, and the remaining nodes $X=(X_3,\ldots,X_{10})$ are shaded. The (partial) regression coefficients $\E\widehat{\beta}(S)=\E(X_S\T X^{}_S)^{-1}X_S\T Y$ are computed for the support candidates $S=S_j$ ($j=1,2,\ldots,5$). For simplicity, we only present a subset of all possible supports. 
    (Right) \klbss{} strictly improves over \bss{} in support recovery:
    An illustration of the improvement for both sparse and dense graphs, summarized from the results in Section~\ref{sec:expt:simu}. 
    }
    \label{fig:demo}
\end{figure}

To provide context for our results, we begin by reviewing the support recovery problem in linear models, of which neighbourhood selection in SEM can be viewed as a special case. To fix notation, consider the prototypical Gaussian linear model:
\begin{align}
\label{eq:lm}
    Y = X\T \beta + \epsilon, \qquad 
    X\sim\mathcal{N}(\mathbf{0}_\dd,\Sigma),\quad
    \epsilon\sim\mathcal{N}(0,\sigma^2),\quad 
    X\indep \epsilon\,,
\end{align}
where $\beta\in\mathbb{R}^\dd$ is the regression coefficient vector and $\Sigma\in\mathbb{R}^{\dd\times \dd}$ is the design covariance matrix. We are interested in the case where the design $\Sigma$ arises from an SEM over $X$ (Figure~\ref{fig:demo}).
The support recovery problem is to recover the nonzero entries in $\beta$; in SEM these entries correspond to the direct parents of a node in the graph. In modern high-dimensional settings where the number of variables $\dd$ grows with the sample size $n$, it is natural to impose sparsity $\|\beta\|_0 =\sps\le \dd$.
Formal preliminaries, including graphical model background and the connection to neighbourhood selection, will be deferred until Section~\ref{sec:pre}.

\begin{remark}
    Throughout this paper, when we refer to ``structure'', we exclusively mean the structure induced on the design $\Sigma$ through an SEM, as opposed to other structural assumptions such as sparsity in $\beta$. For details on this setup, see Section~\ref{sec:pre}.
\end{remark}

Our proposed estimator is based on \bss{}, which we briefly recall here for completeness.
\bss{} searches over all possible candidate supports of size $\sps$ (denoted by $\suppspace$) and outputs the one that minimizes the residual variance in $Y$, i.e.
\begin{align}\label{eq:bss}
    \estsupp^{\bss{}} = \argmin_{S\in\suppspace}\|Y - X_S\T \widehat{\beta}(S)\|^2\,,
\end{align}
where $\widehat{\beta}(S):=(X_S\T X_S^{})^{-1}X_S\T Y$ is the OLS estimate of $Y$ on some subset $S$ of covariates, indicated by $X_S$.
One way to interpret this is as a tournament among candidate supports: For any pair of candidate supports $S,T\in\suppspace$, \bss{} compares them using the residual variance as a score and keeps track of the winner until the best candidate is found. The idea is that the true support will have the smallest residual variance with high probability and thus ``win'' this tournament.
By using the residual variance in this way, \bss{} treats each candidate set $S$ equally. But not all candidate sets in an SEM are equal: Due to the way that information propagates in SEM, some alternative supports will have small (partial) regression coefficients $\widehat{\beta}(S)$, which in principle can be identified and ruled out. 
Figure~\ref{fig:demo} demonstrates this on a simple SEM, alongside a summary of the actual improvement obtained by our proposed method, \klbss{}, compared to \bss{} in randomly generated SEM.

To illustrate this phenomenon, consider the following simple example.

\begin{example}\label{ex:intro}
    Consider the simplest possible nontrivial SEM: $X_1\to X_2$ with a target variable $Y$ that depends only on $X_1$. Thus, the SEM is given by
    \begin{align}\label{eq:intro:sem}
        \left\{\begin{aligned}
            X_1 &= \epsilon_1, & \epsilon_1\sim\mathcal{N}(0,\sigma_1^2) \\
            X_2 &= bX_1+\epsilon_2, & \epsilon_2\sim\mathcal{N}(0,\sigma_2^2) \\
            Y &= \beta X_1 + \epsilon, & \epsilon\sim\mathcal{N}(0,\sigma^2)
        \end{aligned}
        \right.
        \qquad\implies\qquad
        \Sigma = \cov(X) = 
        \begin{pmatrix}
        \sigma_1^2 & b\sigma_1^2 \\
        b\sigma_1^2 & b^2\sigma_1^2 + \sigma_2^2
        \end{pmatrix}.
    \end{align}
    Here, $\beta$ is a scalar. For simplicity in the calculations to follow, we set $\sigma^2_1=\sigma^2_2$ and suppress further dependence on them since this does not change any of the conclusions.
    
    It is clear where the difficulty in variable selection lies: For even moderate sizes of $b$, $X_1$ and $X_2$ will be highly correlated, and so any variable selection method will struggle to correctly distinguish $X_1$ from $X_2$ as the ``correct'' parent of $Y$.
    A closer look, however, reveals an asymmetry: The ``true'' regression coefficient of $Y$ on $X_1$ is $\beta$, whereas the partial regression coefficient of $Y$ on $X_2$ is 
    \begin{align*}
        \Sigma_{22}^{-1}\Sigma_{21}\beta = \frac{1}{b+1/b}\cdot \beta 
        \le \frac\beta2.
    \end{align*}
    In other words, although the strong dependence between $X_1$ and $X_2$  makes it hard to distinguish them solely based on the residual variances, the ``true'' coefficient is substantially larger than the ``wrong'' coefficient. This difference can be leveraged for identification, as the true model carries more signal than the incorrect model.

    This asymmetry is ignored by \bss{}, which suggests its performance can be improved. Indeed, this can be made quite precise: The success of \bss{}, as with many variable selection methods, depends on a certain eigenvalue $\eigvalb(\Sigma)$ of the design covariance matrix $\Sigma$ (cf. Section~\ref{sec:analysis:eigen}; equations (\ref{eq:eigval:bss}-\ref{eq:eigval:klbss}) for the definitions). In this simple example, it turns out that
    \begin{align*}
        \eigvalb(\Sigma)
        = \frac1{1+b^2},
        \tag{\bss{}}
    \end{align*}
    which implies that the sample complexity of \bss{} scales with $b^2$.
    By contrast, the success of our procedure depends on a different eigenvalue-like quantity $\eigvalk(\Sigma)$, which can be bounded independently of $b$ in \eqref{eq:intro:sem}:
    \begin{align*}
        \eigvalk(\Sigma)
        \ge 1.
        \tag{\klbss{}}
    \end{align*}
    As a result, the sample complexity of our procedure will not depend on $b$.
    Thus, there is a quadratic sample complexity gap, which is significant when $b$ is even moderately large (cf. Figure~\ref{fig:ex:pathcancel} in Section~\ref{sec:analysis}).

    These eigenvalues represent the amount of signal in design covariance captured by each method; in other words, as $b\to\infty$, the signal picked up by \bss{} vanishes whereas \klbss{} captures a strong signal regardless of $b$. 
    Moreover, in the opposite scenario with $b\to0$, we see that $\eigvalb(\Sigma)\to1$, which is never larger than $\eigvalk(\Sigma)$. At best, \bss{} matches \klbss{} when its signal is strong. Phrased differently, in the worst case \klbss{} performs no worse than \bss{}. 
    We will see that this is because \bss{} ignores the crucial signal carried by the asymmetry in the partial regression coefficients above.
\end{example}

While illustrative, actually exploiting this asymmetry in general SEM is of course more complicated, and making this all precise along with determining how small is small enough to rule out requires some care. 
We will adopt the same ``tournament'' strategy as \bss{}, but choose winners differently
by modifying the score to compare candidates, accounting for the (partial) regression coefficients.
Instead of relying solely on the residual variance \eqref{eq:bss}, \klbss{} incorporates an additional term (cf.~\eqref{eq:compare:score}) that penalizes the small ``wrong'' coefficient. This extra term comes from the Kullback-Leibler (KL) divergence between the true model and the closest alternative model, hence the name ``\klbss{}'' (see Appendix~\ref{sec:disc:interpret} for details).
The result is a new procedure for support recovery in SEM that significantly outperforms \bss{} when $X$ is generated by an SEM. 
The resulting analysis is somewhat delicate: 
A key theme throughout the paper is that understanding these practical issues requires, at a technical level, a careful understanding of the roles played by the design matrix and its eigenvalues.

\subsection{Contributions}\label{sec:intro:contribution}

Our main contribution is to introduce \klbss{}, a novel method for neighbourhood selection in SEM that improves over classical approaches while highlighting the deficiencies of these approaches.
Specifically: 
\begin{enumerate}
    \item We study both the pointwise (Theorem~\ref{thm:pointwise:eigen}) and minimax (Theorem~\ref{thm:opt:eigen}) sample complexity of support recovery by developing an appropriate eigenvalue condition for \klbss{} and comparing this to existing eigenvalue conditions for \bss{}, as alluded to in Example~\ref{ex:intro}. 
    \klbss{} achieves the optimal sample complexity and requires fewer samples over a broad class of design matrices that naturally arise in SEM.
    \item Through numerous examples (Section~\ref{sec:analysis:sem}), we contrast the behaviour of \klbss{}, \bss{}, and the Lasso. We also show how SEM more easily satisfy the eigenvalue conditions needed by \klbss{}, whereas the corresponding conditions required by other methods for recovery typically fail.
    \item We implement \klbss{} using standard solvers with open-source code available at \url{https://github.com/MingGao97/KL-BSS}. Similar in nature to \bss{}, the overall computational complexity of \klbss{} is of the same order (Section~\ref{sec:prac}).
    \item We perform a comprehensive evaluation of \klbss{} (Section~\ref{sec:expt}), comparing to \bss{} and the Lasso in simulations and an application using pan-cancer gene expression data. Given our motivation in structure learning, we also evaluate \klbss{} as a subroutine for learning the structure of SEM. Overall, our experiments indicate that \klbss{} indeed outperforms classical methods when the covariates possess underlying structure in the form of an SEM that is unknown to the statistician.
\end{enumerate}
\klbss{} is based on a novel KL-decomposition of the support recovery problem that precisely captures the signal that \bss{} misses, which may be a surprise given the folklore wisdom that \bss{} is the gold standard for support recovery.
Moreover, the analysis is nontrivial and somewhat technical out of necessity: It turns out that existing methods and analyses are also optimal for the simplest standard design ($\Sigma=I_\dd$), as well as general designs. To resolve this, we develop novel tail probability bounds for random quadratic programs using tools from random matrix theory, which may be of independent interest.

\subsection{Related work}\label{sec:intro:related}
The literature on support recovery, variable selection, and sparse regression is vast, and we do not intend to attempt a comprehensive review. Only some of the most important or relevant results are discussed here,
with a particular focus on sample complexity results for the exact recovery risk $\mathbbm{1}\{\estsupp\ne \trusupp\}$ where $\estsupp$ is the estimated support and $\trusupp$ is the underlying truth (see Section~\ref{sec:pre:setup} for details).

Most existing work considers standard design, i.e. $\Sigma=I_\dd$ \citep{wang2010informationl,rad2011nearly,fletcher2009necessary,reeves2008sampling,akccakaya2009shannon,aeron2010information,reeves2013approximate,aksoylar2016sparse}, and gives matching upper and lower bounds up to logarithmic factors in the sparsity $\sps$:
\begin{align*}
    \mathcal{O}\bigg(\frac{\log(\dd - \sps)}{\betam^2/\sigma^2}\vee \sps\log\frac{\dd}{\sps}\bigg) 
    \qquad\text{and}\qquad
    \Omega\bigg(\frac{\log(\dd - \sps)}{\betam^2/\sigma^2}\vee \frac{\sps\log\frac{\dd}{\sps}}{\log(1+\sps \betam^2/\sigma^2)}\bigg)\,.
\end{align*}
The upper bound is achieved by \bss{}, and matches the lower bound up to a factor that depends on the signal-to-noise ratio $\betam/\sigma$.
Moving toward general design, there are multi-stage methods based on estimation and thresholding \citep{fletcher2009necessary,meinshausen2009lasso,wasserman2009high,genovese2012comparison,ji2012ups,jin2014optimality,ndaoud2020optimal,wang2020bridge} which usually impose eigenvalue conditions on the design $\Sigma$ that can easily be violated in a graphical model (this is discussed in more detail throughout Section~\ref{sec:analysis}).
Support recovery for general designs is considered in \citep{wainwright2009information, shen2012likelihood,shen2013constrained,verzelen2012minimax}.
Notably, \bss{} is analyzed with general design and known sparsity in \citep{wainwright2009information, shen2012likelihood,shen2013constrained}, and lower bounds are also provided therein.
However, the upper and lower bounds do not match in general.
\citet{verzelen2012minimax} provides impossibility results for support recovery in ultra-high dimensions, but only shows results for fixed design, and the lower bound does not depend on $\Sigma$.

In graphical models, support recovery is mainly used for structure learning, i.e. estimating the underlying graph $G$. 
For undirected graphs, neighbourhood selection reduces to support recovery of the precision matrix, which is well-studied \citep{meinshausen2006high,wang2010information,misra2020information}.
For directed acyclic graphs (DAGs), neighbourhood selection is widely used for both linear \citep[e.g.][]{shojaie2010penalized,loh2014high,buhlmann2014cam} and nonlinear \citep[e.g.][]{margaritis1999bayesian,aliferis2010local,peters2014causal,buhlmann2014cam,azadkia2021fast} models. 
This is closely related to Markov boundary learning, for which many algorithms based on greedy search have been proposed 
\citep{tsamardinos2003algorithms,tsamardinos2003time,pena2007towards,aliferis2010local,gao2016efficient}.
More recently, a growing line of work concerns ordering based DAG learning methods \citep{peters2013identifiability,ghoshal2017learning,chen2019causal,gao2020polynomial,rajendran2021structure}, which first estimates the topological ordering of the underlying DAG, then performs support recovery for each node to identify the parents. %
This prior work mostly focuses on consistency and upper bounds. 
In terms of lower bounds towards optimality, \citet{ghoshal2017information} derive generic lower bounds for learning DAGs without establishing optimality. \citet{gao2022optimal} derive the optimal sample complexity in terms of $\sps$ and $\dd$, but once again impose strong eigenvalue conditions;
e.g. Example~\ref{exmp:motiv} in Section~\ref{sec:analysis:sem} does not satisfy the assumptions in \citet{gao2022optimal}.
By contrast, we explicitly focus on optimality with respect to $\Sigma$ while allowing for diverging eigenvalues.
In doing so, we allow for a much richer class of SEM.
We mention here also that the effect of path cancellation has been noted previously \citep{wasserman2009high,buhlmann2010variable,genovese2012comparison}.

Finally, it is worth recalling alternatives to \bss{} such as $\ell_1$-based methods like the Lasso \citep{tibshirani1996regression} and Dantzig selector \citep{candes2007dantzig}. To achieve exact support recovery, these methods require irrepresentability-type conditions \citep{zhao2006model,zhang2008sparsity,zhang2009some,wainwright2009sharp}. 
Another set of methods is based on Orthogonal Matching Pursuit (OMP) and require mutual incoherence \citep{tropp2007signal,cai2011orthogonal,zhang2011sparse,joseph2013variable}.
The irrepresentable condition can be replaced with incoherence as well by thresholding the Lasso estimate \citep{meinshausen2009lasso,wang2020bridge}.
Nevertheless, all of these conditions can be violated in graphical models with strong dependence. See \citet{van2009conditions} for an overview and Section~\ref{sec:analysis:sem} for details.
The nonconvex variants to relax the $\ell_1$-based methods are able to relax the irrepresentable condition \citep{fan2001variable,zhang2010nearly,loh2017support,feng2019sorted}, but optimal rates are missing.
Finally, \citep{hastie2020best,guo2022snr} study the effect of the signal-to-noise ratio on regression problems.

\subsection{Outline of the paper}\label{sec:intro:outline}

Necessary preliminaries and background are covered in Section~\ref{sec:pre}. 
We introduce \klbss{} in Section~\ref{sec:main} and provide an analysis of its sample complexity in Section~\ref{sec:analysis}. Practical aspects and computational considerations are discussed in Section~\ref{sec:prac} before a detailed empirical evaluation on both real and simulated data in Section~\ref{sec:expt}.
Appendix~\ref{sec:disc} contains additional discussion on interpreting our results and extending them to more general settings.
Technical proofs are deferred to Appendices~\ref{app:opt}-\ref{app:fano}.
Appendix~\ref{app:expt} provides additional details and results for the experiments.
For accessibility and generality, the main methodological construction of \klbss{} in Section~\ref{sec:main} and its generalizations in Section~\ref{sec:prac} can be read by readers without any knowledge of graphical models.

\subsection{Notation}\label{sec:intro:notation}
For any nonnegative integer $m$, let $[m]:=\{1,\ldots,m\}$. Throughout, $S$ and $T$ are subsets of $[\dd]$, write $S\triangle T = (S\setminus T)\cup (T\setminus S)$ to be the symmetric difference, and let $|S|$ be the cardinality of set $S$.
Denote set of all possible supports of dimension $\dd$, and sparsity $\sps$ to be $\suppspace:=\{S\subseteq[\dd]:|S|=\sps\}$, and bounded sparsity to be $\suppspaceub:= \cup_{\sps=0}^{\ubsps}\suppspace=\{S\subseteq[\dd]:|S|\le \ubsps\}$.
Let $\mathbb{S}^d_{++}$ be all positive definite matrices, $\mathbb{R}_{>0}$ be all positive real numbers.
The 2-norm of a vector $x$ is $\|x\|=(\sum_{j}x_j^2)^{1/2}$, the operator 2-norm of a matrix $A$ is $\|A\|=\|A\|_{\textup{op}}=\sup_{\|x\|=1}\|Ax\|$.
The largest and smallest eigenvalues of $A$ are $\lambda_{\max}(A)$ and $\lambda_{\min}(A)$.
Write $x_S$ to be the $|S|$-dimensional sub-vector indexed by $S$. Similarly, for a matrix $A$, write $A_{S}$ to be the sub-matrix with columns indexed by set $S$, and $A_{TS}$ to be the sub-matrix with rows and columns indexed by $T$ and $S$.
For a covariance matrix $\Sigma$, denote the conditional covariance matrix of the variables $S$ given the variables $T$ by 
\vspace{-1em}
\begin{align}
\label{eq:defn:condcov}
\Sigma_{S\given T}:= \Sigma_{(S\setminus T)(S\setminus T)} - \Sigma_{(S\setminus T)T}\Sigma_{TT}^{-1}\Sigma_{T(S\setminus T)}\,. 
\end{align}
\vspace{-2.6em}

\noindent
Here $\Sigma_{S\given T}$ is of size $|S\setminus T|\times |S\setminus T|$, and when $S\cap T=\emptyset$, $|S\setminus T|$ reduces to $|S|$.
For a set $\betaspace\subseteq\mathbb{R}^{\dd}$, write $\betaspace_S=\{\beta_S:\beta\in\betaspace\}$ for the subspace of coordinates indexed by $S$.
Let $\mathbf{1}_m,\mathbf{0}_m$ be all one's and all zero's vector of dimension $m$, and $\mathbbm{1}\{E\}$ be the indicator of event $E$.
Denote the support of a vector by $\supp(x) = \{j:x_j\ne 0\}$.
We say $a\lesssim b$ and $a\gtrsim b$ if $a\le Cb$ and $a\ge cb$ for some positive
constants $C$ and $c$, and $a\asymp b$ if both $a\lesssim b$ and $a\gtrsim b$.
$a\vee b$ and $a\wedge b$ are the maximum and minimum between two numbers $a$ and $b$.
For remainder of the paper, with a little abuse of notation we use $(X,Y)$ to denote the data matrix ($\mathbb{R}^{n\times \dd}\otimes \mathbb{R}^{n}$) and random variables interchangeably.
Write $\Pi_S:=X_S(X_S\T X_S)^{-1}X_S\T$ and $\Pi_S^\perp:=I_n-\Pi_S$ for projection matrices onto and out of the subspace spanned by $X_S$.

\section{Preliminaries}\label{sec:pre}
In this section, we provide necessary formal preliminaries. 
Since our main focus is neighbourhood selection in SEM, we begin by introducing graphical models and SEM, then connect linear models and neighbourhood selection in SEM. Finally, we formalize the support recovery problem in a general setting.
We note that in many places assumptions are made to streamline the presentation and discussion; additional extensions and relaxations of these assumptions
are discussed in Section~\ref{sec:prac} and Appendix~\ref{sec:disc}.

\subsection{Graphical models}\label{sec:pre:gm}
A graphical model is represented by a graph $G=(V,E)$ that reflects the dependencies in a random vector $Z=(Z_1,\ldots,Z_{\dd})$.
As usual, we abuse notation by identifying $V=Z$.
Given a DAG $G$ and a node $k\in V$, $\pa(k)=\{j:(j,k)\in E\}$ is the set of parents, and $\ch(k)=\{j:(k,j)\in E\}$ is the set of children. A directed path is a sequence of distinct nodes $({h_1},\ldots,{h_\ell})$ such that $(h_{j},h_{j+1})\in E$. Then the descendants $\de(k)$ are the nodes that can be reached from $k$ via some directed path, $\nd(k)=V\setminus\de(k)$ is the set of nondescendants, and the ancestors $\an(k)$ is the set of nodes that have directed path(s) to node $k$. 
A distribution $P$ over $Z$ is Markov to $G$ if $P$ factorizes according to $G$, i.e. $P(Z) = \prod_{k=1}^{\dd} P(Z_k\given \pa(k))\,.$
This implies that every $d$-separation relationship in $G$ reflects a genuine conditional independence relationship in $P$.
The detailed definition of $d$-separation---which will not be needed---can be found in any textbook on graphical models \citep[e.g.][]{lauritzen1996graphical,koller2009probabilistic}.
We do not assume faithfulness in this paper. We consider Gaussian linear SEM defined by:
\begin{align}\label{eq:pre:sem}
    Z_k = \sum_{j\in\pa(k)}b_{jk}Z_j + \epsilon_k,
    \quad
    \epsilon_k\sim\mathcal{N}(0,\sigma^2_k)\,,
    \quad \forall k\in[\dd].
\end{align}
See \citet{drton2018algebraic} for an introduction to linear SEM.
It follows from \eqref{eq:pre:sem} that $Z$ has a multivariate Gaussian distribution that is Markov to $G$; note the dependence on $G$ in the parent sets above. See Figure~\ref{fig:illusrate}.

In practice, the graph $G$ of the graphical model is often unknown, and one wishes to learn $G$ from i.i.d. observations of $Z$. 
A basic primitive in this process is neighbourhood selection: Learning the Markov boundary $S$ (often called the \emph{neighbourhood}) of each node $Z_k$ with respect to a set $A\subseteq V$ such that $Z_{k}\indep Z_{A\setminus S} \given Z_S$. This is known to reduce to recovering the support of regression model of $Z_k$ on $Z_A$ in SEM; see Appendix~\ref{app:pre:eqvmb}.
Therefore, to model the support recovery problem in SEM, we append one more node $Z_{\dd+1}$ to the graph by directing edges from a subset of $Z_1,\ldots,Z_{\dd}$ to it. We treat $Z_{\dd+1}$ as the target node and aim to learn the neighbourhood with respect to $Z_1,\ldots,Z_{\dd}$. See Figures~\ref{fig:demo} and~\ref{fig:illusrate} for an illustration.
To further align the notation, we will denote the target $Z_{\dd+1}$ by $Y$, and denote the set of candidate variables $Z_1,\ldots,Z_\dd$ by $X=(X_1,\ldots,X_\dd)$. Thus, the problem reduces to a prototypical regression problem between $Y$ and $X$.
Furthermore, we will let $G$ be the DAG over the variables $X$ (i.e. ignoring $Y$), since it is easy to obtain the full DAG by adding the node $Y$ and edges $X_{k}\to Y$ for $k\in\pa(Y)$.
This implies $X$ are the nondescendants of $Y$ in the full DAG, and thus the neighbourhood becomes the parents of $Y$.

\begin{figure}[t]
    \centering
    \begin{tikzpicture}[scale=0.8, transform shape]
        \node[state] (z1) {$Z_1$};
        \node[state] (z2) [ right = of z1, xshift=-.5cm] {$Z_2$};
        \node[state] (z3) [ right = of z2, xshift=-.5cm] {$Z_3$};
        \node[right = of z3, xshift=-.5cm] (dot)  {$\hdots$};
        \node[state] (zdm2) [ right = of dot, xshift=-.5cm] {$Z_{\dd-2}$};
        \node[state] (zdm1) [ right = of zdm2, xshift=-.5cm] {$Z_{\dd-1}$};
        \node[state] (zd) [ right = of zdm1, xshift=-.5cm] {$Z_{\dd}$};
        
        \path (z1) edge (z2);
        \path (z1) edge[bend left=30] (zdm2);
        \path (zdm2) edge (zdm1);
        \path (zdm2) edge[bend left=30] (zd);
        \path (z3) edge[bend left=30] (zdm1);
        \path (z2) edge[bend left=30] (zd);

        \node[state] (y) [below = of dot] {$Z_{\dd+1}$};
        \path (z2) edge (y);
        \path (zdm2) edge (y);
        \path (zdm1) edge (y);
        
        \node [draw=black,dashed,fit=(z1) (zd), inner sep=0.2cm] (G) {};
        \draw [decorate, decoration = {brace, mirror, amplitude=5pt}, -] ([xshift=-.7cm, yshift=.5cm] z1.north west) -- ([xshift=-.7cm, yshift=-.5cm] z1.south west) node[midway, xshift=-.8cm]{\Large $G$};
    \end{tikzpicture}
    \caption{A graphical model over $Z=(Z_1,\ldots,Z_\dd)$ with one more target node $Z_{\dd+1}$ appended to it. The corresponding $G$ will refer to a DAG over $Z$ (ignoring $Z_{\dd+1}$). The neighbourhood of $Z_{\dd+1}$ is $\{Z_2,Z_{\dd-2},Z_{\dd-1}\}$ in this example.}
    \label{fig:illusrate}
\end{figure}

\subsection{Problem setup}\label{sec:pre:setup}
The preceding discussion formalizes the well-known fact that neighbourhood selection in linear SEM reduces to support recovery in the Gaussian linear model \eqref{eq:lm}, which we restate here:
\vspace{-1em}
\begin{align*}
    Y = X\T \beta + \epsilon, \qquad X\sim\mathcal{N}(\mathbf{0}_\dd,\Sigma),\quad\epsilon\sim\mathcal{N}(0,\sigma^2),\quad X\indep \epsilon \,.
\end{align*}
\vspace{-2.6em}

\noindent
We assume $\|\beta\|_0=\sps\le \ubsps$, where $\sps$ is the sparsity level and $\ubsps$ is an upper bound. As commonly assumed in the support recovery literature, we will also assume that $\sps\le \ubsps \le \dd/2$ (cf. Remark~\ref{rem:ubsps}). 

Denote the support of $\beta$ by $\trusupp = \supp(\beta) \subseteq [\dd]$, i.e. $\supp(\beta)=\{j\in[\dd] : \beta_j\ne 0\}$. 
The model \eqref{eq:lm} defines a joint distribution $P(X,Y)$ that is determined by the tuple of parameters $(\beta,\Sigma,\sigma^2)$, i.e. $P=P_{\beta,\Sigma,\sigma^2}$.
We impose constraints on $\beta$ and---more importantly---$\Sigma$, through the 
constrained parameter spaces $\betaspace \subseteq \mathbb{R}^\dd$ and $\Sigmaspace\subseteq \mathbb{S}^\dd_{++}$.
These constrained spaces allow us in particular to impose graphical structure in the form of an SEM.

For regression coefficients $\beta$, we consider sparse vectors satisfying a beta-min condition:
\vspace{-.5em}
\begin{align}\label{eq:betamin}
    \betaspace= \betaspace_{\dd,\sps}(\betam):=\Big\{\beta\in\mathbb{R}^{\dd}: \|\beta\|_0 = \sps, \min_{j\in\supp(\beta)}|\beta_j|\ge \betam\Big\} \,.
\end{align}
\vspace{-2em}

\noindent
The beta-min condition is commonly assumed in the literature on support recovery \citep{zhao2006model,wainwright2009information}.
Neither $\sps$ nor $\betam$ are required to be known for our method to work (cf. Section~\ref{sec:prac:unknown}).
For the covariance matrix $\Sigma$, define
\vspace{-.5em}
\begin{align}\label{eq:sigmamin}
    \Sigmaspace = \Sigmaspace(\sigmam^2) 
    &= \Big\{
    \cov(X): X \text{ is generated by }\eqref{eq:pre:sem} \text{ for some DAG and }\sigma^2_k\ge\sigmam^2,\forall k\in[\dd]
    \Big\}\,.
\end{align}
\vspace{-2em}

\noindent
When $\sigmam^2\to0$, observe that $\Sigmaspace$ collapses to all of $\mathbb{S}^\dd_{++}$. 
The space $\Sigmaspace$ is also treated as unknown; this only arises in the analysis and is not directly used by our method (cf. Remark~\ref{rem:sigma:alg}).
We are interested in conditions on $\Sigma\in\Sigmaspace$ under which \bss{} can be improved.
Finally, define a parameter space by 
\vspace{-1em}
\begin{align}\label{eq:pre:genericmodel}
    \mclass
    = \mclass(\betaspace,\Sigmaspace,\sigma^2)
    := \Big\{(\beta,\Sigma,\sigma^2) : \beta\in\betaspace,\,\Sigma\in\Sigmaspace\Big\}\,.
\end{align}
\vspace{-3em}

\noindent
Since there is a one-to-one correspondence between parameter tuples $(\beta,\Sigma,\sigma^2)\in \mclass$ and joint distributions $P_{\beta,\Sigma,\sigma^2}$, 
we will frequently abuse notation by referring to $\mclass$ as the model, bearing in mind this one-to-one correspondence. Since the model is identified, this should cause no confusion.

Our goal is to design an estimator of the support $\trusupp$, which is a measurable function $\estsupp$ of the data $(X,Y)$ into the power set $2^{[\dd]}$, i.e. $\estsupp(X,Y)\subseteq[\dd]$.
We study the pointwise and uniform sample complexity for exact support recovery in terms of $\dd,\sps$, and other model parameters.
Specifically, to compare different methods, we wish to determine the sample size $n$---in terms of the parameter tuple $(\beta,\Sigma,\sigma^2)$---such that
\vspace{-3em}
\begin{align}\label{eq:guarantee}
    \prob_{\beta,\Sigma,\sigma^2}(\estsupp\ne \trusupp)\le \delta,
    \qquad\delta>0 \,.
\end{align}
\vspace{-3em}

\noindent
When $n=n(\beta,\Sigma,\sigma^2)$, this corresponds to the \emph{pointwise} sample complexity of the estimator $\estsupp$.
To further characterize the minimax performance,
we study the sufficient and necessary conditions on $n = n(\mclass)$ such that the guarantee \eqref{eq:guarantee} holds \emph{uniformly} for all $(\beta,\Sigma,\sigma^2)\in\mclass$.
When the conditions match up to problem-independent constants, we call the resulting sample size the optimal sample complexity.
We do not suppress logarithmic factors of the dimension $\dd$. An estimator $\estsupp$ is called optimal if it achieves \eqref{eq:guarantee} uniformly over $\mclass$ with the optimal sample complexity.

\begin{remark}
    As is standard, we assume the noise variance $\sigma^{2}$ is fixed for simplicity. 
    Thus, although we could omit the variance parameter, we choose to emphasize the dependence of our results on $\sigma^2$.
    If instead $\sigma^2\in(0,\sigma_{\max}^2]$ for some $\sigma^2_{\max}$, then our results continue to hold by replacing $\sigma^2$ with $\sigma^2_{\max}$.
\end{remark}

\begin{remark}
\label{rem:ubsps}
    The upper bound $\sps\le \dd/2$ is a technical assumption in the analysis to simplify the presentation. For the case where $\sps>\dd/2$, the roles of $(\dd-\sps)$ and $\sps$ in the sample complexity are switched. For example, a complete result of Theorem~\ref{thm:opt:eigen} (similarly for Theorem~\ref{thm:pointwise:eigen}) without the assumption of $\sps\le \dd/2$ would be
    \vspace{-1em}
    \begin{align*}
        \frac{\log(\dd-\sps)\vee \log\sps}{\betam^2\sigmam^2/\sigma^2} + \log\binom{(\dd-\sps)\vee \sps}{(\dd-\sps)\wedge \sps}\,.
    \end{align*}
\end{remark}

\section{\klbss{}: Support recovery in SEM}
\label{sec:main}

We now introduce our estimator of the support $\trusupp$, called $\klbss{}$. 
Throughout this section we assume that $\betaspace=\betaspace_{\dd,\sps}(\betam)$ is given, i.e. the knowledge of sparsity level $\sps$ and the signal strength $\betam$, thus the candidate supports are $\suppspace$, which we recall denotes all possible supports of size $\sps$. Extensions to unknown sparsity and beta-min are discussed in Section~\ref{sec:prac}.
Since \klbss{} adopts the tournament interpretation of \bss{} as searching over all possible supports and conducting pairwise comparisons, we start by introducing the building block of our estimator: Comparing two candidate supports (Algorithm~\ref{alg:main:twoccase2}).

\subsection{Comparing two candidates}\label{sec:main:twocase}

For any two candidate supports $S,T\in\suppspace$, instead of directly comparing residual variances as in \bss{}, Algorithm~\ref{alg:main:twoccase2} chooses the ``better'' candidate 
using a newly defined score $\mathcal{L}$ with an additional term that arises from the KL-divergence between each candidate model, explained below. In order to avoid notational clutter, let the shared component be $W:=S\cap T$, and the difference between two candidate supports be $S':=S\setminus T$ and $T':=T\setminus S$, so that 
we can write both $S$ and $T$ as $R\cup W$ with $R\in\{S',T'\}$. Then the score for $R\cup W$ is given by
\begin{align}
\label{eq:compare:score}
\mathcal{L}(R\cup W;(S,T)) 
:= \underbrace{\frac{\|\Pi_{R\cup W}^\perp Y\|^2}{n-\sps}}_{\text{residual variance from \bss{}}} + \underbrace{\min_{{\gamma}\in\betaspace_R}(\widehat{\gamma} - {\gamma})\T \frac{\widetilde{X}_R\T \widetilde{X}_R}{n-(\sps-r)} (\widehat{\gamma} - {\gamma})}_{\text{violation of constraint $\betaspace_R$}} \,,
\end{align}
where 
$\widetilde{X}_R = \Pi^\perp_W X_R$ partials out the effect from the shared component $X_W$ on $X_R$, and $\widehat{\gamma}$ collects the OLS regression coefficients of $\widetilde{Y}=\Pi^\perp_W Y$ on $\widetilde{X}_R$, recall that $\betaspace_R=\{\beta_R:\beta\in\betaspace\}$ and $\Pi_W^\perp=I_n-X_W(X_W\T X_W)^{-1}X_W\T$ is the projection matrix.

The first term in $\mathcal{L}$ is exactly the residual variance used in \bss{} since $\|\Pi_S^\perp Y\|^2=\|Y - X_S\T \widehat{\beta}(S)\|^2$ in~\eqref{eq:bss}.
The second term quantifies the extent to which $\widehat{\gamma}$ ``violates'' the constraint $\betaspace_R$: That is, the partial regression coefficients $\widehat{\gamma}$ need not be in $\betaspace_R$ when they are close to zero, and the second term measures how far away $\widehat{\gamma}$ is from $\betaspace_{R}$. This term can be interpreted as the (weighted) $L^{2}$-projection of $\widehat{\gamma}$ onto $\betaspace_{R}$. 
When either $S$ or $T$ is the true support, the second term is zero in expectation, while that of the incorrect support can be positive. So this additional term helps to detect when a candidate set has its partial regression coefficients close to zero. 
In other words, the second term captures available signal that \bss{} ignores:
The two terms in \eqref{eq:compare:score} can be interpreted as the minimum KL divergence between any two models specified by the two candidate supports $S,T$ given their intersection $W$, which motivates the design of \klbss{}. See Appendix~\ref{sec:disc:interpret} for details.
We refer to Algorithm~\ref{alg:main:twoccase2} as the \compare{} algorithm.

\subsection{The proposed estimator}\label{sec:main:main}

\begin{algorithm}[t]
    \caption{\compare{} algorithm}\label{alg:main:twoccase2}
    \raggedright \textbf{Input:} Data matrix $X$; response $Y$; candidate supports $S,T\in\suppspace$; coefficient space $\betaspace$.\\
    \textbf{Output:} Estimated support $\estsupp$.
    \setlist{nolistsep}
    \begin{enumerate}
        \item Let $S':=S\setminus T, T':=T\setminus S, W:=S\cap T$, $r:=|S'|=|T'|$ 
        \item Compute $\widetilde{X}_{S'}=\Pi_W^\perp X_{S'},\widetilde{X}_{T'}=\Pi_W^\perp X_{T'},\widetilde{Y}=\Pi_W^\perp Y$
        \item For $R\in\{S',T'\}$:
        \begin{enumerate}
            \item $\widehat{\gamma} = (\widetilde{X}_R\T \widetilde{X}_R)^{-1}\widetilde{X}_R\T \widetilde{Y}$;
            \item $\mathcal{L}(R\cup W;(S,T)) = \frac{\|\Pi_{R\cup W}^\perp Y\|^2}{n-\sps} + \min_{{\gamma}\in\betaspace_R}(\widehat{\gamma} - {\gamma})\T \frac{\widetilde{X}_R\T \widetilde{X}_R}{n-(\sps-r)} (\widehat{\gamma} - {\gamma})$;
        \end{enumerate}
        \item Output $\estsupp = \argmin_{D\in\{S,T\}}\mathcal{L}(D;(S,T))$.
    \end{enumerate}
\end{algorithm}

\begin{algorithm}[t]
    \caption{\klbss{}}\label{alg:main:simple}
    \raggedright 
    \textbf{Input:} Data matrix $X$; response $Y$; coefficient space $\betaspace$; sparsity $\sps$. \\
    \textbf{Output:} Estimated support $\estsupp$.
    \setlist{nolistsep}
    \begin{enumerate}
        \item Let $M=|\suppspace|$, randomly order the elements in $\suppspace$ to be $S_1,S_2,\ldots,S_M$;
        \item Initialize $\estsupp=S_1$;
        \item For $j=2,3,\ldots,M$:
        \begin{enumerate}
            \item $\estsupp = \compare{}(X,Y,\estsupp,S_j,\betaspace)$;
        \end{enumerate}
        \item Output $\estsupp$.
    \end{enumerate}
\end{algorithm}

Using \compare{} as our workhorse, we now introduce our proposed estimator, which we call \klbss{} since the estimator can be interpreted via the KL divergence decomposition discussed in Appendix~\ref{sec:disc:interpret}.
Conceptually, we can line up all the candidates according to some order, then start with comparing the first two candidates using the prescribed score, and proceed with the winner to compete with the third, etc. After running through each pairwise comparison in this order, a winner is declared.
In \compare{}, the pairwise comparison between $S$ and $T$ depends on the shared component $W$, so the relationship between scores $\mathcal{L}$ defined in~\eqref{eq:compare:score} is not necessarily transitive.
Therefore, we adopt this conceptual tournament idea (realized in Algorithm~\ref{alg:main:simple}) to find the final winner using pairwise comparison along some order.

It is worth mentioning that the program in the second part of \eqref{eq:compare:score} is nonconvex when $\betaspace$ is nonconvex. Of course, since \bss{} is itself solving a nonconvex combinatorial optimization problem, this is to be expected. Moreover, this is out of necessity: Under standard complexity assumptions, polynomial-time algorithms achieving the optimal rate under general dependence (i.e. our setting) cannot exist in a precise sense (cf. Remark~\ref{rem:compstat}).
If the space $\betaspace_R$ is formed by $r=|R|$ many ``bounded-away-from-zero'' constraints (i.e. $|\beta_j|\ge\betam,\forall j\in R$, $r\in[\sps]$), then this program can be solved via $2^r$ quadratic programs with box constraints, of which each one can be solved very fast.
Moreover, this program can be cast as a standard mixed integer program (Section~\ref{sec:prac:mip}). Later in Sections~\ref{sec:prac}-\ref{sec:expt}, this procedure will be implemented and explored on finite samples.

\begin{remark}\label{rem:sigma:alg}
    Neither Algorithm~\ref{alg:main:twoccase2} nor Algorithm~\ref{alg:main:simple} uses $\Sigmaspace$ as an input. As a consequence, any structural assumptions on $\Sigma$ (e.g. SEM assumptions) are not explicitly enforced by the algorithm. In this way, \klbss{} \emph{implicitly} exploits unknown structure without explicitly imposing it. The dependence on $\Sigmaspace$ only arises in the analysis, where the sample complexity will depend on $\Sigma$ and/or $\Sigmaspace$.
\end{remark}

\subsection{Comparison with \bss{}}\label{sec:main:bss}

In addition to generalizing \bss{}, \klbss{} has the important property that on average, it performs at least as well as \bss{}; this will be a consequence of Theorems~\ref{thm:pointwise:eigen}-\ref{thm:opt:eigen} in the next section. 
Thus, even if the model is not necessarily an SEM, \klbss{} still enjoys all of the storied optimality properties of \bss{} for general design matrices.

Specifically, the difference between \bss{} and \klbss{} lies in the second term in \eqref{eq:compare:score}, particularly the way it invokes the constraint $\betaspace$ in its minimization. 
In fact, when $\betam=0$---i.e. $\betaspace=\mathbb{R}^{\dd}$---\klbss{} reduces to \bss{}. This continues to be true as long as $\betam\approx0$, in which case the beta-min condition will never be violated. As $\betam$ increases, the second term measures the extent to which partial regression coefficients in the model violate the constraint $\betaspace$, and whenever this term is positive, \klbss{} will improve upon \bss{}. 
Thus, there is a precise sense in which \klbss{} generalizes \bss{}.
This will be made formal in the next sections (Sections~\ref{sec:analysis}-\ref{sec:prac}) and empirically demonstrated in Section~\ref{sec:expt}.

\section{Analysis of \klbss{}}\label{sec:analysis}

In this section, we analyze \klbss{}, with a particular focus on comparing its statistical properties with those of \bss{}. The key takeaways are: 1) It performs at least as well as \bss{} on average, and often strictly better when the model is an SEM, and 2) It is optimal over a larger family of design matrices.

\subsection{Eigenvalue conditions}\label{sec:analysis:eigen}

The support recovery literature expresses the difficulty of recovery in terms of eigenvalue-type conditions
that capture the signal for recovering the true support. Following in this tradition, we define a corresponding eigenvalue quantity for \klbss{} and show that it captures the pointwise and minimax sample complexity of \klbss{}.
Through this subsection, let $\Sigma$ be an arbitrary positive-definite matrix.

We first recall the following quantity from \citet{wainwright2009information}, which defines an appropriate eigenvalue for \bss{}:
\begin{align}\label{eq:eigval:bss}
    \eigvalb(\Sigma) &:= \min_{T\in\suppspace\setminus\{\trusupp\}} \bigg[\min_{v\in\mathbb{R}^r} \frac{v\T \Sigma_{\trusupp \given T}v}{\|v\|^2} \bigg] = \min_{T\in\suppspace\setminus\{\trusupp\}} \lambda_{\min}(\Sigma_{\trusupp\given T}) \,,
\end{align}
where $r:=|\trusupp\setminus T|$ is the size of $\Sigma_{\trusupp\given T}$ (cf.~\eqref{eq:defn:condcov}). This quantity characterizes the information carried by covariates in the true support that is unexplained by alternatives, which is connected to the minimum eigenvalue of $\Sigma$ and also appears in \citet{shen2012likelihood,shen2013constrained}.
The idea is that larger $\eigvalb(\Sigma)$ means a stronger signal for support recovery and thus a smaller sample complexity. 
\begin{definition}\label{defn:eigval:klbss}
Given a design matrix $\Sigma$ and true support $\trusupp$, the \klbss{} eigenvalue is defined as
\begin{align}\label{eq:eigval:klbss}
    \eigvalk (\Sigma) & := \min_{T\in\suppspace\setminus\{\trusupp\}} \bigg[ \min_{u\in\betaspace_{\trusupp\triangle T}} \frac{u \T \Sigma_{\trusupp\cup T\given \trusupp\cap T} u}{\min_{|R|=r} \|u_R\|^2}  \bigg] \qquad \text{where } r:=|\trusupp\setminus T| \,.
\end{align}
\end{definition}
\noindent
Compared to \eqref{eq:eigval:bss}, the size of $\Sigma_{\trusupp\cup T\given \trusupp\cap T}$ is $2r = |\trusupp\triangle T|$. 
The difference between \eqref{eq:eigval:bss} and \eqref{eq:eigval:klbss} lies in the conditioning set and size for $\Sigma$, the additional restriction in the denominator, and the constraints on $u$. 
Roughly speaking, \eqref{eq:eigval:klbss} is minimizing a larger quantity over a more constrained family, which yields a larger signal, and thus a lower sample complexity.

To make this precise, define a class of design matrices by 
\begin{align}
\label{eq:opt:Omegagap:real} 
    \Sigmaspace_\Delta & = \Big\{\Sigma\in\Sigmaspace: \min_{T\in\suppspace\setminus\{\trusupp\}}\lambda_{\min}(\Sigma_{\trusupp\given T}) < \min_{T\in\suppspace\setminus\{\trusupp\}}\lambda_{\min}(\Sigma_{T\given \trusupp}) \Big\}.
\end{align}
This class will be used frequently in the sequel: It corresponds to models where \klbss{} strictly outperforms \bss{}.
This is captured by the following lemma, which formalizes this idea that \klbss{} exploits more signal than \bss{}:
\begin{lemma}\label{lem:eigen:comp}
For any $\Sigma\succ0$, $\eigvalk(\Sigma)\ge\eigvalb(\Sigma)$ and if $\Sigma\in\Sigmaspace_\Delta$ then $\eigvalk(\Sigma)>\eigvalb(\Sigma)$.
\end{lemma}
\noindent
See Appendix~\ref{app:analysis:eigen:comp} for a proof. 
The construction of $\Sigmaspace_\Delta$ is based on the asymmetry alluded to in Section~\ref{sec:intro:overview}, and will become more clear through the examples in Section~\ref{sec:analysis:sem}.
For brevity, we simply refer to $\eigvalk$ and $\eigvalb$ as eigenvalues in the sequel, more specifically, the \bss{} eigenvalue and the \klbss{} eigenvalue.

\subsection{Strict improvement over \bss{}}\label{sec:analysis:strict}

Roughly speaking, as long as these eigenvalues are nonzero, support recovery is possible, and the larger the eigenvalue, the easier recovery will be. To characterize the performance gap and optimality between \klbss{} and \bss{}, we consider the SEM class $\Sigmaspace$ introduced in (\ref{eq:sigmamin}-\ref{eq:pre:genericmodel}).

We begin with a comparison of the pointwise sample complexity (cf. \eqref{eq:guarantee}) between \klbss{} and \bss{}: 
\begin{theorem}
\label{thm:pointwise:eigen}
    Assume $\sps\le \dd/2$ with $\beta\in\betaspace$. For any $\Sigma\in\Sigmaspace$, the (pointwise) sample complexities for support recovery of \klbss{} and \bss{} are
    \begin{align*}
        \underbrace{\frac{\log(\dd-\sps)}{\eigvalk(\Sigma)\betam^2 / \sigma^2} \vee \log \binom{\dd-\sps}{\sps}}_{\textup{\klbss{}}}
      \quad \AND \quad 
    \underbrace{\frac{\log(\dd-\sps)}{\eigvalb(\Sigma)\betam^2 / \sigma^2} \vee \log \binom{\dd-\sps}{\sps}}_{\textup{\bss{}}} \,.
    \end{align*}
    In particular, \klbss{} is more efficient than \bss{} as long as $\eigvalk(\Sigma)>\eigvalb(\Sigma)$. 
\end{theorem}
\noindent
By Lemma~\ref{lem:eigen:comp}, we have that \emph{strict} improvement holds for any $\Sigma\in \Sigmaspace_\Delta$.
The implications of Theorem~\ref{thm:pointwise:eigen} are twofold: 1) \klbss{} is always \emph{at least} as sample efficient as \bss{}, and 2) On $\Sigmaspace_\Delta$, \klbss{} improves \bss{} and the improvement is strict. A more technical version of this result also holds for any fixed $(\beta,\Sigma,\sigma^2)$; see Remark~\ref{rem:pw:thm} in Appendix~\ref{sec:disc}.

Next, we move on to characterize the minimax optimality of \klbss{} in SEM through the parameter space $\Sigmaspace$ in \eqref{eq:sigmamin}. Define for any constant $c_0>0$ the following two classes of SEM design matrices: 
\begin{align}\label{eq:opt:Omega:real}
    \Sigmaspacek = \Big\{\Sigma\in\Sigmaspace: \eigvalk(\Sigma) \ge  c_0 \sigmam^2\Big\},
    \qquad
    \Sigmaspaceb = \Big\{\Sigma\in\Sigmaspace: \eigvalb(\Sigma) \ge c_0 \sigmam^2\Big\}.
\end{align}
Recall that $\sigmam^2$ is the minimum noise variance in the SEM (cf. \eqref{eq:sigmamin}). By Lemma~\ref{lem:eigen:comp}, we have $\Sigmaspaceb\subseteq\Sigmaspacek$. 
Then the following theorem gives the desired minimax characterization in SEM. 
\begin{theorem}\label{thm:opt:eigen}
    Assume $\sps\le \dd/2$ with $\beta\in\betaspace$. Then the minimax optimal sample complexity over both $\Sigmaspaceb$ and $\Sigmaspacek$ is 
    \begin{align*}
         \frac{\log(\dd-\sps)}{\sigmam^2\betam^2 / \sigma^2} \vee \log \binom{\dd-\sps}{\sps}\,.
    \end{align*}
    Moreover, \klbss{} achieves the optimal sample complexity over both $\Sigmaspaceb$ and $\Sigmaspacek$.
\end{theorem}
\noindent
For comparison with \bss{}, it is known that \bss{} achieves the optimal sample complexity over $\Sigmaspaceb$ \citep{wainwright2009information}, while \klbss{} extends the optimality to $\Sigmaspacek\supseteq\Sigmaspaceb$.
These results underscore the critical roles played by the eigenvalues $\eigvalk$ and $\eigvalb$. The proofs are in Appendix~\ref{app:analysis:pointwise:eigen}-\ref{app:analysis:opt:eigen}.
\begin{remark}\label{rmk:general}
    The optimality results in Theorem~\ref{thm:opt:eigen} for $\Sigmaspacek$ and $\Sigmaspaceb$ can be extended beyond SEM, e.g. $\Sigmaspacek'=\{\Sigma\in\mathbb{S}^{\dd}_{++}:\eigvalk(\Sigma) \ge \omega\}$ and $\Sigmaspaceb'=\{\Sigma\in\mathbb{S}^{\dd}_{++}:\eigvalb(\Sigma) \ge \omega\}$ with $\sigmam^2$ replaced by $\omega$. See the proof of Theorem~\ref{thm:opt:lb1} and Remark~\ref{rem:eigval:fundamental} in Appendix~\ref{app:analysis:opt:eigen} for details.
\end{remark}

Theorems~\ref{thm:pointwise:eigen} and \ref{thm:opt:eigen} together imply that not only is \klbss{} minimax optimal over a larger family of designs, but also that there is a class of designs---those in $\Sigmaspace_\Delta$---over which \klbss{} strictly outperforms \bss{}.
The following example sheds some light on just how large this class is:
\begin{example}[Models where \klbss{} outperforms \bss{}]
\label{ex:ABC}
Consider any SEM as in~\eqref{eq:sigmamin} with design matrix $\Sigma$. Given the target variable $Y$ with true support $\trusupp$, let $C=\Sigma_{\trusupp}$, $D=\Sigma_{\trusupp^c\given\trusupp}$, so we can write (up to a permutation of the rows and columns)
\begin{align}\label{eq:opt:Omegagap}
    \Sigma 
    = \Sigma(A,C,D)
    = \begin{pmatrix}
        C & A\T \\
        A & D + AC^{-1} A\T
    \end{pmatrix}
\end{align}
for some matrix $A\in\mathbb{R}^{(\dd-\sps)\times \sps}\ne 0$.
Then we have $\Sigma\in\Sigmaspace_{\Delta}$ as long as 
\begin{align}
\label{eq:sem:strong}
    \lambda_{\min}(D) \ge \lambda_{\min}(C)
    \iff
    \lambda_{\min}(\Sigma_{\trusupp^c\given\trusupp}) \ge \lambda_{\min}(\Sigma_{\trusupp}).
\end{align}
In particular, by Lemma~\ref{lem:eigen:comp} and Theorem~\ref{thm:pointwise:eigen}, this gives explicit examples where \klbss{} is strictly more sample efficient than \bss{}.
Since $A$ and $C$ here are essentially arbitrary, the only constraint appears on $D$ through the eigenvalue constraint \eqref{eq:sem:strong}. 
While this captures a wide range of models, of course this may not always hold. Fortunately, \eqref{eq:sem:strong} is merely a sufficient condition, and the weaker condition in \eqref{eq:opt:Omegagap:real} substantially relaxes this sufficient condition.
See Appendix~\ref{app:general:ex:ABC}.
\end{example} 
\noindent
In the next section, we will see how the properties of SEM make these conditions easy to satisfy. This will also help motivate the full relaxation in \eqref{eq:opt:Omegagap:real}, which will prove useful for general SEM.

\subsection{Comparison in SEM}\label{sec:analysis:sem}

Example~\ref{ex:ABC} provides a general class of designs where \klbss{} outperforms \bss{}. We now consider how this example manifests in SEM, and explain why both \bss{} and the Lasso are likely to fail in SEM.
The examples in this section are intended to illustrate how and why there is good reason to expect \eqref{eq:opt:Omegagap:real}
to hold in SEM. 
Throughout this section, it is useful to bear in mind that \emph{smaller} eigenvalues indicate \emph{more} dependence; thus maximizing dependence corresponds to minimizing eigenvalues.

The first example helps illustrate why minimizing over $T$ in $\Sigmaspace_\Delta$ is useful, and provides some intuition behind why SEM typically fall into this gap. Recall that a $v$-structure is any triplet of nodes converging at one node, i.e. $X_{k}\to X_{j}\leftarrow X_{\ell}$, where $X_j$ is called a \emph{collider}: This is precisely the kind of structure that cannot be embedded within undirected graphical models where \bss{} is known to be optimal.

\begin{example}[General SEM and collider bias]
\label{eq:sem:fail}
    Consider a general SEM $G$ over $X$ with $\trusupp$. 
    Condition~\eqref{eq:opt:Omegagap:real} asks us to find a subset $T\subseteq\trusupp^c$ that \emph{maximizes} the dependence (i.e. minimizing $\lambda_{\min}(\Sigma_{\trusupp\given T})$) between the parents of $Y$ after conditioning on $T$. There are two cases:
    
    \begin{enumerate}
        \item \emph{Dependent parents.} If the parents are already (marginally) dependent---which would be typical---then $\lambda_{\min}(\Sigma_{\trusupp})$ will be small, making it easier to satisfy \eqref{eq:opt:Omegagap:real} since $\lambda_{\min}(\Sigma_{\trusupp\given T}) \le \lambda_{\min}(\Sigma_{\trusupp})$ for all $T$.
        In fact, this bound is how we arrive at the simple case $\lambda_{\min}(D) \ge \lambda_{\min}(C)$ in Example~\ref{ex:ABC}. 
        In this case, minimizing over $T$ would not be necessary.
        \item \emph{Independent parents.} If the parents are independent (or weakly dependent)---which is the exceptional case---then $\lambda_{\min}(\Sigma_{\trusupp})$ is not expected to be small due to independence, and we can no longer rely on this to control $\lambda_{\min}(\Sigma_{\trusupp\given T})$ in \eqref{eq:opt:Omegagap:real}. But, as long as some descendants of $\trusupp$ participate in a $v$-structure, then conditioning on these descendants will induce dependence between the parents. This is the well-known \emph{collider bias} phenomenon in SEM, closely related to selection bias in observational studies \citep{greenland1999causal,hernan2004structural,elwert2014endogenous}. The conditional dependence induced by collider bias has the effect of shrinking the right side of \eqref{eq:opt:Omegagap:real}, making this condition likely to satisfy in SEM.
        In this case, minimizing over $T$ is helpful.
    \end{enumerate}
    Thus, in either case, we see that the structure of SEM encourages the inequality in \eqref{eq:opt:Omegagap:real} to hold.
    
\end{example}

The next example provides a concrete SEM construction where \bss{} fails.
More specifically, under what circumstances will an SEM satisfy 
$\eigvalb(\Sigma)\ge c_0\sigmam^2$ in \eqref{eq:opt:Omega:real}?
In fact, this condition cannot be satisfied by any design whose minimum eigenvalue shrinks, which is quite common in SEMs with growing degree.

\begin{example}[Failure of \bss{} in SEM with growing degree]\label{exmp:gap}
Let $G$ be any SEM, and suppose that some parent of $Y$ is a source node in $G$, say $k^*\in\trusupp$. This local structure is depicted below, where node $k^*$ has $\sps$ children, denoted here by $T$, and the remaining ancestors of $T$ are denoted as $\mathbf{X}_e$. 
\vspace{.3cm}

\begin{minipage}[c]{0.5\linewidth}
    \raggedleft
    \begin{tikzpicture}[scale=1, transform shape]
        \node[state] (xk) {$X_{k^*}$};
        \node [state, below = of xk] (xt2) {$X_{T_2}$};
        \node [state, left = of xt2, xshift=0.8cm] (xt1) {$X_{T_1}$};
        \node (dot) [right = of xt2, xshift=-1cm] {$\hdots$};
        \node [state, right = of dot, xshift=-0.8cm] (xts) {$X_{T_\sps}$};
        \node [state, above = of xts] (xe) {$\mathbf{X}_{e}$};

        \path (xk) edge (xt1);
        \path (xk) edge (xt2);
        \path (xk) edge (xts);

        \node [draw=black,dashed,fit=(xt1) (xts), inner sep=0.2cm, ultra thick] (TT) {};
        \path (xe) edge ([xshift=-.7cm] TT.north east);
        
        \draw [decorate, decoration = {brace, amplitude=5pt}, -] ([xshift=.2cm] TT.north east) -- ([xshift=.2cm] TT.south east) node[midway, xshift=.5cm]{\Large $T$};   
    \end{tikzpicture} 
\end{minipage}
\begin{minipage}[c]{0.5\linewidth}
\begin{align*}
    & k^*\in \trusupp,\,\, T=\ch(k^*),\,\, |T|=\sps, \\[0.8em]
    & X_T = b X_{k^*} + B \mathbf{X}_e + \epsilon_T, \\[0.8em]
    & \mathbf{X}_e  = \an(T) \setminus\{k^*\} \,.
\end{align*}
\end{minipage}

\vspace{.5cm}
\noindent
The $X_T$ are generated by the equations on the right where $B\in\mathbb{R}^{s\times |\mathbf{X}_e|}$ and $b\in\mathbb{R}^\sps$. 
Aside from the constraint that $k^*\in\trusupp$, $G$ and its SEM coefficients are allowed to be otherwise arbitrary. Therefore, $k^*$ is a hub node with growing degree, and there could be other hub nodes in $\mathbf{X}_e$.

This simple structure highlights a case where \bss{} fails but \klbss{} can succeed. 
In order for \bss{} to succeed, \eqref{eq:eigval:bss} must remain bounded away from zero.
But this local structure guarantees that $\eigvalb(\Sigma)$
will vanish as the sparsity level grows since
\begin{align*}
    \eigvalb(\Sigma)\overset{(i)}{\le} \lambda_{\min}(\Sigma_{\trusupp\given T}) \le  \var(X_{k^*} \given T) \to 0    \quad\text{as $s\to\infty$} \,.
\end{align*}
See Lemma~\ref{lem:opt:eigen3} for a formal statement.
Thus, the eigenvalue for \bss{} will shrink, making it 
impossible to satisfy $\eigvalb(\Sigma)\ge c_0\sigmam^2$. As a consequence, any SEM with such local structure cannot live in $\Sigmaspaceb$.

On the other hand, the vanishing of $\lambda_{\min}(\Sigma_{\trusupp\given T})$ is not a problem for \klbss{}---recall Example~\ref{eq:sem:fail}. Indeed,
this local structure does not directly affect $\eigvalk(\Sigma)$, which can still satisfy $\eigvalk(\Sigma)\ge c_0\sigmam^2$ even when $\eigvalb(\Sigma)$ fails to.
That is, the crucial upper bound (i) \emph{does not hold} for $\eigvalk(\Sigma)$, and this is how \klbss{} is able to exploit the signal that \bss{} ignores. As a result, although such SEM cannot live in $\Sigmaspaceb$, they can still be found in $\Sigmaspacek$.
To investigate the effect of degree and multiple hub nodes, we conduct experiments on SEMs generated by complete graphs (DAGs that are fully connected) in Section~\ref{sec:expt:simu} and the simple structure in this example in Appendix~\ref{app:expt:growdeg} to empirically verify this performance gap between \klbss{} and \bss{} in growing degree SEMs.

\end{example}

\begin{example}[Failure of Lasso]
\label{ex:lasso:fail}
It is also easy to construct explicit examples where the Lasso fails. We use Example~\ref{ex:ABC} to construct examples where \klbss{} outperforms \bss{} and for which the Lasso is also guaranteed to fail.
    For example, take $C = I_\sps$, $D=bI_{\dd-\sps}$ with $b>1$ in Example~\ref{ex:ABC}. Then \eqref{eq:sem:strong} is automatically satisfied. Let $A$ be any matrix whose entries are strictly bounded and construct $(\beta,\Sigma)$ as follows: Let $\Sigma=\Sigma(A,C,D)$ and $\beta$ be any vector with $\text{sgn}(\beta_{\trusupp})=\text{sgn}(a_j)$ for some $j$, where $a_j:= A\T_j \in \mathbb{R}^\sps$ is the $j$-th column of $A$.
    Writing $\widetilde{\Sigma}_{ij} = \Sigma_{ij} / \sqrt{\Sigma_{ii}\Sigma_{jj}}$, the incoherence parameter satisfies
    \begin{align*}
        \max_{k\in S^c_{*}}|\widetilde{\Sigma}_{k\trusupp}\widetilde{\Sigma}_{\trusupp\trusupp}^{-1}\text{sgn}(\beta_{\trusupp})| 
        \ge \frac{\|a_j\|_1}{\sqrt{b+\|a_j\|_2^2}} 
        \asymp \sqrt{\sps} > 1\,.
    \end{align*}
    \noindent
    Then the irrepresentability condition is (badly) violated and the Lasso is inconsistent. 
\end{example}

The next example helps illustrate why the analysis of \klbss{}---or, for that matter, any method for variable selection in an SEM, including \bss{}---is difficult. Namely, the well-known phenomenon of path cancellation in SEM. For an introduction to and illustration of path cancellation, we refer the reader to Section~2 of \citet{uhler2013geometry}.

\begin{figure}[t]
    \centering
    \begin{minipage}[b]{0.5\linewidth}
    \centering
    \begin{tikzpicture}[scale=0.7, transform shape]
        \node[state] (x1) {$X_1$};
        \node (x2) [state, right = of x1, xshift=-0.6cm] {$X_2$};
        \node (dot1) [right = of x2, xshift=-0.8cm] {$\cdots$};
        \node (xsm1) [state, right = of dot1, xshift=-0.8cm] {$X_{\sps-1}$};
        \node (xsp1) [state, fill=gray, fill opacity=0.5, text opacity=1, above = of x1, yshift=-0.5cm] {$X_{\sps+1}$};
        \node (xsp2) [state, fill=gray, fill opacity=0.5, text opacity=1, above = of x2, yshift=-0.5cm] {$X_{\sps+2}$};
        \node (dot2) [above = of dot1, yshift=0.1cm] {$\cdots$};
        \node (x2sm1) [state, fill=gray, fill opacity=0.5, text opacity=1, above = of xsm1, yshift=-0.5cm] {$X_{2\sps-1}$};

        \node (xs) [state, below = of x1, xshift=-0.6cm, yshift=0.8cm] {$X_\sps$};
        \node (x2s) [state, fill=gray, fill opacity=0.5, text opacity=1, above = of xsp1, xshift=-0.6cm, yshift=-0.8cm] {$X_{2\sps}$};
        
        \node [state, fill=gray, fill opacity=0.1, text opacity=1, below = of dot1, xshift=-.6cm, yshift=-0.6cm] (y) {$Y$};

        \path (x1) edge [line width=.5mm] (y);
        \path (x2) edge [line width=.5mm] (y);
        \path (xsm1) edge [line width=.5mm] (y);
        \path (xs) edge [line width=.5mm] (y);

        \path (x1) edge (xsp1);
        \path (x2) edge (xsp2);
        \path (xsm1) edge (x2sm1);

        \path (x1) edge (xs);
        \path (xsm1) edge (xs);
        \path (xsp1) edge (x2s);
        \path (x2sm1) edge (x2s);

        \draw [decorate, decoration = {brace, amplitude=10pt}, -] ([xshift=0.5cm] xsm1.north east) -- ([yshift=-1.2cm, xshift=0.5cm] xsm1.south east) node[midway, xshift=0.5cm]{\Large $\trusupp$};
    \end{tikzpicture} 
    \end{minipage}%
    \begin{minipage}[b]{0.5\linewidth}
        \centering
        \includegraphics[scale=.6]{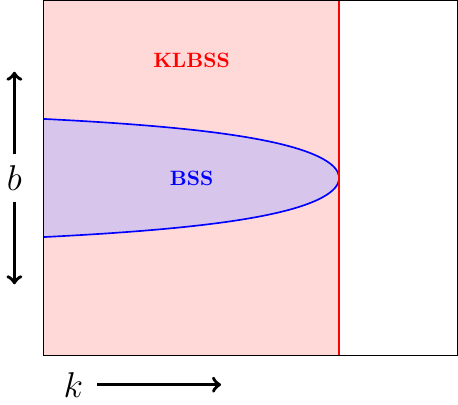}
    \end{minipage}
    \caption{(Left) The DAG of the SEM in Example~\ref{ex:pathcancel} with $\dd=2\sps$ nodes. The true parents (support) of $Y$ are $\trusupp=\{X_1,X_2,\ldots,X_{\sps-1},X_{\sps}\}$ and the remaining nodes are shaded. The edges from $\trusupp$ to $Y$ are in bold.
    (Right) Recovery performance of \klbss{} and \bss{} in terms of parameters $k$ and $b$: Shaded regions indicate parameters for which each method achieves a fixed recovery probability. \klbss{} is independent of $b$ while the performance of \bss{} quickly degrades as $b^2$ increases. The unshaded region on the right indicates the parameter tuples $(k,b)$ for which neither method achieves the same recovery probability.}
    \label{fig:ex:pathcancel}
\end{figure}

\begin{example}[Path cancellation]\label{ex:pathcancel}
    To illustrate the effect of path cancellation, in Appendix~\ref{app:general:ex:pathcancel} we construct a two-parameter family of SEM over $\dd=2\sps$ nodes (Figure~\ref{fig:ex:pathcancel}) where both eigenvalues can be computed and compared, and for which \klbss{} still significantly outperforms \bss{}.
    The two key parameters are the strength of dependence $b$ in the latent SEM (i.e. the edge coefficients), which plays the same role as in Example~\ref{ex:intro}, and the number of parents $k$ of $X_\sps$. 
    In this example, increasing $k$ also increases the number of potential paths between $(X_1,\ldots,X_{\sps-1})$ and $Y$, so the parameter $k$ effectively controls the amount of path cancellation, with $k=\sps-1$ maximizing the amount of path cancellation, and $k=0$ eliminating path cancellation altogether.

    We can compute both eigenvalues for this family as follows:
    \begin{align}
        \eigvalb(\Sigma)
        \asymp \frac1{1 + b^2 + \frac{k}{s-k}},
        \qquad
        \eigvalk(\Sigma)
        \asymp \frac1{1 + \frac{k}{s-k}} \,.
    \end{align}
    In particular, $\eigvalk(\Sigma)>\eigvalb(\Sigma)$ for any $k$ and any $b\ne 0$.
    When $k=0$, we recover the same behaviour as Example~\ref{ex:intro}. When $k>0$, path cancellation becomes possible between $(X_1,\ldots,X_{\sps-1})$ and $Y$ and as $k$ increases both eigenvalues shrink. 
    The takeaway is that while both eigenvalues are affected by path cancellation through $k$, only the \bss{} eigenvalue $\eigvalb(\Sigma)$ is affected by covariate dependence through $b$. In fact, for even moderate coefficient sizes, the performance of \bss{} degrades quickly. See Figure~\ref{fig:ex:pathcancel}.

    Since $\eigvalb(\Sigma)$ and $\eigvalk(\Sigma)$ capture the worst-case, minimax performance of each method, the calculations of these eigenvalues must consider the worst-case behaviour of any regression vector $\beta$ via the minimizations in~\eqref{eq:eigval:bss} and~\eqref{eq:eigval:klbss}. 
    The result is that even though path cancellation only affects very specific parameterizations where cancellation occurs (i.e. for certain combinations of $\beta$ and the SEM coefficients), this will affect the minimax rate through the eigenvalues.
    Nonetheless, this type of cancellation is ``rare'' in the sense that randomly sampled SEM will (almost surely) not exhibit path cancellation~\citep[Theorem~3.2, ][]{spirtes2000}. One can characterize this ``rareness'' using the technical machinery introduced in Appendix~\ref{sec:disc:analysis}; see details in Appendix~\ref{app:general:ex:pathcancel}.
    Nonetheless, even with such cancellation, \klbss{} still outperforms \bss{} since $\eigvalk(\Sigma)> \eigvalb(\Sigma)$ when $b\ne0$. 
\end{example}

We conclude with a concrete example to demonstrate these ideas and for direct comparison with existing methods. 

\begin{example}[Comparison with existing methods]\label{exmp:motiv}
Consider the following SEM, with $b>\betam >0$:

\begin{minipage}[c]{0.37\linewidth}
    \begin{tikzpicture}[scale=0.75, transform shape]
        \node[state] (x1) {$X_1$};
        \node (dot1) [below = of x1, yshift=1.2cm] {$\vdots$};
        \node [state, below = of dot1, yshift=1cm] (xs) {$X_{\sps}$};
        \node [state, fill=gray, fill opacity=0.5, text opacity=1, above right = of x1, yshift=-1.1cm, xshift=1cm] (xs1) {$X_{\sps+1}$};
        \node [state, fill=gray, fill opacity=0.5, text opacity=1, below = of xs1, yshift=0.8cm] (xs2) {$X_{\sps+2}$};
        \node [below = of xs2, yshift=1.2cm] (dot2) {$\vdots$};
        \node [state, fill=gray, fill opacity=0.5, text opacity=1, below = of dot2, yshift=1cm] (xd) {$X_{\dd}$};
        \node [state, fill=gray, fill opacity=0.1, text opacity=1, left = of dot1, xshift=-0.7cm] (y) {$Y$};

        \path (x1) edge (y);
        \path (xs) edge (y);
        \path (x1) edge (xs1);
        \path (x1) edge (xs2);
        \path (x1) edge (xd);
        \path (xs) edge (xs1);
        \path (xs) edge (xs2);
        \path (xs) edge (xd);

        \node [draw=black,dashed,fit=(x1) (dot1) (xs), inner sep=0.1cm] (supp) {};
        \draw [decorate, decoration = {brace, amplitude=5pt}, -] ([yshift=.1cm] supp.north west) -- ([yshift=.1cm] supp.north east) node[midway, yshift=.2cm]{$\trusupp$};
    \end{tikzpicture} 
\end{minipage}
\begin{minipage}[c]{0.53\linewidth}
\begin{align}
\label{eq:opt:mtvtexmp}
\begin{aligned}
    & X_k = \begin{cases}
    \epsilon_k,  & k\in [\sps] \\
    \sum_{j\in\pa(k)}b_{jk}X_j + \epsilon_k & k \in \{\sps+1,\ldots,\dd\} \\
    \end{cases}\\
    &Y = \betam \times \sum_{k=1}^\sps X_k + \epsilon \\
    & \epsilon,\epsilon_k\sim \mathcal{N}(0,1) \qquad b_{jk}=b,\, \forall j,k \,.
\end{aligned}
\end{align}
\vfil
\end{minipage}

The uniform choice $b_{jk}\equiv b$ makes the calculation below easier, and helps to expose the dependence on the SEM coefficients $b_{jk}$ more explicitly. 
The underlying DAG is a bipartite graph with two layers, where the true support is $\trusupp=[\sps]$ and the alternative variables $\{{\sps+1},\ldots,\dd\}$ form the second layer. The covariance and the correlation matrix are
\begin{align*}
    \Sigma  = \begin{pmatrix}
        \Sigma_{\trusupp \trusupp} & \Sigma_{\trusupp \trusupp^c} \\
        \Sigma_{\trusupp^c \trusupp} & \Sigma_{\trusupp^c \trusupp^c}
    \end{pmatrix} = 
    \begin{pmatrix}
        I_\sps & b \mathbf{1}_\sps \mathbf{1}_{\dd-\sps}\T \\
        b\mathbf{1}_{\dd-\sps} \mathbf{1}_{\sps}\T & I_{\dd-\sps} + \sps b^2\mathbf{1}_{\dd-\sps}\mathbf{1}_{\dd-\sps}\T
    \end{pmatrix}\AND
    \widetilde{\Sigma}_{ij} = \frac{\Sigma_{ij}}{\sqrt{\Sigma_{ii}\Sigma_{jj}}}.
\end{align*}
Table~\ref{tab:compare} compares the performance of \klbss{} on this model with existing methods. We include \bss{}, OMP, Lasso, and Lasso-based multi-stage methods, which depend on additional eigenvalue conditions that \bss{} and \klbss{} do not require. In this example,
\klbss{} outperforms \bss{} at least by a factor of $\sps$ in sample complexity, and other methods either require large sample size (again by a factor of $\sps$), or fail to meet the existing conditions for exact recovery. Note also the additional dependence of \bss{} on the SEM coefficients $b^2$, which \klbss{} avoids.
\end{example}

Examples~\ref{eq:sem:fail}-\ref{exmp:motiv} provide insight into how and why \klbss{} outperforms classical approaches in SEM.
Later, in Section~\ref{sec:expt}, we provide additional empirical evidence that $\Sigmaspace_\Delta$~\eqref{eq:opt:Omegagap:real} holds approximately 30-50\% of the time, depending on the topology of the SEM.
Nonetheless, it is important to bear in mind that even when this condition is not satisfied, \klbss{} will perform as well as \bss{}, so that \klbss{} still enjoys the desirable properties of \bss{} even when the model is not an SEM.

\begin{table}[t]
\centering
\caption{Summary of comparison of \klbss{} with existing methods under model~\eqref{eq:opt:mtvtexmp}. The first column is the names of methods under comparison; the second column is the sample complexity upper bound of each method specified under~\eqref{eq:opt:mtvtexmp} or the conclusion about the critical condition of the method; the last column is further explanation about the second column.}
\begin{tabular}{@{}l|l|l}
\toprule
                                & Sample complexity under~\eqref{eq:opt:mtvtexmp}             & Comments                                                                                   \\ \midrule
\klbss{}                        & $n\gtrsim \log(\dd-\sps)/\betam^2 \vee \log\binom{\dd-\sps}{\sps}$                                                    \\
\bss{}                          & $n\gtrsim b^2\sps^2\log(\dd-\sps) /\betam^2 $                                                               \\
Multi-stage methods & $n\gtrsim b^2\sps^2\log\dd$ & Needed to satisfy RE condition                                                                                   \\
 Lasso       & Irrepresentability Fails & $\gamma > \sqrt{\sps}/2\to \infty $ violates $\gamma<1$                                                          \\
OMP             & Mutual incoherence Fails & $\mu = \frac{\sps b^2}{\sps b^2+1}\to 1$ violates $\mu\le \frac{1}{2\sps-1}$ \\ 
\bottomrule
\end{tabular}
\label{tab:compare}
\end{table}

\section{Practical considerations}\label{sec:prac}
Similar to \bss{}, \klbss{} as described in Algorithm~\ref{alg:main:simple} involves an exhaustive search over all candidate supports and assumes various problem parameters such as the sparsity level are known. This is for theoretical convenience and simplicity, and is not necessary in practice.
We now discuss how to implement \klbss{} using standard solvers, and extend this implementation to more practical settings where problem parameters are unknown.

\subsection{Vanilla \klbss{}: A mixed integer reformulation}\label{sec:prac:mip}
We can reformulate \klbss{} as a standard Mixed Integer Program (MIP), borrowing from \citet{bertsimas2016best}. In this way, we can leverage recent advances in MIP solvers to achieve faster computation as well eliminate the dependence on random ordering.
Given $\betaspace_{\dd,\sps}(\betam)$, we solve the following program:
\begin{align}
    & \min_{\beta,{\gamma},z,w} \frac{1}{n-\sps}\|Y - X\beta\|^2 + (\beta - {\gamma})\T \frac{X\T X}{n}(\beta - {\gamma}) \label{eq:mip:obj} \\
    \text{such that} \qquad & z \in\{0,1\}^\dd,\quad w\in\{0,1\}^\dd,\quad \sum_{k} z_k = \sps \label{eq:mip:constraint} \\ \nonumber
    \forall k\in[\dd], \quad & (\beta_k,1-z_k): \text{SOS-1}\,,\quad  ({\gamma}_k,1-z_k): \text{SOS-1} \\ \nonumber
    & {\gamma}_k + Mw_k \ge \betam z_k\,, \quad  -{\gamma}_k + M(1-w_k) \ge \betam z_k\,,
\end{align}
where SOS-1 stands for Ordered Sets of Type 1, which means at most one variable in $(\beta_k,1-z_k)$ can be nonzero. Each $z_k\in\{0,1\}$ indicates whether the $k$-th variable is selected, hence the final estimate would be $\supp(z)$. 
We also introduce integer variables $w_k$ to enforce the nonconvex absolute value constraints $|{\gamma}_k|\ge \betam$, i.e. $\gamma_k\ge \betam$ or $\gamma_k\le -\betam$, with a large enough positive constant parameter $M$, e.g. $M\ge \text{upper bound of }|\beta_k|+\betam$. In this way, only one of the last two constraints would be effective with $M$ large enough.
The equality constraint $\sum_{k}z_k=\sps$ can be replaced by $\sum_{k}z_k\le \ubsps$ for the extension in Section~\ref{sec:prac:unknown}.

This formulation is equivalent to Algorithm~\ref{alg:main:simple}, except that it skips the step of partialling out the support intersection when evaluating scores in Algorithm~\ref{alg:main:twoccase2}. It follows that the estimator no longer depends on the candidate support ordering, which is convenient in practice for reproducibility. In Section~\ref{sec:expt}, we evaluate the effect of skipping this step: In practice, this actually slightly \emph{improves} the performance on average. But there is a tradeoff in terms of the worst-case performance; for more details see  Appendix~\ref{app:klbssv}. We refer to this modified approach that skips the partialing step as Vanilla \klbss{}. Using standard MIP solvers, Vanilla~\klbss{} scales to large problem sizes (up to 1000 variables in our experiments).

\subsection{Practical extensions}\label{sec:prac:unknown}
We can extend \klbss{} to the case where both the exact sparsity $\sps$ and $\betam$ are unknown.
We start with the unknown sparsity case by assuming an upper bound $\ubsps\ge \sps$ is given.
In this case, the space of candidate supports expands from $\suppspace$ to $\suppspaceub=\{S\subseteq[\dd]:|S|\le \ubsps\}$.
Then we adopt the same strategy of minimizing scores of candidate supports with an additional penalty proportional to their cardinality. 
Specifically, we modify the MIP objective~\eqref{eq:mip:obj} to
\begin{align*}
    \min_{\beta,{\gamma},z,w} \frac{1}{n-\sps}\|Y - X\beta\|^2 + (\beta - {\gamma})\T \frac{X\T X}{n}(\beta - {\gamma}) + \penalt \sum_kz_k \,
\end{align*}
and replace the exact sparsity constraint~\eqref{eq:mip:constraint} by an upper bound $\sum_kz_k\le\ubsps$.
Here, the parameter $\penalt$ measures the strength of the penalty. A choice based on $\mclass$ leads to a sample complexity analogous to that in  Theorem~\ref{thm:pointwise:eigen}, with a modified signal to account for the enlarged support space, see Theorem~\ref{thm:main:ub:simplefull:unknown} in Appendix~\ref{sec:main:unknown}.
In practice, popular choices are $\penalt=\log n$ (BIC) and $\penalt=\log \dd$ (extended BIC), whose performance will be investigated in the experiments in Section~\ref{sec:expt:param}.
The same extension also applies without the MIP reformulation; e.g. we can modify the output in Algorithm~\ref{alg:main:twoccase2} as
\begin{align*}
\estsupp := \argmin_{D\in\{S,T\}}\big(\mathcal{L}(D;(S,T)) + \penalt|D|\big)\,.
\end{align*}
Then the algorithm inputs are changed and \klbss{} is called as $\klbss{}(X,Y,\betaspace,\ubsps,\penalt)$.

We next consider the case where $\betam$ is also unknown. When $\betam$ is unknown, we must choose a surrogate value $\widetilde{\beta}_{\min}$ to plug into the input space $\widetilde{\betaspace}:=\betaspace(\widetilde{\beta}_{\min})$. In Appendix~\ref{sec:disc:adap}, we provide a theoretical choice achieving the sample complexity in Theorem~\ref{thm:opt:eigen}.
In practice, when there is no guidance on the choice of $\widetilde{\beta}_{\min}$, a smaller value of $\widetilde{\beta}_{\min}$ is conservative but safe, because it will not over-penalize the true support, as verified by the experiments in Section~\ref{sec:expt:param}.
In practice, we can also use cross-validation (CV) to choose $\widetilde{\beta}_{\min}$.
Given a range of possible choices $\{\betam^\ell\}_{\ell=1}^L$, we consider the standard $K$-fold CV procedure for choosing $\widetilde{\beta}_{\min}$ and estimating the support.

In Algorithm~\ref{alg:main:cv}, we provide a consolidated procedure that does not require the knowledge of exact sparsity and $\betam$, which outlines the detailed steps of CV.
This approach is generic and directly applies to both \klbss{} (Algorithm~\ref{alg:main:simple}) and Vanilla \klbss{} (MIP formulation, Section~\ref{sec:prac:mip}).
The effectiveness of selecting $\widetilde{\beta}_{\min}$ by CV is verified by the experiments in Section~\ref{sec:expt:param}, and the performance of Algorithm~\ref{alg:main:cv} is demonstrated in Section~\ref{sec:expt:simu}.

\begin{algorithm}[t]
    \caption{\klbss{} with unknown sparsity and $\betam$}\label{alg:main:cv}
    \raggedright
    \textbf{Input:} Data matrix $X$; response $Y$; candidate choices $\{\betam^\ell\}_{\ell=1}^L$; sparsity upper bound $\ubsps$; penalty $\penalt$\\
    \textbf{Output:} Estimated support $\estsupp$.
    \setlist{nolistsep}
    \begin{enumerate}
        \item Randomly divide the dataset $\mathcal{D}=(X,Y)$ into $K$ folds: $\mathcal{D}^{(1)},\ldots,\mathcal{D}^{(K)}$, let $\mathcal{D}^{(-k)}=\cup_{j\ne k}\mathcal{D}^{(j)}$;
        \item For $\ell = 1,2,\ldots,L$ and $k=1,2,\ldots,K$:
        \begin{enumerate}
            \item Let $\estsupp^{(k)}_\ell = \klbss{}(\mathcal{D}^{(-k)}, \betaspace(\betam^\ell),\ubsps,\tau)$;
            \item Use data $\mathcal{D}^{(-k)}$ to compute $\widehat{\beta}^{(k)}_\ell$, the OLS vector of regressing $Y$ onto $\estsupp^{(k)}_\ell$;
            \item Let $r^{(k)}_\ell:= \sum_{i\in \mathcal{D}^{(k)}}(Y_i - X_{i,\estsupp^{(k)}}\T \widehat{\beta}^{(k)}_\ell)^2$;
            \item Let $r_\ell := \sum_{k=1}^K r^{(k)}_\ell$ be the risk of $\betam^\ell$;
        \end{enumerate}
        \item Output $\estsupp= \klbss{}(\mathcal{D}, \betaspace(\betam^{\widehat{\ell}}),\ubsps,\tau)$, where $\widehat{\ell}=\argmin_{\ell\in[L]}r_\ell$.
    \end{enumerate}
\end{algorithm}

\begin{remark}[Standardization]\label{rmk:stdz}
    Standardization is a common pre-processing step in regression analyses. A simple approach to handling standardized data is to use a data-driven choice of $\betam$, which is also applicable with the CV framework in Algorithm~\ref{alg:main:cv}: Instead of fixing one $\betam$ for all variables, replace $\betam$ with $\betam\cdot\widehat{sd}_j$ in \eqref{eq:betamin},
    where $\widehat{sd}_j$ is the sample standard deviation of $X_j$. This adjusts for the effect of standardization, only requires tuning one single parameter $\betam$, and yields similar performance to the experiments in Section~\ref{sec:expt}.
    Moreover, in the theoretical results, standardization does not affect $\eigvalk$, except that the space of minimization $\Theta_{\trusupp\triangle T}$ in Definition~\ref{defn:eigval:klbss} is changed according to the sample standard deviations. 
    Standardization also does not affect the relationship $\eigvalk(\Sigma)\ge\eigvalb(\Sigma)$ in Lemma~\ref{lem:eigen:comp}, and the space of covariance matrices $\Sigmaspace_\Delta$ where strict improvement takes place, i.e. when the requirement of $\Sigmaspace_\Delta$ holds pre-standardization, it still holds post-standardization. 
\end{remark}

\subsection{Computational complexity}\label{sec:prac:comp}
Since both \bss{} and \klbss{} can be implemented as an MIP, in practice the required computation for each method is on the same order, and this will be confirmed later by our experiments in Section~\ref{sec:expt:param}. In order to provide a more precise worst-case comparison, however, we can use the na\"ive tournament-style interpretation of each method to compare their respective computational complexity as follows.
As such, this is not intended to be a rigourous complexity analysis, but rather a worst-case comparison to highlight the small computational cost of \klbss{} relative to \bss{}.

Given the sparsity level $\sps$, \bss{} searches over all possible candidates, which has size $\binom{\dd}{\sps}\asymp \dd^\sps$. For each candidate, \bss{} needs to evaluate the residual variance. By contrast, \klbss{} conducts $\binom{\dd}{\sps}-1$ pairwise comparisons along the given order, and each comparison requires computations of scores (cf.~\eqref{eq:compare:score}) for both candidates. 
These can be obtained na\"ively by solving $2^\sps$ quadratic programs with box constraints, each of which (as well as the residual variance computation for \bss{}) are standard convex programs. Thus, the computational complexity of \klbss{} is, in the worst-case,
\begin{align*}
\bigg[\binom{\dd}{\sps}-1\bigg] \times 2 \times (2^\sps+1) \asymp (2\dd)^\sps \,.
\end{align*}
Compared with \bss{} with complexity $\dd^\sps$, the cost that \klbss{} pays is of smaller order.

\begin{remark}\label{rem:compstat}
    A natural question is whether or not neighbourhood selection in SEM can be achieved with efficient (i.e. polynomial-time) estimators, such as a Lasso-based method. It is known that the Lasso needs strong conditions on the design matrix, which we have already shown are easily violated in SEM (Examples~\ref{ex:lasso:fail} and \ref{exmp:motiv}). In fact, this can be strengthened: Under standard conjectures in complexity theory, \emph{any} polynomial-time estimator for support recovery under general design cannot avoid the restricted eigenvalue condition, even if the sparsity is known \citep{gao2025optimality}.
\end{remark}

\section{Experiments}\label{sec:expt}
In this section, we conduct experiments to empirically evaluate the performance of \klbss{}, in particular compared to \bss{} and the Lasso. We start with a comprehensive simulation study to show a significant sample complexity gap in randomly generated SEMs. Next, we validate various practical aspects discussed in Section~\ref{sec:prac}. Finally, we compare \klbss{} and \bss{} in the context of structure learning and a real-data application using gene expression data, assessing both recovery and prediction performance.

\subsection{Simulation study: empirical sample complexity gap}\label{sec:expt:simu}

\begin{figure}[t]
    \centering
    \includegraphics[width=1.\linewidth]{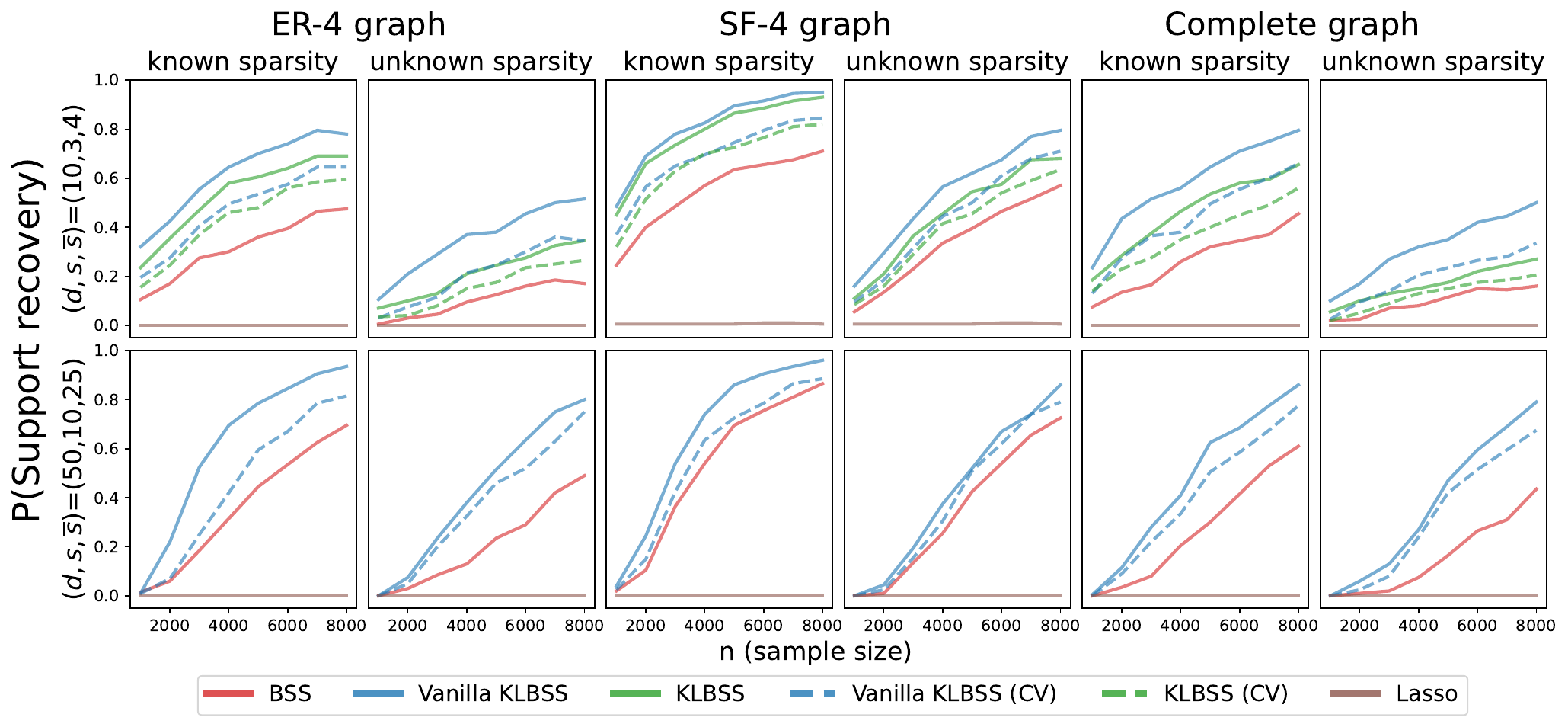}
    \caption{Comparison on support recovery performance of \bss{}, \klbss{} and Lasso under different types of graphs and dimensions $(\dd,\sps,\ubsps)$ averaged over $200$ replications. The horizontal axis is sample size, the vertical axis is probability of exact recovery. The first/middle/last two columns are for ER graph, SF graph, and complete graph. The left/right columns of each graph type indicate known and unknown sparsity cases. 
    The solid/dashed lines indicate known and unknown $\betam$ (by CV).
    There is a notable performance gap between \klbss{} and \bss{}. Lasso is never consistent.}
    \label{fig:main1}
\end{figure}

We begin with a simulation study where the ground truth is known and we can simulate from different types of SEM.
Full experiment results and all implementation details can be found in Appendix~\ref{app:expt}; here 
we briefly illustrate a representative slice of the results in Figure~\ref{fig:main1} to show the \emph{empirical improvements} in the sample complexity.
Results for other metrics, e.g. Hamming distance, TPR, FDR are available in Appendix~\ref{app:expt:add}, particularly a summary across all simulation setups is given in Figure~\ref{fig:eval2} in Appendix~\ref{app:expt:eval}, showing a significant and overall improvement over \bss{}.

The results cover various combinations of $(\dd,\sps,\ubsps)$ and three types of underlying DAGs: Erd\"os-R\'enyi (ER) and Scale-Free (SF) graphs with expected number of edges to be $4\dd$, and Complete graphs where all possible edges are present, and the data is simulated according to~\eqref{eq:pre:sem} (specifically~\eqref{eq:lm} and~\eqref{eq:sigmamin}) where the noise $\{\epsilon_k\}_{k=1}^\dd$ have mixed, possibly non-Gaussian, distributions, the true supports and SEM coefficients are randomly sampled.
For $\dd=50$, we implement both \klbss{} and \bss{} with the MIP reformulation in Section~\ref{sec:prac:mip}; for unknown sparsity, we input with $\ubsps$ and apply the BIC penalty discussed in Section~\ref{sec:prac:unknown}; for unknown $\betam$, we apply the CV procedure introduced in Algorithm~\ref{alg:main:cv}.
Both ER-4 and SF-4 graphs are sparse random graph ensembles whereas Complete graphs represent a dense graph setting; the latter is particularly interesting since it represents a setting that is closer to the general design setting where \bss{} is commonly perceived to be optimal. In each instance, the optimality conditions imposed in~\eqref{eq:opt:Omega:real}
are not always guaranteed to be satisfied, and thus our simulations also represent a more realistic evaluation where our theoretical assumptions are likely to be violated. In particular, path cancellation is ubiquitous and we make no efforts to eliminate path cancellation. 

The results in Figure~\ref{fig:main1} confirm that \klbss{} is significantly more sample-efficient compared to \bss{} and the Lasso and show \klbss{} is robust to deviations from our theoretical assumptions.
This helps illustrate the benefits and improvement \klbss{} brings in a more \emph{general and practical} class of SEM.
Moreover, Vanilla \klbss{} performs slightly better than \klbss{} in the average sense. This does not contradict the minimax optimality of \klbss{}: In Appendix~\ref{app:expt:fxv}, we demonstrate that there exist hard cases where Vanilla \klbss{} performs worse than \klbss{}.
The Lasso fails in all graphs even with large sample size due to the strong dependence structure in $\Sigma$.

\subsection{Choice of unknown parameters and time complexity}\label{sec:expt:param}

\begin{figure}[t]
    \centering
    \includegraphics[width=1.\linewidth]{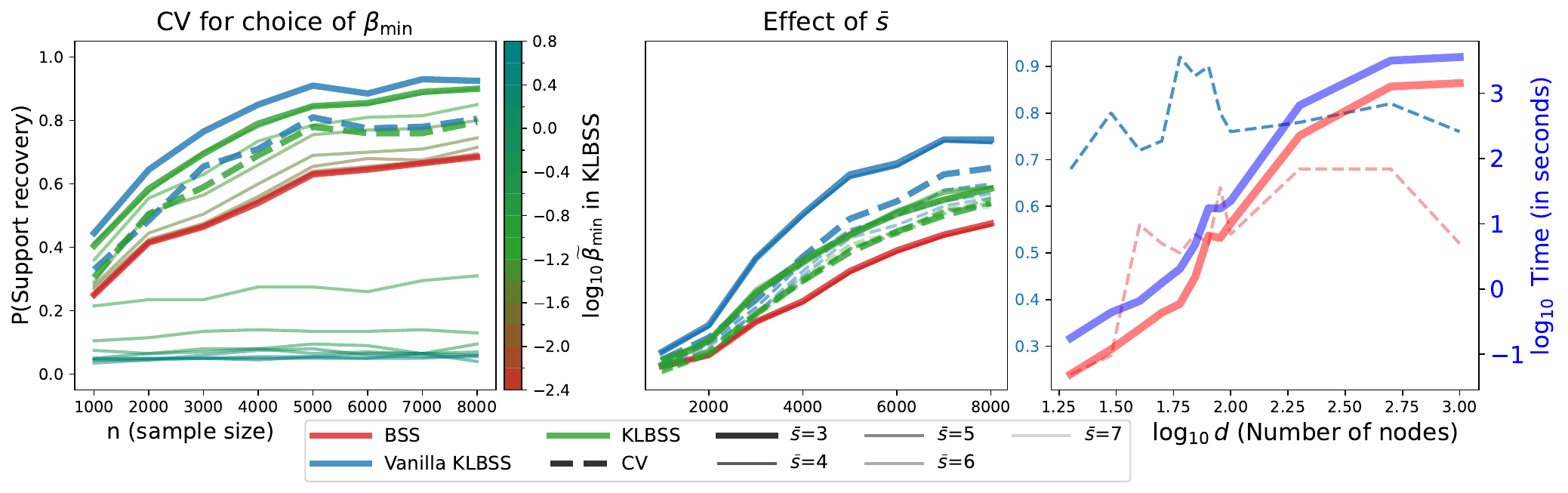}
    \caption{
        (Left) Cross-validation for the choice of $\betam$. The solid lines plot \klbss{} with correct $\betam$. The dashed lines plot the CV performance. The thinner lines plot \klbss{} with each candidate of $\betam$'s, ranging from red to green and to cyan. The CV estimate is slightly insuperior to \klbss{} with correct $\betam$, but still performs better than \bss{}.
        (Middle) Effect of unknown sparsity on recovery performance.  \klbss{} and \bss{} with various specifications of $\ubsps$, indicated by the opacity. Performance of CV is included by dashed lines. The performance of each methods is robust to the given sparsity upper bound (lines are overlapped due to similar performances).
        (Right) Time complexity in log-log plot (dark blue/red solid lines) and recovery performance (light blue/red dashed lines) of \klbss{}/\bss{} using MIP. \klbss{} runs in the same computation order as \bss{}, incurring a small overhead while achieving better recovery performance. }
    \label{fig:main2}
\end{figure}

Next we explore the various practical aspects discussed in Section~\ref{sec:prac}. Specifically, we examine the effectiveness of CV for selecting $\betam$; examine the robustness of different specifications of $\ubsps$ with unknown sparsity and CV; and investigate the time complexity of \klbss{} using MIP and compare with \bss{}.

In the first two experiments, we consider SF graphs with $\dd = 7, \sps = 3$.
In the left panel of Figure~\ref{fig:main2}, we consider known sparsity and apply CV to select $\betam$ from a range of candidates, and compare the performance against \klbss{} with oracle knowledge of $\betam$ as input. Although the CV-optimized choice leads to a small performance loss compared to the oracle, there is still a significant gap compared to \bss{}.
We also include the performance of \klbss{} when input with each of the candidate $\betam$ values, indicated by the thinner lines and the color bar.
The range spans from red (near zero $\betam$, reducing to \bss{}) to green (correct $\betam$, corresponding to the solid line of \klbss{}) and to cyan (overspecified $\betam$ that is too large).
This demonstrates that CV indeed provides a reasonable choice of $\betam$ from the given range.
For the middle panel of Figure~\ref{fig:main2}, we consider simultaneously unknown sparsity with given and CV-optimized $\betam$.
We apply BIC penalty with $\ubsps$ ranging from 3 to 7 (from true $\sps$ to $\dd$). We observe stable recovery performance across different values of $\ubsps$ and improvement over \bss{}. Similar results were obtained for other information criteria such as extended BIC.

For the right panel of Figure~\ref{fig:main2}, to better understand the time complexity of \klbss{} and compare with \bss{}, both implemented using MIP, we record the time used in solving their respective MIP formulation to a specified MIP gap of 0.01, which represents the tolerance for the solution precision.
We consider ER graphs up to 1000 nodes with $\sps=10$ and $n=5000$ under Gaussian noise.
We observe this MIP gap indeed provides reasonable recovery ability for \klbss{} (light blue dashed line, around 80\%), which is significantly better than \bss{} (light red dashed line, around 50\%). While the computation for \klbss{} to achieve this tolerance is shown in dark blue solid line (by log-log plot), where it takes around 30 seconds for $\dd=100$, and less than an hour for $\dd=1000$. Compared to \bss{}, \klbss{} only pays a small overhead.

\subsection{Structure learning}\label{sec:expt:dag}

Since a primary motivation for this work is neighbourhood selection in SEM, we also apply \klbss{} for structure learning, i.e. to recover the DAG $G$ that generates the data $X$.
When the topological ordering of $G$ is known, the problem reduces to support recovery for each variable from the preceding nodes on the ordering with unknown sparsity.
We follow the experiment setup as in Section~\ref{sec:expt:simu} with ER/SF graphs and $\dd=10$. By setting $\sigma_k\equiv\sigma=2$ for each $k$, the DAG $G$ is identifiable \citep{peters2014identifiability} and can be estimated via the EQVAR algorithm \citep{chen2019causal}.
We consider two cases: 1) A valid oracle topological ordering is given, and 2) The topological ordering is estimated with the EQVAR algorithm from scratch. The latter probes the robustness of \klbss{} to misspecification of the ordering, and illustrates its applicability for structure learning in practice. 
A comparison using Structural Hamming Distance (SHD) between the estimated and true graphs is shown in Figure~\ref{fig:dag}, where significant improvement can be observed for \klbss{} over \bss{}, while the difference between \klbss{} and Vanilla \klbss{} is negligible.
In the second case where the ordering is estimated from data, overall recovery performance is affected for all the methods, but not by much.
Results on Lasso are not shown since it is never consistent and produces SHD around 10-20.

\begin{figure}[t]
    \centering
    \includegraphics[width=1.\linewidth]{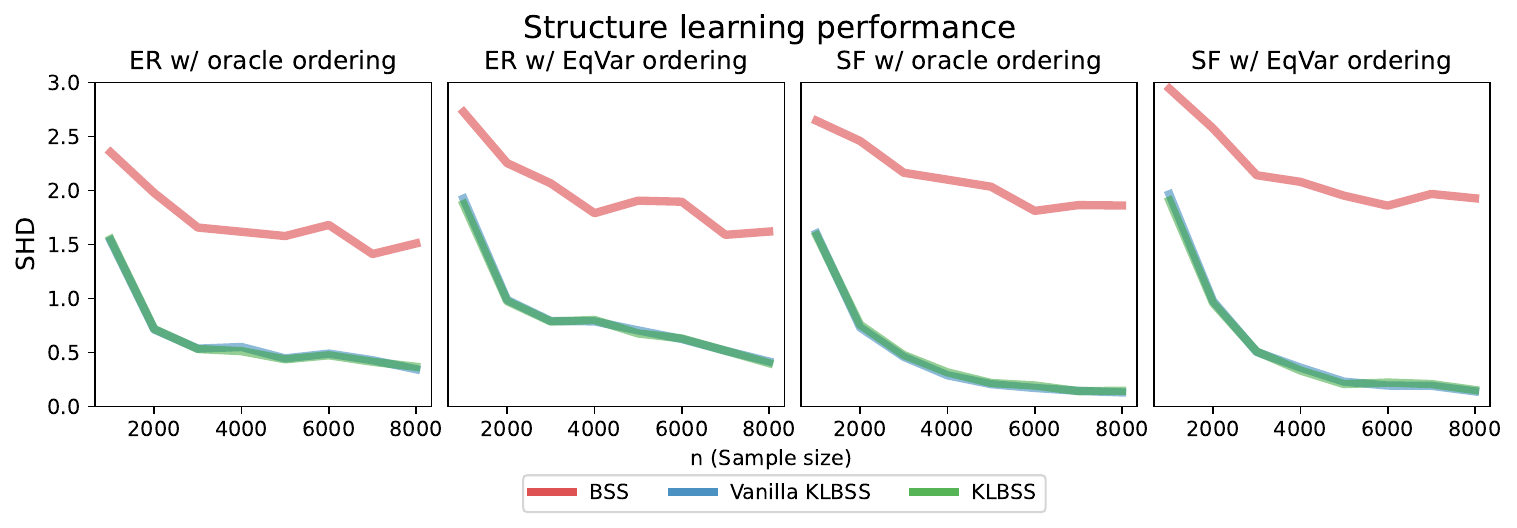}
    \caption{
    Comparison on structure learning performance of \bss{} and \klbss{} on ER and SF graphs: Structural Hamming Distance (SHD) vs. sample size. The first and third panels are input with oracle valid topological ordering of the graph, the second and forth apply EqVar algorithm for ordering estimation, then conduct neighbourhood selection via \klbss{} or \bss{}. 
    \klbss{} gives better structure learning performance than \bss{}. \klbss{} and Vanilla \klbss{} perform similarly thus lines are overlapped.}
    \label{fig:dag}
\end{figure}

\begin{figure}[t]
    \centering
    \includegraphics[width=.65\linewidth]{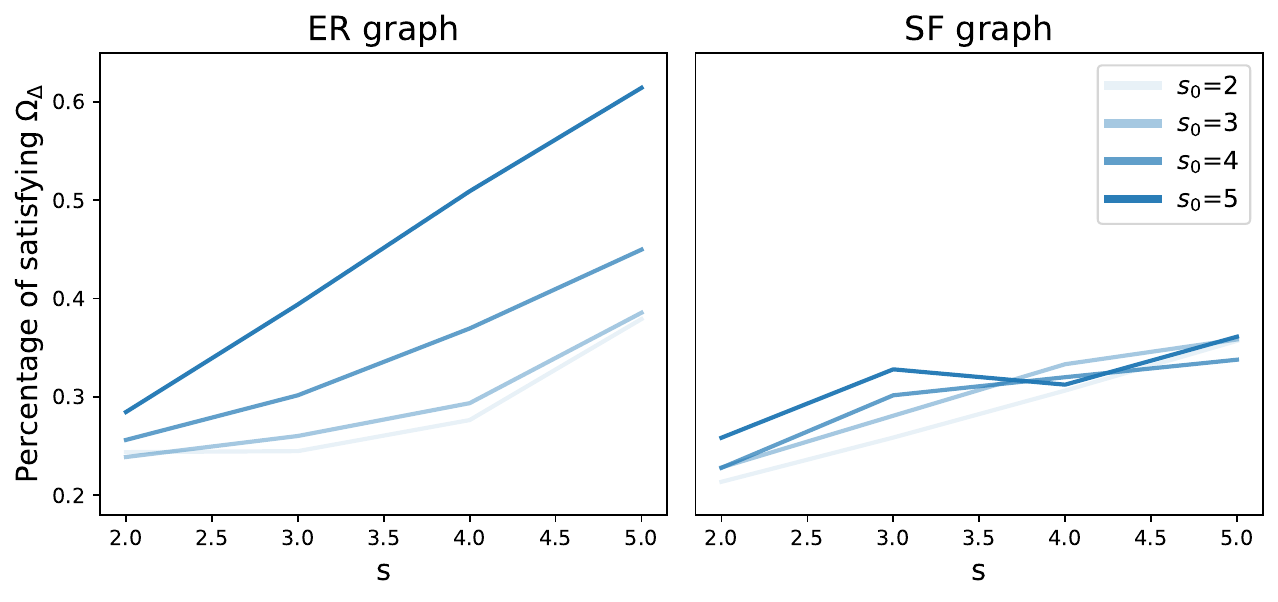}
    \caption{
    The percentage of randomly sampled SEM covariances that satisfy the constraint in $\Sigmaspace_\Delta$~\eqref{eq:opt:Omegagap:real} for ER/SF graphs with various expected number of edges ($s_0 \times \dd$) and sparsity $\sps$ and fixed $\dd=12$. %
    }
    \label{fig:OmegaDelta}
\end{figure}

\subsection{Verification of $\Sigmaspace_\Delta$}\label{sec:expt:OmegaDelta}
In order to explore the prevalence of design matrices where \klbss{} outperforms \bss{}, in this subsection we explore how often the constraint in $\Sigmaspace_\Delta$ is satisfied in randomly generated SEM. 
Recall that $\Sigmaspace_\Delta$ defined in \eqref{eq:opt:Omegagap:real} represents design matrices where \klbss{} has a \emph{provably} better sample complexity according to Theorem~\ref{thm:pointwise:eigen}.
For the random SEMs sampled in Section~\ref{sec:expt:simu}, we check if the inequality in~\eqref{eq:opt:Omegagap:real} holds.
We vary the sparsity of the random graphs by specifying the expected number of edges in $G$ using both ER/SF graphs ($s_0\times \dd$). We consider various sparsity levels $\sps$ for $\beta$ as well. We randomly sample 5,000 SEM covariances (and supports $\trusupp$) and record the proportion that satisfy $\Sigmaspace_\Delta$. 
The result summarized in Figure~\ref{fig:OmegaDelta} indicates a significant proportion of random SEMs have covariance matrices in $\Sigmaspace_\Delta$, i.e. showing improvement of \klbss{}. Especially, the proportion grows as the graph becomes denser and the sparsity level of $\beta$ increases.
The effect of the edge density $s_0$ stands out in ER graphs while it is less significant in SF graphs. Nonetheless, SF graphs still exhibit an average 30\% proportion, and ER graphs can reach as high as 60\%. This confirms that SEM are indeed likely to fall into the class of design matrices where \klbss{} \emph{strictly} improves the sample efficiency of neighbourhood selection.

\subsection{Application to pan-cancer data}\label{sec:expt:real}

Finally, to evaluate the performance of \klbss{} on real data, we apply it to a pan-cancer dataset consisting of RNA-Seq gene expression measurements from $n=801$ patients with 5 different types of cancers \citep{misc_gene_expression_cancer_rna-seq_401}. Since the linear model is certainly misspecified on such data, and there is no known ``ground truth'', this dataset allows us to evaluate 1) robustness to misspecification of linear SEM and 2) performance on downstream prediction tasks.

\paragraph{Selection of genes}
In the first experiment, we use the pan-cancer data for the covariates $X$ and construct the response $Y$ from $X$ by simulation. In this way, we can deal with covariates related by real genetic processes, meanwhile, we also know $\trusupp$ and are able to evaluate the estimate $\estsupp$. 
Specifically, we group the genes according to their variances into $\dd$ bins.
Then for each replication, we randomly sample one gene from each bin to form the $X$ matrix (of dimension $\dd$), with $Y$ simulated as in Section~\ref{sec:expt:simu}. 
We fix $\sps=10$ and show results for increasing $\dd$ from 50 to 90 in the left panel of Figure~\ref{fig:main3} indicated by solid and dashed lines for \klbss{} and \bss{}. The barplots depict the gap in the performances of two methods (difference between solid and dashed lines). We can still observe the improved performance against \bss{}, especially, the gap becomes more pronounced as the dimension gets larger and for smaller sample size.

\paragraph{Empirical evaluation on downstream predictions}
In the second experiment, we avoid simulations altogether.
Since there is no ``true'' support, 
we instead evaluate the selected models by the prediction performance on the gene with the largest variance $(Y)$ using the support estimated from the remaining genes as candidate predictors ($X$). 
For each replication, we randomly choose $\dd=50$ genes as $X$, and randomly split the dataset in training / test sets. We apply \klbss{} and \bss{} with $\sps=10$ on the training set and compute prediction error on the test set.
We use CV for the choice of $\betam$ for \klbss{}, perform $N=100$ replications, and display the result by scatterplot of prediction errors of \klbss{} vs. \bss{} in Figure~\ref{fig:main3}.
In a majority (73\%) of the evaluations, \klbss{} selected genes with a lower out-of-sample prediction error vs. \bss{} (indicated by the points above the $y=x$ line).
Some (24\%) points lie exactly on the $y=x$ line because \klbss{} and \bss{} both estimate the same support $\estsupp$.
This demonstrates that \klbss{} selects models that yield better out-of-sample predictions compared to \bss{}, on realistic data where the underlying model may not be an exact SEM.

\begin{figure}[t]
    \centering
    \includegraphics[width=.8\linewidth]{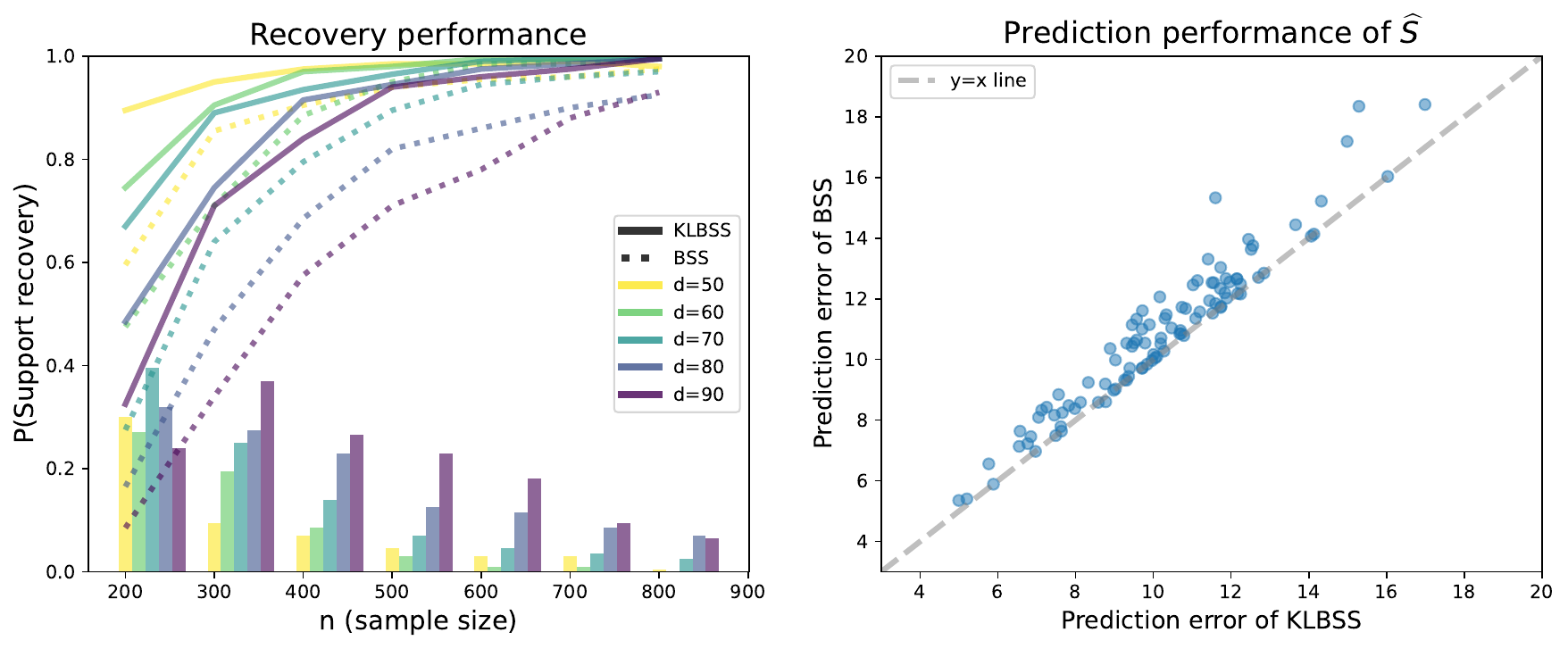}
    \caption{
    (Left) Recovery performance comparison on RNA-Seq gene expression data with $\sps=10$. The solid lines are for \klbss{} and dashed lines are for \bss{}. The barplots indicate the gap between the performances of \klbss{} and \bss{}, which is more significant for larger $\dd$ and smaller sample size.
    (Right) Prediction performance of $\estsupp$ given by \klbss{} with CV choice of $\betam$ and \bss{} on gene expression data. Each point is one random sampling of $\dd=50$ variables from $20,531$ genes to choose $\sps=10$ in training set for prediction in test set. The grey dashed line is the $y=x$ line.
    The genes selected by \klbss{} result in better prediction error than \bss{} (above the $y=x$ line).}
    \label{fig:main3}
\end{figure}

\section{Conclusion}\label{sec:conc}

In this paper, we studied the problem of neighbourhood selection (also known as support recovery, variable selection, and Markov boundary learning) in an SEM. We observed that existing results for general design fail to capture the nuances of this problem and are overly pessimistic as a result. Inspired by this observation, we proposed \klbss{}, a new method for support recovery that excels for neighbourhood selection in SEM.
Through a detailed pointwise and minimax analysis of neighbourhood selection, as well as extensive experiments, we showed that \klbss{} indeed improves upon \bss{} in both selection performance and prediction, confirming that the pessimism of \bss{} is not just a theoretical artifact. 

This has several important consequences. Most importantly, for applications of structure learning (e.g. causal discovery and causal machine learning), we should not simply default to standard approaches. This is especially important given the trend in recent years to focus mostly on topological order recovery in SEM, and to leave neighbourhood selection to existing methods such as \bss{} and the Lasso. Our work shows that there is still much to be learned about neighbourhood selection, the second stage of structure learning, and our results provide a foundation for future study in this direction. 

A useful property of \klbss{} is that its performance at worst degenerates to the performance of \bss{}, meaning that in practice \klbss{} inherits all of the desirable properties of \bss{} at a small computational cost (and of course, with significant statistical improvements). An important problem for future work is to develop computationally efficient approaches to neighbourhood selection, although it is worth recalling that this problem is NP-hard and there are even stronger obstructions in general (see Remark~\ref{rem:compstat}). Thus, it remains to understand these computational tradeoffs more precisely and to design algorithms that realize these tradeoffs (e.g. under stronger assumptions).

Finally, an intriguing aspect of \klbss{} is that it does not explicitly use structural information about the DAG $G$ or the model $\Sigmaspace$. More formally, the implementation of \klbss{} does not depend in any way on $G$ or $\Sigmaspace$. This shows that improvements to variable selection in structured settings can be achieved even when this structure is unknown to the statistician.
The resulting analysis of \klbss{} should be of independent interest, and helps provide some insight into how unknown structure can be exploited. This is crucial in applications where structure is present but unknown.

\bibliographystyle{abbrvnat} 
\bibliography{SuppRec}

\newpage
\appendix

\begin{center}
    \Large\bf Supplementary Materials for ``\klbss{}: Rethinking optimality \\ for neighbourhood selection in structural equation models''
\end{center}

\bigskip

\noindent
In Appendix~\ref{sec:disc}, we discuss various technical aspects of \klbss{}, including its KL-divergence interpretation and theoretical extensions. 
In particular, Appendix~\ref{sec:disc:analysis} introduces signal definitions that are used in the analysis of \klbss{} and offers key intuitions that are useful for the subsequent proofs.
Proofs of the main SEM results and examples are in Appendices~\ref{app:opt}-\ref{app:general:ex}. These proofs rely on a detailed technical analysis of \klbss{}, which can be found in Appendix~\ref{app:main}. The connection between neighbourhood selection and support recovery is discussed in Appendix~\ref{app:pre:eqvmb}. Other technical tools used are in Appendix~\ref{app:chisq}-\ref{app:fano}. Finally, Appendix~\ref{app:expt} gives all the experiment and implementation details.

In all appendices, for all the displays and technical proofs, we write the conditional variance more formally by specifying the set of variables, i.e. $\Sigma_{S\setminus T\given T} = \cov(X_{S\setminus T}\given X_T)$, to make the matrix size clear.

\section{Discussion}\label{sec:disc}
This appendix starts by introducing the signals of \klbss{} and \bss{}, which will be useful in the analysis and proofs (Appendix~\ref{sec:disc:analysis}). After this, we collect miscellaneous (optional) discussions for interested readers: Interpreting \klbss{} (Appendix~\ref{sec:disc:interpret}),
analyzing Vanilla \klbss{}, which was introduced for computational reasons (Appendix~\ref{app:klbssv}),
theoretical results for unknown sparsity (Appendix~\ref{sec:main:unknown}) and unknown $\betam$ (Appendix~\ref{sec:disc:adap}), and finally non-Gaussian designs (Appendix~\ref{sec:disc:ext}).

\subsection{Signal and analysis of \klbss{} in general support recovery}\label{sec:disc:analysis}
We introduce the signals used in the analysis, which explicitly illustrate the deficiency in \bss{}: There is an additional signal component that is being ignored.
To see this, let us first define the signals for distinguishing two supports $S$ and $T$:
\begin{definition}\label{defn:main:twocase:signal}
For any $(\beta,\Sigma,\sigma^2)\in\mclass(\betaspace,\Sigmaspace
,\sigma^2)$, and any two sets $S,T\subseteq [\dd]$, define 
\begin{align}\label{eq:main:twocase:signal}
\begin{aligned}
    \signalone(S,T) &:= \frac{1}{\sigma^2}\beta\T_{S\setminus T}\Sigma_{S\setminus T\given T}\beta_{S\setminus T}, \\
    \signaltwo(S,T) &:= \frac{1}{\sigma^2}\min_{{\alpha}\in\betaspace_{T\setminus S}}\Big(\ptregcoef(S,T)  - {\alpha}\Big)\T\Sigma_{T\setminus S\given S\cap T}\Big(\ptregcoef(S,T)  - {\alpha}\Big) \,,
\end{aligned}
\end{align}
where 
$\ptregcoef(S,T):=\Sigma_{T\setminus S\given S\cap T}^{-1}\Sigma_{(T\setminus S)(S\setminus T)\given S\cap T}\beta_{S\setminus T}$ is the partial regression coefficients of $X_{\trusupp\setminus T}\T\beta_{\trusupp\setminus T}$ onto $X_{T\setminus \trusupp}$.
\end{definition}
\noindent
Although both $\signalone$ and $\signaltwo$ depend on the parameters $(\beta,\Sigma,\sigma^2)$, we omit them in the arguments for brevity.
$\signalone$ is the variance contributed by $\trusupp$ that is not captured by $T$, while $\signaltwo$ characterizes the violation of $\ptregcoef$ to the beta-min condition.
Therefore, Algorithm~\ref{alg:main:twoccase2} aims to estimate $\signalone$ and $\signaltwo$ by their sample counterparts, while \bss{} only estimates $\signalone$ and ignores the signal conveyed by $\signaltwo$ entirely.
By contrast, \klbss{} adapts to both situations where either $\signalone$ or $\signaltwo$ is larger.

The relation with the eigenvalues $\lambda_K(\Sigma)$ and $\lambda_B(\Sigma)$ that we focus in the main paper is the latter provide lower bounds on the signals and hence the sample complexities ultimately obtained in Theorem~\ref{thm:opt:eigen}:
\begin{align}\label{eq:opt:signal_eigen_lb}
\begin{aligned}
    \signalone(\trusupp,T) + \signaltwo(\trusupp,T) &\ge |\trusupp\setminus T|\times \betam^2\lambda_K(\Sigma) / \sigma^2 \\
    \signalone(\trusupp,T) & \ge |\trusupp\setminus T|\times \betam^2\lambda_B(\Sigma) / \sigma^2 \,.
\end{aligned}
\end{align}
Notice that $\signaltwo$ is large when entries in $\ptregcoef$ are close to zero and thus fall outside of $\betaspace$.
This property that \klbss{} is better at distinguishing alternatives with small regression coefficients is particularly useful in the context of SEM, where the information flows in one direction and accumulates at near-sink nodes 
(cf.~Figure~\ref{fig:demo}). 
By the definition of $\ptregcoef$, an alternative containing these nodes can lead to small regression coefficients. This materializes the phenomenon discussed at a high-level in Section~\ref{sec:intro:overview}. 

Define the (global) signal to be
\begin{align}\label{eq:main:signal}
    \signal(\mclass) & := 
    \min_{(\beta,\Sigma,\sigma^2)\in\mclass} \;\; \min_{T\in\suppspace\setminus\{\trusupp\}} 
    \frac{1}{|\trusupp\setminus T|}\Big(\signalone(\trusupp,T) \vee \signaltwo(\trusupp,T) \Big) \,.
\end{align}
For comparison, note that 
\begin{align}\label{eq:main:signal:bss}
    \signalone(\mclass) :=  \min_{(\beta,\Sigma,\sigma^2)\in\mclass} \;\; \min_{T\in\suppspace\setminus\{\trusupp\}}\frac{1}{|\trusupp\setminus T|}\signalone(\trusupp,T)    
\end{align}
is the signal for \bss{}, which is similarly defined in \citet{wainwright2009information}.

We can now state the sample complexity result for \klbss{} below, which will be the main technical machinery used for analyzing the performance of \klbss{}:

\begin{theorem}\label{thm:main:ub:simplefull}
Assume $\sps\le \dd/2$ and let $(\beta,\Sigma,\sigma^2)\in\mclass:=\mclass(\betaspace,\Sigmaspace,\sigma^2)$.
Given $n$ i.i.d. samples from $P_{\beta,\Sigma,\sigma^2}$,
if $\signal(\mclass)>0$ and the sample size satisfies
\begin{align}\label{eq:thm:main:ub:simplefull}
    \begin{aligned}
    n - \sps &\gtrsim  \max_{r\in[\sps]}\frac{\log\binom{\dd-\sps}{r} + \log(1/\delta)}{r\signal(\mclass) \wedge 1 } \\
    & \asymp \max\bigg\{ \frac{\log\big(\dd-\sps\big) + \log(1/\delta)}{\signal(\mclass)} ,\, \log \binom{\dd-\sps }{\sps} + \log(1/\delta)\bigg\} \,,
    \end{aligned}
\end{align}
then $\prob_{\beta,\Sigma,\sigma^2}(\estsupp=\trusupp)\ge 1-\delta$, where $\estsupp$ is given by Algorithm~\ref{alg:main:simple}. 
\end{theorem}
\noindent
The detailed proof
is postponed to Appendix~\ref{app:main:ub:simplefull}. 
For comparison, the sample complexity of \bss{} \citep[adapted to our setting from][]{wainwright2009information} is
\begin{align}
\label{eq:bss:constrained}
    \frac{\log(\dd-\sps)}{\signalone(\mclass)} \vee \log\binom{\dd-\sps}{\sps}\,.
\end{align}
Obviously, $\signalone(\mclass)\le \signal(\mclass)$, i.e. $\signal(\mclass)$ captures at least as much signal as \bss{}. 

\begin{remark}
\label{rem:pw:thm}
    A pointwise version of Theorem~\ref{thm:main:ub:simplefull} (for any fixed $(\beta,\Sigma,\sigma^2)\in\mclass$) also holds with $\signal(\beta,\Sigma,\sigma^2):=\min_{T\in\suppspace\setminus\{\trusupp\}} (\signalone(\trusupp,T) \vee \signaltwo(\trusupp,T) ) / |\trusupp\setminus T|$ in place of $\signal(\mclass)$, i.e. without the first minimization in~\eqref{eq:main:signal}. This is clear from the proof of Theorem~\ref{thm:main:ub:simplefull}, whose analysis is uniform for all $(\beta,\Sigma,\sigma^2)\in\mclass$.
\end{remark}

\begin{remark}
Instead of using $\ptregcoef(S, T)$, one could exploit the regression vector of the whole alternative support $T$ without the partialing out step. However, it would be less sample efficient to estimate the relatively small extra signal, which is based on the intuition that the coefficients of $X_{S\cap T}$ barely violate the beta-min condition. This leads to Vanilla \klbss{} (Section~\ref{sec:prac:mip}).
By doing so we indeed lose some signal, and we will discuss how much is lost in Appendix~\ref{sec:disc:interpret} and~\ref{app:klbssv}.
\end{remark}

We end this section by characterizing the signal to distinguish any two supports $S$ and $T$. Lemma~\ref{lem:main:signallb} below, whose proof is in Appendix~\ref{app:main:signallb}, implicitly supports the idea that it should be easier to distinguish $S$ and $T$ when the discrepancy between them is larger. It shows the signal $\signalone(S,T)\vee \signaltwo(S,T)$ is the same order as the conditional variance of a linear combination of $|S\setminus T|+|T\setminus S|$ many random variables. This validates the $|\trusupp\setminus T|$ scaling factor in the denominator in our definition of the global signal~\eqref{eq:main:signal}.
\begin{lemma}\label{lem:main:signallb}
    For any $(\beta,\Sigma,\sigma^2)\in\mclass(\betaspace,\Sigmaspace,\sigma^2)$, and any two sets $S,T\in \suppspace$,
    \begin{align*}
        \signalone(S,T)\vee \signaltwo(S,T) 
        \asymp \frac{1}{\sigma^2}\min_{\alpha_{T\setminus S}\in\betaspace_{T\setminus S}}\var\Big[X_{S\setminus T}\T\beta_{S\setminus T} - X_{T\setminus S}\T\alpha_{T\setminus S}\given S\cap T\Big] \,.
    \end{align*}
    Moreover, the constant is within $[1/2,1]$.
\end{lemma}

\subsection{Interpretation of \klbss{}}\label{sec:disc:interpret}

In this appendix, we shed light on the main ideas behind the design of \klbss{} through a KL divergence decomposition of the support recovery problem. Especially, we focus on how the score~\eqref{eq:compare:score} is constructed.
The difference between \bss{} and \klbss{} is an additional term in the score, which is a minimizer of a constrained quadratic program and characterizes the violation of the OLS regression vector to the parameter space $\betaspace$.  
The choice of this additional term is motivated by the worst case (i.e.  minimum) KL divergence between the true underlying model and its closest alternative, which also coincides with Algorithm~\ref{alg:main:twoccase2}, the main ingredient of \klbss{}.

\subsubsection{KL divergence decomposition}
Suppose $\trusupp$ is the true support and fix an alternative support $T$ distinct from $\trusupp$, both of size $\sps$ and not necessarily disjoint. 
Let the true model with $\trusupp$ be $P$ below, and consider an alternative model $P'$ with $T$  by varying the linear coefficients: 
\begin{align}
\label{eq:decomp:model1}
\begin{aligned}
    P&: Y = X_{\trusupp}\T \beta + \epsilon \\
    P'&: Y = X_T\T \alpha + \epsilon 
\end{aligned}
\end{align}
where $\beta\in\betaspace_{\trusupp}\subseteq\mathbb{R}^\sps,\alpha\in\betaspace_T\subseteq\mathbb{R}^\sps$.
Then the KL divergence between these two models decomposes as
\begin{align}
\begin{aligned}\label{eq:disc:interpret1}
    \mathbf{KL}(P\| P') & \propto \frac{1}{\sigma^2}\times \E(X_{\trusupp}\T\beta - X_T\T\alpha)^2 \\
    & = \frac{1}{\sigma^2}\times \bigg(\beta\T \Sigma_{\trusupp\given T} \beta + (\Sigma_{TT}^{-1}\Sigma_{T\trusupp}\beta  - \alpha)\T\Sigma_{TT}(\Sigma_{TT}^{-1}\Sigma_{TS}\beta  - \alpha)\bigg) \\
    & = \underbrace{\frac{1}{\sigma^2}\times \beta_{\trusupp\setminus T}\T \Sigma_{\trusupp\setminus T\given T} \beta_{\trusupp\setminus T}}_{\signalone} + \underbrace{\frac{1}{\sigma^2}\times(\regcoef  - \alpha)\T\Sigma_{TT}(\regcoef  - \alpha)}_{:=\signalltwo(\alpha)} 
\end{aligned}
\end{align}
where $\regcoef:=\Sigma_{TT}^{-1}\Sigma_{T\trusupp}\beta$.
Given $\beta$, the closest $P'$ to $P$ is parameterized by $\widetilde{\alpha}^*=\argmin_{{\alpha}\in\betaspace_T}\signalltwo({\alpha})$
and the corresponding minimum KL divergence is proportional to 
$\signalone + \signalltwo(\widetilde{\alpha}^*) = \signalone + \signalltwo$.
It is easy to see that $\signalltwo =\min_{{\alpha}\in\betaspace_T}\signalltwo({\alpha})$ is nonnegative, and is zero when $\regcoef\in\betaspace_T$. While for some $\beta$ on the boundary of $\betaspace_{\trusupp}$ and certain covariance structure, $\signalltwo$ could be positive and even significantly larger than $\signalone$. 
Since $\signalone+\signalltwo$ is the KL divergence between $P\AND P'$, one needs (e.g. by Lemma~\ref{lem:lecam}) at least $n\gtrsim 1/(\signalone+\signalltwo) $
samples to distinguish these two distributions information-theoretically. 
Thus, $\signalone+\signalltwo$ quantifies the worst-case distinguishability between the true support $\trusupp$ and the alternative $T$, meanwhile, it represents the maximum signal available to separate $\trusupp$ from $T$ in the hardest configuration. \klbss{} is designed to achieve this optimal signal, improving upon \bss{}, which leverages only $\signalone$ but ignores the additional contribution $\signalltwo$.

\subsubsection{Connection to \klbss{}} 
Based on the KL decomposition in \eqref{eq:disc:interpret1}, it may not yet be clear where \klbss{} exactly comes from, because \klbss{} leverages information in $\signaltwo$ instead of $\signalltwo$ (cf. Section~\ref{sec:disc:analysis}).
Using $\signalltwo$ leads to Vanilla \klbss{}, introduced in Section~\ref{sec:prac:mip} for computational purposes. It turns out there is a subtle interplay between the sparsity $s$ and the signal $\signalltwo$ that ever so slightly degrades the performance of Vanilla \klbss{} in a minimax sense, although on average it typically outperforms \klbss{} as demonstrated in Section~\ref{sec:expt}. We postpone further analysis of Vanilla \klbss{} to Appendix~\ref{app:klbssv}, where we will see that using $\signalltwo$ leads to an extra dependence of $n\gtrsim \sps / (\signalone\vee \signalltwo)$ in the sample complexity, which is mainly due to the error in matrix estimation.
To avoid this, \klbss{} makes a slight sacrifice on the signal by considering a finer decomposition of KL divergence. 

We still consider distinguishing $\trusupp$ from the alternative $T$, but will be specific about their intersection, i.e. we write $W=\trusupp\cap T$, $S'=\trusupp \setminus T$, $T'=T\setminus \trusupp$, $|S'|=|T'|=r$, $|W|=\sps-r$. 
Again, they specify two models with support being $S$ or $T$ by varying the linear coefficients:
\begin{align}
\label{eq:decomp:model2}
\begin{aligned}
    & P: Y = X_{S'}\T \beta + X_W\T \beta_W + \epsilon \\
    & P': Y = X_{T'}\T \alpha + X_W\T \alpha_W+ \epsilon 
\end{aligned}
\end{align}
where $\beta\in \betaspace_{S'}\subseteq \mathbb{R}^{r},\alpha\in \betaspace_{T'}\subseteq\mathbb{R}^r,\beta_W,\alpha_W\in\betaspace_W\subseteq\mathbb{R}^{\sps-r}$, for some $(\betaspace_{S'},\betaspace_{T'},\betaspace_W)$. $\beta,\alpha,\beta_W,\alpha_W$ are free parameters for $P$ and $P'$. 
Note that the vector $\alpha$ in \eqref{eq:decomp:model1} is $(\alpha,\alpha_W)$ here with a little abuse of notation.
This is simply rewriting the model \eqref{eq:decomp:model1} above; we are not introducing anything new here. Then the KL divergence between $P$ and $P'$ decomposes as 
\begin{align}
\begin{aligned}\label{eq:disc:interpret2}
    & \quad \mathbf{KL}(P\| P')\\
     & \propto \frac{1}{\sigma^2}\times \E(X_{S'}\T\beta + X_W\T \beta_W- X_{T'}\T\alpha - X_W\T \alpha_W)^2 \\
    & = \underbrace{\frac{1}{\sigma^2}\times \beta\T \Sigma_{S'\given T} \beta}_{\signalone} + \underbrace{\frac{1}{\sigma^2}\times(\ptregcoef-\alpha)\T\Sigma_{T'\given W}(\ptregcoef-\alpha)}_{:=\signaltwo(\alpha)} \\
    & \quad+ \underbrace{\frac{1}{\sigma^2}\times\big(\beta_W - \alpha_W + \Sigma_{WW}^{-1}(\Sigma_{WS'}\beta - \Sigma_{WT'}\alpha)\big)\T \Sigma_{WW} \big(\beta_W - \alpha_W + \Sigma_{WW}^{-1}(\Sigma_{WS'}\beta - \Sigma_{WT'}\alpha)\big)}_{:=\signalthree(\alpha,\alpha_W)} \,,
\end{aligned}
\end{align}
where we recall $\ptregcoef:=\Sigma_{T'\given W}^{-1}\Sigma_{T'S'\given W}\beta$ and definitions in~\eqref{eq:main:twocase:signal}. 
Note that $\signalltwo((\alpha,\alpha_W)) \equiv \signaltwo(\alpha) + \signalthree(\alpha,\alpha_W) \ge \signaltwo(\alpha) $ with corresponding definition of $\alpha$. Since $\signaltwo = \min_{{\alpha}\in\betaspace_{T'}}\signaltwo({\alpha})$ and $\signalthree(\alpha,\alpha_W)$ is nonnegative, we have $\signaltwo \le \signalltwo$.
Algorithm~\ref{alg:main:twoccase2} estimates $\signalone$ and $\signaltwo$ using their sample counterparts. 
Working with $\signaltwo$ instead of $\signalltwo$, \klbss{} sacrifices some information, but will enjoy explicit improvement in terms of sample complexity over \bss{} (Theorem~\ref{thm:main:ub:simplefull} vs. Theorem~\ref{thm:main:ub:vanilla}). 
We use Example~\ref{exmp:motiv} to show this only incurs a small sacrifice in signal below.

\subsubsection{Signal loss in Example~\ref{exmp:motiv}}
We can show we do not lose information in terms of rate by exploiting $\signaltwo$ instead of $\signalltwo$ in Example~\ref{exmp:motiv}. 
To quantify how much signal do we lose by exploiting $\signaltwo$ instead of $\signalltwo$, following two upper bounds on $\signalltwo$ using $\signaltwo$ would be useful:
\begin{align*}
    \signalltwo & \le \signaltwo  + \min_{{\alpha}_W \in \betaspace_W}\signalthree(\alpha^*,{\alpha}_W) \\
    \signalltwo & \le \min_{{\alpha}\in\betaspace_{T'}}\big[\signaltwo({\alpha}) + \signalthree({\alpha},\beta_W)\big] \,.
\end{align*}
Both bounds hold in general, rather than being limited to Example~\ref{exmp:motiv}.
\begin{proposition}\label{prop:disc:interpret:signal}
    Consider model~\eqref{eq:opt:mtvtexmp}, let $S=\trusupp$, for any $T\in\suppspace\setminus\{\trusupp\}$ with $|\trusupp\setminus T|=r$, denote
    \begin{align*}
        & \signaltwo = \min_{\alpha\in\betaspace_{T'}}\signaltwo(\alpha) \\
        & \signalltwo = \min_{\alpha\in\betaspace_{T}}\signalltwo(\alpha)\,, 
    \end{align*}
    where $\signaltwo(\alpha)$ and $\signalltwo(\alpha)$ are defined in~\eqref{eq:disc:interpret2} and~\eqref{eq:disc:interpret1}, then we have
    \begin{align*}
        \signaltwo\asymp \signalltwo \asymp r\betam^2 \,.
    \end{align*}
\end{proposition}
Thus, up to constants, we do not lose too much in this example. Figure~\ref{fig:signal_comp} numerically shows the signals $\signalone,\signaltwo,\signalltwo$ on different number of missing variables $r$ with $\sps=12, b=5,\betam=0.1$, from which we can see a significant discrepancy between $\signalone$ and $\signalltwo (\signaltwo)$, a small loss from $\signaltwo$ to $\signalltwo$, and $\signaltwo$ is tightly lower bounded by $ r\times \betam^2$.
The zig-zag shape of the curves is due to some technicalities of this particular example in the
optimization for $r$ being even or odd, but is genuine.
\begin{figure}[t]
    \centering
    \includegraphics[width=0.7\linewidth]{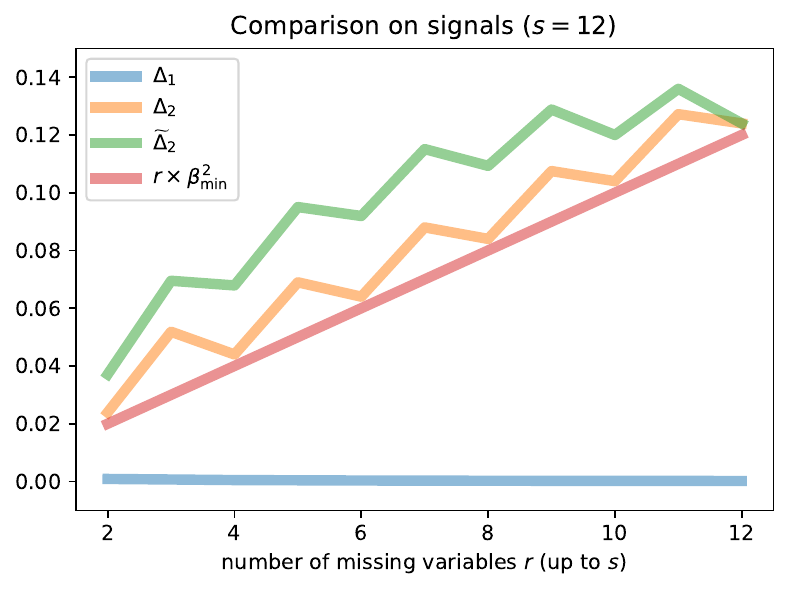}
    \caption{Signals $\signalone,\signaltwo,\signalltwo$ for fixed $\sps=12,b=5,\betam=0.1$.}
    \label{fig:signal_comp}
\end{figure}
\begin{proof}[Proof of Proposition~\ref{prop:disc:interpret:signal}]
For any other alternative $T$, without loss of generality, let $S'=\trusupp \setminus T = [r]$, $W=\trusupp\cap T = \{r+1,\ldots,\sps\}$, $T'=T\setminus \trusupp=\{\sps+1,\ldots,\sps+r\}$. Based on the calculation in Example~\ref{exmp:motiv}, we have $\beta_{W}=\betam\mathbf{1}_{\sps-r}$, $\Sigma_{WW}=I_{\sps-r}$, $\Sigma_{W S'}=0$, $\Sigma_{WT'}=b\mathbf{1}_{\sps-r}\mathbf{1}_r\T$, $\ptregcoef=\frac{rb}{1+r^2b}\betam\mathbf{1}_r$. We can upper bound $\signalltwo \le \min_{{\alpha}\in\betaspace_{T'}}\big[\signaltwo({\alpha}) + \signalthree({\alpha},\beta_{W})\big]$
where
\begin{align*}
    & \quad \signaltwo({\alpha}) + \signalthree({\alpha},\beta_{W}) \\
    & = (\ptregcoef - {\alpha})\T (I_r + rb^2 \mathbf{1}_r\mathbf{1}_r\T)(\ptregcoef - {\alpha}) 
    + \|b\mathbf{1}_{\sps-r}\mathbf{1}_r\T {\alpha}\|^2 \\
    & \le 2\bigg(\ptregcoef\T (I_r + rb^2 \mathbf{1}_r\mathbf{1}_r\T)\ptregcoef + {\alpha}\T (I_r + rb^2 \mathbf{1}_r\mathbf{1}_r\T){\alpha}\bigg) + 
    b^2(\sps-r){\alpha}\T\mathbf{1}_r\mathbf{1}_r\T{\alpha} \\
    & \le 2\bigg(\ptregcoef\T (I_r + rb^2 \mathbf{1}_r\mathbf{1}_r\T)\ptregcoef + {\alpha}\T (I_r + rb^2 \mathbf{1}_r\mathbf{1}_r\T){\alpha} + 
    b^2(\sps-r){\alpha}\T\mathbf{1}_r\mathbf{1}_r\T{\alpha} \bigg)\\ 
    & \le  2\bigg(\betam^2 \times\frac{1}{r^2b^2} \mathbf{1}_r\T (I_r + rb^2 \mathbf{1}_r\mathbf{1}_r\T)\mathbf{1}_r + {\alpha}\T (I_r + \sps b^2 \mathbf{1}_r\mathbf{1}_r\T){\alpha} \bigg)\\
    & = 2\bigg(\betam^2 \times(r + \frac{1}{rb^2}) + {\alpha}\T (I_r + \sps b^2 \mathbf{1}_r\mathbf{1}_r\T){\alpha}\bigg)\,.
\end{align*}
Consider $r\ge 2$, let 
\begin{align*}
    \alpha_0 = \begin{cases}
    \betam\times (\mathbf{1}_{r/2}\T,-\mathbf{1}_{r/2}\T)\T & r \text{ is even}\\
    \betam\times (2,-1,-1,\mathbf{1}_{(r-3)/2}\T,-\mathbf{1}_{(r-3)/2}\T)\T & r \text{ is odd}
    \end{cases}\,.
\end{align*}
Then when $b\ge 1$,
\begin{align*}
    \signalltwo & \le \min_{{\alpha}\in\betaspace_{T'}}\big[\signaltwo({\alpha}) + \signalthree({\alpha},\beta_{W})\big]  \\
    & \le \signaltwo(\alpha_0) + \signalthree(\alpha_0,\beta_{W}) \\
    & \le 2\betam^2(r + \frac{1}{rb^2} + r+3) \\
    & \le 4\betam^2(r+2)\,.
\end{align*}
Therefore, the signals are sandwiched as 
\begin{align*}
r\betam^2 \times \frac{1}{4} \le \signaltwo \le \signalltwo &\le r\betam^2\times 4(1+2/r).\qedhere
\end{align*}
\end{proof}

\subsection{Vanilla \klbss{}}\label{app:klbssv}
In this appendix, we re-visit Vanilla \klbss{}, which was introduced in Section~\ref{sec:prac:mip} to reformulate \klbss{} into an MIP for faster computation.
Vanilla \klbss{} is inspired by the KL-divergence interpretation given in Appendix~\ref{sec:disc:interpret}, and is simpler and more natural to exploit a larger signal $\signalltwo(S,T)$ (defined below) compared to $\signaltwo(S,T)$ used by \klbss{}. However, it leads to an extra factor of $s$ that \klbss{} avoids.
\begin{definition}\label{defn:main:twocase1:signal}
For any $(\beta,\Sigma,\sigma^2)\in\mclass(\betaspace,\Sigmaspace
,\sigma^2)$, and any two sets $S,T\subseteq [\dd]$, denote
\begin{align}\label{eq:main:twocase1:signal}
    \signalltwo(S,T) &:= \frac{1}{\sigma^2}\min_{{\alpha}\in\betaspace_T}\Big(\regcoef(S,T)  - {\alpha}\Big)\T\Sigma_{TT}\Big(\regcoef(S,T)  - {\alpha}\Big)  \,,
\end{align}
where $\regcoef(S,T):=\Sigma_{TT}^{-1}\Sigma_{TS}\beta_S$.
\end{definition}
\noindent
Similarly, $\regcoef(S,T)$ is the coefficient vector of regressing $X_S\T\beta_S$ onto $X_T$ and $\signalltwo(S,T)$ characterizes the violation to the constrained space $\betaspace$ for $T$ as a whole.
Following the same strategy of \klbss{}, we compare the candidate supports using residual variances plus the sample counterpart of $\signalltwo(S,T)$ (as opposed to $\signaltwo(S,T)$). For completeness, the resulting algorithm is shown in Algorithm~\ref{alg:main:twoccase1}.
\begin{algorithm}[t]
    \caption{Algorithm for two candidate case}\label{alg:main:twoccase1}
    \raggedright
    \textbf{Input:} Data matrix $X$; response $Y$; candidate supports $S,T\in\suppspace$; coefficient space $\betaspace$.  \\
    \textbf{Output:} Estimated support $\estsupp$.
    \begin{enumerate}
        \item For $R=S$ or $T$:
        \begin{enumerate}
            \item Compute $\widehat{\gamma} = (X_R\T X_R)^{-1}X_R\T Y$;
            \item Compute $\mathcal{L}(R) = \frac{\|\Pi_R^\perp Y\|^2}{n-\sps} + \min_{{\gamma}\in\betaspace_R}(\widehat{\gamma} - {\gamma})\T \frac{X_R\T X_R}{n} (\widehat{\gamma} - {\gamma})$;
        \end{enumerate}
        \item Output $\estsupp = \argmin_{R\in\{S,T\}}\mathcal{L}(R)$.
    \end{enumerate}
\end{algorithm}
Analysis of this procedure leads to the following sample complexity for distinguishing the true support $\trusupp$ against any other alternative $T$:
\begin{lemma}\label{thm:main:twoccase1}
For any $(\beta,\Sigma,\sigma^2)\in\mclass(\betaspace,\Sigmaspace
,\sigma^2)$, let $\trusupp=\supp(\beta)$ and $|\trusupp|=\sps$.
Given i.i.d. data $(X,Y)\sim P_{\beta,\Sigma,\sigma^2}$, apply Algorithm~\ref{alg:main:twoccase1} to estimate support from $\trusupp$ and $T$ with output $\estsupp$. 
Let $\signalone := \signalone(\trusupp,T)$ and $\signalltwo :=\signalltwo(\trusupp,T)$.
If sample size $n \gtrsim \sps+\frac{\sps}{\signalone\vee \signalltwo}$, we have for some constant $C_0$,
\begin{align*}
    \prob_{\beta,\Sigma,\sigma^2}(\estsupp=\trusupp) \gtrsim 1- 9\exp\bigg( -C_0(n-\sps) \min\bigg( \signalone \vee \signalltwo , 1 \bigg) +\sps \bigg)\,.
\end{align*}
\end{lemma}
Since Algorithm~\ref{alg:main:twoccase1} can be viewed as a special case of Algorithm~\ref{alg:main:twoccase2} when $S\cap T = \emptyset$, we omit the proof. 
The difference between Algorithms~\ref{alg:main:twoccase2} and~\ref{alg:main:twoccase1} is also revealed in the error probability (Lemma~\ref{thm:main:twoccase2} vs. Lemma~\ref{thm:main:twoccase1}): 
Since the calculation in the second term of $\mathcal{L}$ in Algorithm~\ref{alg:main:twoccase1} involves estimating an $\sps$-dimensional covariance matrix, Lemma~\ref{thm:main:twoccase1} has an additional dependence on $\sps$ (compared to $r=|\trusupp\setminus T|$ in Lemma~\ref{thm:main:twoccase2}), but enjoys a larger signal due to $\signaltwo(\trusupp,T)\le \signalltwo(\trusupp,T)$. This leads to a tradeoff between $s$ and $\signalltwo(\trusupp,T)$ that makes Vanilla \klbss{} slightly suboptimal in the worst-case, although our experiments show that it actually outperforms \klbss{} on average.

\begin{algorithm}[h]
    \caption{Vanilla \klbss{}}\label{alg:main:vanilla}
    \raggedright
    \textbf{Input:} Data matrix $X$; response $Y$; coefficient space $\betaspace$; sparsity $\sps$. \\
    \textbf{Output:} Estimated support $\estsupp$.
    \begin{enumerate}
        \item For $S\in\suppspace$:
        \begin{enumerate}
            \item Compute $\widehat{\gamma} = (X_S\T X_S)^{-1}X_S\T Y$;
            \item Compute $\mathcal{L}(S) := \frac{\|\Pi_S^\perp Y\|^2}{n-\sps} + \min_{{\gamma}\in\betaspace}(\widehat{\gamma} - {\gamma})\T \frac{X_S\T X_S}{n} (\widehat{\gamma} - {\gamma})$;
        \end{enumerate}
        \item Output $\estsupp = \argmin_{S\in\suppspace}\mathcal{L}(S)$.
    \end{enumerate}
\end{algorithm}

A straightforward application of Algorithm~\ref{alg:main:twoccase1} leads to Algorithm~\ref{alg:main:vanilla}, which yields Vanilla \klbss{}.
Algorithm~\ref{alg:main:vanilla} can be equivalently re-cast as the MIP in Section~\ref{sec:prac:mip}, and so everything below applies equally well to the MIP version of Vanilla \klbss{}.
Similar to \bss{}, it can be written as the following estimator: Instead of using sum of squared residual as score, Vanilla \klbss{} minimizes the score $\mathcal{L}(S)$ defined in Algorithm~\ref{alg:main:vanilla}.
\begin{align*}
    \estsupp = \argmin_{S\in\suppspace}\mathcal{L}(S)\,.
\end{align*}
Similarly, we define the uniform signal using $\signalltwo(S,T)$ instead of $\signaltwo(S,T)$ below.
For $\mclass := \mclass(\betaspace,\Sigmaspace,\sigma^2)$
define
\begin{align}\label{eq:main:signal1}
    \signalt(\mclass) & := 
    \frac{1}{\sigma^2} \min_{(\beta,\Sigma,\sigma^2)\in\mclass} \;\; \min_{T\in\suppspace\setminus\{\trusupp\}} 
    \frac{1}{|\trusupp\setminus T|}\Big(\signalone(\trusupp,T) \vee \signalltwo(\trusupp,T) \Big) \,.
\end{align}
Theorem~\ref{thm:main:ub:vanilla} below establishes the sample complexity of Vanilla \klbss{} for successful support recovery:
\begin{theorem}\label{thm:main:ub:vanilla}
Assuming $\sps\le \dd/2$, for any $(\beta,\Sigma,\sigma^2)\in\mclass:=\mclass(\betaspace,\Sigmaspace,\sigma^2)$, let $\trusupp=\supp(\beta)$ and $|\trusupp|=\sps$.
Given i.i.d. samples from $P_{\beta,\Sigma,\sigma^2}$, if the sample size satisfies
\begin{align*}
    n - \sps &\gtrsim  \max_{r\in[\sps]}\frac{\log\binom{\dd-\sps}{r} + \sps + \log(1/\delta)}{r\signalt(\mclass) \wedge 1 }  \\
    & \asymp \max\bigg\{ \frac{\log\big(\dd-\sps\big) + \sps + \log(1/\delta)}{\signalt(\mclass)} , \log \binom{\dd-\sps }{\sps} + \log(1/\delta)\bigg\} \,,
\end{align*}
then $\prob_{\beta,\Sigma,\sigma^2}(\estsupp=\trusupp)\ge 1-\delta$, where $\estsupp$ is given by Algorithm~\ref{alg:main:vanilla}.
\end{theorem}
\noindent
The proof is the same as Theorem~\ref{thm:main:ub:simplefull} by applying Lemma~\ref{thm:main:twoccase1} and is omitted. 
Compared to Theorem~\ref{thm:main:ub:simplefull} for \klbss{}, there are two differences: The signal is larger since $\signalt(\mclass)\ge \signal(\mclass)$, but there is an additional factor of $\sps$ in the numerator. Thus, Theorem~\ref{thm:main:ub:vanilla} is not necessarily an improvement due to the extra factor $\sps/\signalt(\mclass)$, although the tradeoff is minimal in light of the dominant factor of $\log\binom{\dd-\sps}{\sps}\asymp\sps\log\dd$ in the sample complexity of both methods. In Appendix~\ref{app:expt:fxv}, we illustrate this point on a concrete example.

\subsection{Unknown sparsity}\label{sec:main:unknown}

In this appendix, we provide a sample complexity bound for the modification to \klbss{} to unknown sparsity, as discussed in Section~\ref{sec:prac:unknown}.
Here we assume an upper bound $\ubsps$, but do not know $\sps$. Therefore, the candidate supports are now $\suppspaceub$. In this case, as allured in Section~\ref{sec:prac:unknown}, we modified \compare{} procedure by adding a penalty proportional to their cardinality:
\begin{align*}
\estsupp := \argmin_{D\in\{S,T\}}\bigg(\mathcal{L}(D;(S,T)) + \penalt|D|\bigg).
\end{align*}
The modified \compare{} procedure is outlined in Algorithm~\ref{alg:main:twoccase2:unknown}.
\begin{algorithm}[t]
    \caption{\compare{} algorithm under unknown sparsity setting}\label{alg:main:twoccase2:unknown}
    \raggedright
    \textbf{Input:} Data matrix $X$; response $Y$; candidate supports $S,T\in\suppspace$; coefficient space $\betaspace$; unit penalty $\penalt$. \\
    \textbf{Output:} Estimated support $\estsupp$.
    \begin{enumerate}
        \item Let $S'=S\setminus T, T'=T\setminus S, W=S\cap T$;
        \item Compute $\widetilde{X}_{S'}=\Pi_W^\perp X_{S'},\widetilde{X}_{T'}=\Pi_W^\perp X_{T'},\widetilde{Y}=\Pi_W^\perp Y$;
        \item For $R=S'$ or $T'$:
        \begin{enumerate}
            \item Compute $\widehat{\gamma} = (\widetilde{X}_R\T \widetilde{X}_R)^{-1}\widetilde{X}_R\T \widetilde{Y}$;
            \item Compute $\mathcal{L}(R\cup W;(S,T)) = \frac{\|\Pi_{R\cup W}^\perp Y\|^2}{n-|R\cup W|} + \min_{{\gamma}\in\betaspace_R}(\widehat{\gamma} - {\gamma})\T \frac{\widetilde{X}_R\T \widetilde{X}_R}{n-|W|} (\widehat{\gamma} - {\gamma})$;
        \end{enumerate}
        \item Output $\estsupp = \argmin_{D\in\{S,T\}}\bigg(\mathcal{L}(D;(S,T)) + \penalt|D|\bigg)$.
\end{enumerate}
\end{algorithm}
In Section~\ref{sec:prac:unknown}, we suggest applying $\penalt=\log n$ (BIC) and $\penalt=\log \dd$ (extended BIC) for practical use. 
Here we give a theoretical choice of $\penalt$ that enjoys a similar upper bound guarantee as Theorem~\ref{thm:main:ub:simplefull}. The basic conclusion is that $\sps$ is replaced with $\ubsps$ in Theorem~\ref{thm:main:ub:simplefull}.

We define the signal under unknown sparsity by modifying the definition of $\signal(\mclass)$ in \eqref{eq:main:signal} as follows: 
\begin{align}\label{eq:signal:unknown}
    \signalb(\mclass) & := \frac{1}{\sigma^2} \min_{(\beta,\Sigma,\sigma^2)\in\mclass} \;\; \min_{T\in\suppspaceub\setminus\{\trusupp\}}\frac{1}{|\trusupp\setminus T|}\Big(\signalone(\trusupp,T) \vee \signaltwo(\trusupp,T)\Big)
    \,.
\end{align}
The only difference between $\signal(\mclass)$ and $\signalb(\mclass)$ is that $\suppspace$ is replaced with $\suppspaceub$.
Secondly, a finer analysis of the score $\mathcal{L}(\cdot;(S,T))$ leads to Lemma~\ref{thm:main:twoccase2:unknown} in Appendix~\ref{app:main:twoccase2:unknown}, which says with high probability, 
\begin{align*}
    \mathcal{L}(T;(\trusupp,T)) - \mathcal{L}(\trusupp;(\trusupp,T))  \ge \sigma^2\bigg[\frac{1}{2}\signalone(\trusupp,T)\vee\signaltwo(\trusupp,T) - \frac{1}{4} \ell'\signalb(\mclass)\bigg]\,,
\end{align*}
where $\ell':=\max\{|T|-|\trusupp|,0\}$.
Therefore, the additive penalty term in Algorithm~\ref{alg:main:twoccase2:unknown} actually plays a role of compensating the $\frac{1}{4}\ell'\signalb(\mclass)$ term in the RHS of this lower bound, and $\ell'$ coincides with the cardinality difference between the supports, which is the reason we set the penalty scales with cardinality.
Hence, we only need to replace \compare{} algorithm in the framework of Algorithm~\ref{alg:main:simple} with Algorithm~\ref{alg:main:twoccase2:unknown} for comparison between two candidates. 
Recall the definitions of $\signalone$ and $\signaltwo$ and their relationship with $\eigvalk$ (cf.~Appendix~\ref{sec:disc:analysis} and~\eqref{eq:opt:signal_eigen_lb}), then we have the following sample complexity:
\begin{theorem}\label{thm:main:ub:simplefull:unknown}
    Assuming $\ubsps\le \dd/2$, for any $(\beta,\Sigma,\sigma^2)\in\mclass:=\mclass(\betaspace,\Sigmaspace,\sigma^2)$, let $\trusupp=\supp(\beta)$ and $|\trusupp|=\sps\le \ubsps$.
    Given $n$ i.i.d. samples from $P_{\beta,\Sigma,\sigma^2}$,  apply Algorithm~\ref{alg:main:simple} with \compare{} replaced by Algorithm~\ref{alg:main:twoccase2:unknown}, $\suppspace$ replaced by $\suppspaceub$, and choice $\penalt=\frac{1}{4}\signalb(\mclass)\times \sigma^2$.
    Let the output be $\estsupp$, if the sample size satisfies
    \begin{align*}
        n - \ubsps &\gtrsim  \max\bigg\{ \frac{\log\big(\dd\big) + \log(1/\delta)}{\signalb(\mclass)} , \log \binom{\dd}{\ubsps} + \log(1/\delta)\bigg\} \,,
    \end{align*}
    then $\prob_{\beta,\Sigma,\sigma^2}(\estsupp=\trusupp)\ge 1-\delta$.
\end{theorem}
\noindent
The proof of Theorem~\ref{thm:main:ub:simplefull:unknown} is in Appendix~\ref{app:main:ub:simplefull:unknown}.

\subsection{Theoretical choice of $\betam$}\label{sec:disc:adap}
In this appendix, we provide a theoretical procedure that tunes the parameter $\betam$ from data and achieves the same sample complexity bounds as \klbss{} in Theorem~\ref{thm:opt:eigen}.
The approach borrows the idea from Proposition~4.2 of \citet{ndaoud2020optimal}.
For example, consider the DAG model $\Sigmaspacek$, suppose our sample size satisfies the upper bound in Theorem~\ref{thm:opt:eigen}:
\begin{align*}
    n - \sps &\gtrsim   \max\bigg\{ \frac{\log\big(\dd-\sps\big) +\log(1/\delta)}{\betam^2\sigmam^2/\sigma^2} , \log \binom{\dd-\sps }{\sps} + \log(1/\delta)\bigg\} \,.
\end{align*}
Using this, it is not hard to show that the choice 
\begin{align*}
    \widetilde{\beta}_{\min}^2 \asymp \frac{\log(\dd-\sps) + \log(1/\delta)}{(n-\sps)\sigmam^2/\sigma^2} \,.
\end{align*}
ensures $\betam \ge \widetilde{\beta}_{\min}$. Thus $\betaspace_{\dd,\sps}(\widetilde{\beta}_{\min}) \supseteq \betaspace_{\dd,\sps}(\betam)$.
For the analysis of the error probability, running \klbss{} with $\betaspace_{\dd,\sps}(\widetilde{\beta}_{\min})$ instead of $\betaspace_{\dd,\sps}(\betam)$ is equivalent to having a smaller signal $\signal(\mclass) \ge \widetilde{\beta}_{\min}^2\sigmam^2/\sigma^2$ compared to the signal lower bound with the knowledge of $\betam$ (cf. Appendix~\ref{app:analysis:pointwise:eigen}). Thus, in the proof of Theorem~\ref{thm:main:ub:simplefull} in Appendix~\ref{app:main:ub:simplefull},
\begin{align*}
    \prob(\estsupp \ne \trusupp)  & \le\max_r \exp\bigg(5\log \binom{\dd-\sps}{r} -C_0(n-\sps) \min\big(r\widetilde{\beta}_{\min}^2\sigmam^2/\sigma^2 , 1 \big)\bigg) \,.
\end{align*}
For any $r\in[\sps]$, if $\min\big(r\widetilde{\beta}_{\min}^2\sigmam^2/\sigma^2 , 1 \big) = 1$, the analysis does not depend on choice of $\widetilde{\beta}_{\min}$. Otherwise, suppose $\widetilde{\beta}_{\min}^2 =\widetilde{C} \frac{\log(\dd-\sps) + \log(1/\delta)}{(n-\sps)\sigmam^2/\sigma^2}$ for large enough $\widetilde{C}\ge 10/C_0$, then
\begin{align*}
    & \exp\bigg(5\log \binom{\dd-\sps}{r} -C_0(n-\sps) r\widetilde{\beta}_{\min}^2\sigmam^2/\sigma^2 \bigg) \\
    \le & \exp\bigg(10r\log(\dd-\sps) - C_0 r\widetilde{C}(\log(\dd-\sps) + \log(1/\delta))\bigg) \\
    \le & \exp\bigg((10-C_0\widetilde{C})r\log(\dd-\sps) - C_0\widetilde{C}\log(1/\delta)\bigg) \\
    \le & \exp\bigg(-\log(1/\delta)\bigg) =\delta \,.
\end{align*}
In addition, the requirement for knowledge of $\sigmam^2$ and $\sigma^2$ can be relaxed to be estimated with sample splitting. Specifically, suppose we dataset $\mathcal{D} = \mathcal{D}_1\cup \mathcal{D}_2 \cup \mathcal{D}_3$ with evenly $3n$ many data points. 
Let $\widehat{\sigma}^2_{\min} = \min_{k}\frac{1}{n}\sum_{i\in \mathcal{D}_1} {X_{ik}}^2$ be the minimum marginal sample variance, which is consistent for this particular bipartite graph model with equal noise variance; 
and $\widehat{\sigma}^2 = \frac{1}{n}\sum_{i\in \mathcal{D}_2}(Y_i - X_i\T \widehat{\beta})^2$ where $\widehat{\beta}$ estimated using some sparse regression such that $\widehat{\sigma}^2$ is consistent. 
Then we perform \klbss{} over $\mathcal{D}_3$ to avoid dependence.

\subsection{Beyond Gaussian design}\label{sec:disc:ext}

To avoid technical complications, we have assumed Gaussian design and noise in~\eqref{eq:lm}. Under Gaussianity, the residual variance, which is the main object to deal with in the proofs, is conditionally subject to a $\chi^2$ distribution, e.g.~\eqref{eq:residvarexample}, for which we can apply concentration bounds as in Lemma~\ref{lem:chisq}.
Extended to the non-Gaussian setting, one can still derive similar results by resorting to concentration inequalities of sample (co)variance of (uncorrelated) random variables, e.g. sub-Gaussian with Bernstein type bounds.
The main modification to the setup and proof will be as follows.
Consider i.i.d. sub-Gaussian random vectors $X$ and sub-Gaussian noise variable $\epsilon$, and $X$ is independent of $\epsilon$ (cf. \eqref{eq:lm}).
Our main results rely on Lemma~\ref{thm:main:twoccase2}, which is further proved by Lemma~\ref{lem:samcovmat},~\ref{lem:main:twoccase2:lem1} and~\ref{lem:main:twoccase2:lem2}.
For Lemma~\ref{lem:samcovmat}, the same arguments apply for sub-Gaussian covariance matrix estimation.
For Lemma~\ref{lem:main:twoccase2:lem1} and~\ref{lem:main:twoccase2:lem2}, we use Hanson-Wright inequality \citep{rudelson2013hanson} for norm of projected sub-Gaussian random vectors (e.g. $\|\Pi_T^\perp \epsilon\|^2$) instead of Lemma~\ref{lem:chisq} for concentration of $\chi^2$ distribution.

\section{Proofs for optimality and SEM (Section~\ref{sec:analysis})}\label{app:opt}

\subsection{Proof of Lemma~\ref{lem:eigen:comp}}\label{app:analysis:eigen:comp}
\begin{proof}[Proof of Lemma~\ref{lem:eigen:comp}]
    The proof is given by the following chain of inequalities: by definition, For any $T\in\suppspace\setminus\{\trusupp\}$, let $S'=\trusupp\setminus T$, $T'=T\setminus \trusupp$, $W=\trusupp\cap T$, $r:=|\trusupp\setminus T|$, $u=(u_{S'},u_{T'})=(u_1,u_2)$,
    we have 
    \begin{align*}
        \min_{u\in\betaspace_{\trusupp\triangle T}} \frac{u \T \Sigma_{\trusupp\triangle T\given \trusupp\cap T} u}{\min_{|R|=r} \|u_R\|^2} & \ge \min_{u\in \mathbb{R}^{2r}} \frac{u \T \Sigma_{\trusupp\triangle T\given \trusupp\cap T} u}{\min_{|R|=r} \|u_R\|^2} \,.
    \end{align*}
    Then
    \begin{align*}
        \frac{u \T \Sigma_{S'\cup T'\given } u}{\min_{|R|=r} \|u_R\|^2} & = \frac{(u_2 - \Sigma_{T'T'\given W }^{-1}\Sigma_{T'S'\given W}u_1 )\T\Sigma_{S'\given W}(u_2 - \Sigma_{T'T'\given W }^{-1}\Sigma_{T'S'\given W}u_1 ) + u_1\T \Sigma_{S'\given T} u_1}{\min_{|R|=r} \|u_R\|^2} \\
        & \ge \frac{u_1\T \Sigma_{S'\given T} u_1}{\min_{|R|=r} \|u_R\|^2} \ge  \frac{\lambda_{\min}(\Sigma_{S'\given T})\|u_1\|^2 }{\min_{|R|=r} \|u_R\|^2} \ge \lambda_{\min}(\Sigma_{S'\given T}) \,.
    \end{align*}
    Taking minimum on both sides yields the statement.

    To see $\eigvalk(\Sigma) > \eigvalb(\Sigma)$ for any $\Sigma\in\Sigmaspace_\Delta$, by construction of $\Sigmaspace_\Delta$, we only need to show $\eigvalk(\Sigma) \ge \min_{T\in\suppspace\setminus\{\trusupp\}}\lambda_{\min}(\Sigma_{T\setminus \trusupp\given \trusupp})$. Again, for any $T\in\suppspace\setminus\{\trusupp\}$, 
    \begin{align*}
        \frac{u \T \Sigma_{S'\cup T'\given } u}{\min_{|R|=r} \|u_R\|^2} & = \frac{(u_1 - \Sigma_{S'S'\given W }^{-1}\Sigma_{S'T'\given W}u_2 )\T\Sigma_{S'\given W}(u_1 - \Sigma_{S'S'\given W }^{-1}\Sigma_{S'T'\given W}u_2 ) + u_2\T \Sigma_{T'\given \trusupp} u_2}{\min_{|R|=r} \|u_R\|^2} \\
        & \ge \frac{u_2\T \Sigma_{T'\given \trusupp} u_2}{\min_{|R|=r} \|u_R\|^2} \ge  \frac{\lambda_{\min}(\Sigma_{T'\given \trusupp})\|u_2\|^2 }{\min_{|R|=r} \|u_R\|^2} \ge \lambda_{\min}(\Sigma_{T'\given \trusupp}) \,.
    \end{align*}
    Taking minimum over all $T$ completes the proof.
\end{proof}

\subsection{Proof of Theorem~\ref{thm:pointwise:eigen}}\label{app:analysis:pointwise:eigen}
\begin{proof}[Proof of Theorem~\ref{thm:pointwise:eigen}]
    We will invoke the analysis in Appendix~\ref{sec:disc:analysis}. Especially, we show $\eigvalb$ and $\eigvalk$ provide lower bound for the signal $\signalone$ and $\signaltwo$. For any $T\in\suppspace\setminus\{\trusupp\}$, by Lemma~\ref{lem:main:signallb},
    \begin{align*}
        \signalone(\trusupp,T)\vee \signaltwo(\trusupp,T) & \asymp \frac{1}{\sigma^2}\min_{\alpha_{T\setminus \trusupp}\in\betaspace_{T\setminus \trusupp}}\var\Big[X_{\trusupp\setminus T}\T\beta_{\trusupp\setminus T} - X_{T\setminus \trusupp}\T\alpha_{T\setminus \trusupp}\given \trusupp\cap T\Big] \\
        & = \frac{1}{\sigma^2}\min_{\alpha_{T\setminus \trusupp}\in\betaspace_{T\setminus \trusupp}} \Big[ \beta_{\trusupp\setminus T}, -\alpha_{T\setminus \trusupp} \Big]\T \Sigma_{\trusupp\triangle T\given \trusupp\cap T} \Big[\beta_{\trusupp\setminus T}, -\alpha_{T\setminus \trusupp}\Big] \\
        & \ge |\trusupp\setminus T|\times \betam^2\eigvalk(\Sigma) / \sigma^2 \,.
    \end{align*}
    Similarly, 
    \begin{align*}
        \signalone(\trusupp,T) \ge |\trusupp\setminus T|\times \betam^2\eigvalb(\Sigma) / \sigma^2 \,.
    \end{align*}
    Therefore, the pointwise sample complexity is given by invoking Theorem~\ref{thm:main:ub:simplefull} and~\eqref{eq:bss:constrained}. 
\end{proof}

\subsection{Proof of Theorem~\ref{thm:opt:eigen}}\label{app:analysis:opt:eigen}
\begin{proof}[Proof of Theorem~\ref{thm:opt:eigen}]
    We prove by providing sample complexity upper and lower bounds.

    \noindent
    \textbf{Upper bounds:}
    By construction, for any $\Sigma$ coming from $\Sigmaspacek$ and $\Sigmaspaceb$, we have $\eigvalk(\Sigma) \ge c_0\sigmam^2$ and $\eigvalb(\Sigma)\ge c_0\sigmam^2$, respectively.
    By Lemma~\ref{app:analysis:eigen:comp} and 
    Theorem~\ref{thm:pointwise:eigen}, we conclude the sample complexity of \klbss{} on $\Sigmaspacek$ and $\Sigmaspaceb$, and \bss{} on $\Sigmaspaceb$ are
    \begin{align*}
        \frac{\log \dd}{\betam^2\sigmam^2/\sigma^2} \vee \log \binom{\dd-\sps}{\sps} \,.
    \end{align*}
    
    \noindent
    \textbf{Lower bounds: }
    The lower bounds of $\Sigmaspacek$ or $\Sigmaspaceb$ are based on Theorem~\ref{thm:opt:lb1} and Corollary~\ref{coro:opt:lb3} presented below, whose proofs are given later on.
    \begin{theorem}\label{thm:opt:lb1}
    Given $n$ i.i.d. samples from $P_{\beta,\Sigma,\sigma^2}$ with $(\beta,\Sigma,\sigma^2)\in \mclass:=\mclass(\betaspace,\Sigmaspace,\sigma^2)$ for $\Sigmaspace=\Sigmaspaceb$ or $\Sigmaspacek$.
    If the sample size is bounded as
    \begin{align*}
        n & \le  \frac{1-2\delta}{2}\times \frac{\log(\dd-\sps)}{\betam^2\sigmam^2/ \sigma^2} \,,
    \end{align*}
    then for any estimator $\estsupp$ for $\trusupp=\supp(\beta)$,
    \begin{align*}
        \inf_{\estsupp}\sup_{(\beta,\Sigma,\sigma^2)\in  \mclass} \prob_{\beta,\Sigma,\sigma^2}(\estsupp\ne \trusupp) \ge \delta - \frac{\log 2}{\log (\dd-1)} \,.
    \end{align*}
    \end{theorem}
    \begin{corollary}\label{coro:opt:lb3}
    Given $n$ i.i.d. samples from $P_{\beta,\Sigma,\sigma^2}$ with $(\beta,\Sigma,\sigma^2)\in\mclass:=\mclass(\betaspace,\Sigmaspace,\sigma^2)$ for $\Sigmaspace=\Sigmaspaceb$ or $\Sigmaspacek$.
    If the sample size is bounded as 
    \begin{align*}
        n\le 2(1-\delta)\times \frac{\log\binom{\dd-1}{\sps}-1}{\log(1 + \sps\betam^2\sigmam^2/\sigma^2)}\,,
    \end{align*}
    then for any estimator $\estsupp$ for $\trusupp=\supp(\beta)$,
    \begin{align*}
    \inf_{\estsupp}\sup_{\substack{(\beta,\Sigma,\sigma^2)\in  \mclass}} \prob_{\beta,\Sigma,\sigma^2}(\estsupp\ne \trusupp) \ge \delta\,.
    \end{align*}        
    \end{corollary}
    The upper and lower bounds above together conclude the optimality.
\end{proof}

\begin{remark}
\label{rem:eigval:fundamental}
    Both Theorem~\ref{thm:opt:lb1} and Corollary~\ref{coro:opt:lb3} extend beyond SEM to general designs as follows: Replace $\Sigmaspacek$ with $\Sigmaspacek'=\{\Sigma\in\mathbb{S}^{\dd}_{++}:\eigvalk(\Sigma) \ge \omega\}$ and $\Sigmaspaceb$ with $\Sigmaspaceb'=\{\Sigma\in\mathbb{S}^{\dd}_{++}:\eigvalb(\Sigma) \ge \omega\}$. Then the terms involving $\sigmam^2$ in the lower bounds can be replaced with $\omega$. For example, the lower bound in Theorem~\ref{thm:opt:lb1} becomes
    \begin{align*}
        n\asymp \frac{\log(\dd-\sps)}{\betam^2\omega/ \sigma^2} \,.
    \end{align*}
    Similarly, matching upper bounds can be derived as well by Theorem~\ref{thm:pointwise:eigen}.
    This confirms that the sample complexity of both methods genuinely depends on the eigenvalues $\eigvalk$ and $\eigvalb$ (i.e. through $\omega$) as opposed to $\sigmam^2$.
\end{remark}

\begin{proof}[Proof of Theorem~\ref{thm:opt:lb1}]
    Consider the design covariance generated by an empty graph where $X_k=\epsilon_k\sim\mathcal{N}(0,\sigmam^2)$ for all $k\in[\dd]$.
    The SEM generated in this way satisfies $\eigvalk(\Sigma) =  \eigvalb(\Sigma) = \sigmam^2$.
    We fix $\Sigma$, the empty graph $G$ and coefficient vector $\beta = \betam\mathbf{1}_\dd$, construct the ensemble solely by varying support.
    \begin{align*}
        \mathcal{S} :=  \bigg\{S : S = \{1,2,\cdots,\sps-1\}\cup \{X_\ell\}, \ell\in\{\sps,\sps+1,\ldots,\dd\}\bigg\} \,.
    \end{align*}
    Therefore, $|\mathcal{S}|= \binom{\dd- (\sps-1)}{1} = \dd-\sps+1\ge \dd-\sps$, and for any two elements $S,T$, we have form
    \begin{align*}
        & S = \{1,2,\cdots,\sps\}\cup \{j\} \\ 
        & T = \{1,2,\cdots,\sps\}\cup \{k\} \,,
    \end{align*}
    with $j\ne k$ and $j,k\in \{\sps,\sps+1,\ldots,\dd\}$.
    Denote the models determined by $S$ and $T$ to be $P_S$ and $P_T$, we now calculate the KL divergence between them:
    \begin{align*}
        \mathbf{KL}(P_S\| P_T) & = \E_{P_S} \log \frac{P_S}{P_T} \\
        & = \E_X  (X_S\T\beta_S - X_T\T\beta_T)^2 / 2\sigma^2\\ 
        & = \E_X (X_j - X_{k})^2 \betam^2/ 2\sigma^2 \\ 
        & = \E_X (\epsilon_j - \epsilon_{k})^2 \betam^2/ 2\sigma^2 \\ 
        & = \sigmam^2\betam^2/ \sigma^2\,.
    \end{align*}
    Finally, we apply Fano's inequality Corollary~\ref{coro:fano} with KL divergence upper bound $\betam^2\sigmam^2/\sigma^2$ and ensemble cardinality lower bound $\dd-\sps$, which completes the proof.
\end{proof}
\begin{proof}[Proof of Corollary~\ref{coro:opt:lb3}]
    Following the proof of Theorem~1 in \citet{wang2010informationl}, we construct a mixture of all possible supports in $\suppspace$.
    We adopt the graph and SEM in the proof of Theorem~\ref{thm:opt:lb1}, which satisfies $\eigvalk(\Sigma)= \eigvalb(\Sigma)= \sigmam^2$.
    To align the notation with \citet{wang2010informationl}, denote the data matrix $\widetilde{X}=(X_1,X_2,\ldots,X_{\dd})\in\mathbb{R}^{n\times (\dd-1)}$, $\widetilde{Y}=Y\in\mathbb{R}^n$, let $\mu(\widetilde{X})=\E[\widetilde{Y}\given \widetilde{X}]\in\mathbb{R}^n$ and
    \begin{align*}
        \Lambda(\widetilde{X}) = \E[\widetilde{Y}\widetilde{Y}\T \given \widetilde{X}] - \mu(\widetilde{X})\mu(\widetilde{X})\T \in\mathbb{R}^{n\times n}
    \end{align*}
    be the conditional mean and variance of $\widetilde{Y}$ given $\widetilde{X}$. 
    If suffices to recognize that Lemma~1 in \citet{wang2010informationl} still holds with covariance matrix $I_{\dd}$ replaced by this construction, i.e. 
    \begin{align*}
        \E_{\widetilde{X}}[ \Lambda(\widetilde{X})] = \bigg(\sigma^2 + \sps\betam^2\sigmam^2 \Big(1 - \frac{\sps}{\dd}\Big)\bigg)I_{n}\,.
    \end{align*}
    Then combining Lemma~1 with equation (17) in \citet{wang2010informationl} leads to the lower bound.
\end{proof}

\section{Proofs for examples}\label{app:general:ex}
\subsection{Proof of Example~\ref{ex:ABC}}\label{app:general:ex:ABC}

Example~\ref{ex:ABC} uses the eigenvalue condition \eqref{eq:sem:strong}, which is sufficient but not necessary. Before proving sufficiency, we remark that the condition \eqref{eq:sem:strong} can further be relaxed to 
\begin{align}
\label{eq:sem:eig}
\lambda_{\min}(D)
= \lambda_{\min}(\Sigma_{\trusupp^c\given\trusupp})
&> \min_{T\subseteq \trusupp^c, |T|=\sps}\lambda_{\min}(\Sigma_{\trusupp\given T}),
\end{align}
which can be relaxed even further to 
\begin{align}\label{eq:sem:eig2}
    \min_{T\in\suppspace\setminus\{\trusupp\}}\lambda_{\min}(\Sigma_{T\given\trusupp})
    &> \min_{T\in\suppspace\setminus\{\trusupp\}}\lambda_{\min}(\Sigma_{\trusupp\given T}),
\end{align}
which recovers the original definition of $\Sigmaspace_{\Delta}$ in \eqref{eq:opt:Omegagap:real}. Thus, we see that \eqref{eq:sem:strong} is just a special case of \eqref{eq:opt:Omegagap:real} in light of the relations
\begin{align*}
    \min_{T\in\suppspace\setminus\{\trusupp\}}\lambda_{\min}(\Sigma_{T\given\trusupp})
    \ge\lambda_{\min}(\Sigma_{\trusupp^c\given\trusupp})
    \quad\text{and}\quad
    \lambda_{\min}(\Sigma_{\trusupp})
    \ge \min_{T\in\suppspace\setminus\{\trusupp\}}\lambda_{\min}(\Sigma_{\trusupp\given T}).
\end{align*}
This is because the right hand side is essentially $\lambda_{\min}(\Sigma_{\trusupp\given T})$ that we want to upper bound in $\Sigmaspace_{\Delta}$.

We now prove the following generalized version of Example~\ref{ex:ABC}. It follows from Lemma~\ref{lem:opt:eigen2} below that the example covariance class defined by \eqref{eq:sem:strong} is contained in $\Sigmaspace_\Delta$, therefore, \klbss{} improves \bss{} in the sense that $\eigvalk(\Sigma) > \eigvalb(\Sigma)$. 
\begin{lemma}\label{lem:opt:eigen2}
    For any $\Sigma$ from
    \begin{align*}
    \bigg\{
    \begin{pmatrix}
        C & A\T \\
        A & B + AC^{-1} A\T
    \end{pmatrix} 
    \bigg|\,\,\lambda_{\min}(B) > \min_{T\subseteq \trusupp^c, |T|=\sps}\lambda_{\min}\Big(C - A_T\T ( B_{TT} + A_T C^{-1} A_T\T)^{-1} A_T\Big) 
    \bigg\} \,,
    \end{align*}
    we have $\min_{T\in\suppspace\setminus\{\trusupp\}}\lambda_{\min}(\Sigma_{\trusupp\setminus T\given T}) < \min_{T\in\suppspace\setminus\{\trusupp\}}\lambda_{\min}(\Sigma_{T\setminus \trusupp\given \trusupp})$. 
\end{lemma}
\begin{proof}[Proof of Lemma~\ref{lem:opt:eigen2}]
    For any 
    \begin{align*}
        \Sigma = 
    \begin{pmatrix}
        C & A\T \\
        A & B + AC^{-1}A\T
    \end{pmatrix}\,,
    \end{align*}
    conditioning on $X_{\trusupp}$, we have $\cov(X_{\trusupp^c}\given X_{\trusupp}) =  B$.
    For any $T\in\suppspace\setminus\{\trusupp\}$, 
    \begin{align*}
        \lambda_{\min}(\Sigma_{T\setminus \trusupp \given\trusupp}) \ge \lambda_{\min}(\Sigma_{\trusupp^c\given\trusupp}) = \lambda_{\min}(B) \,.
    \end{align*}
    We only need to upper bound $\min_{T\in\suppspace\setminus\{\trusupp\}}\lambda_{\min}(\Sigma_{\trusupp\setminus T\given T})$ by $\lambda_{\min}(B)$.
    To this end, take the $T\subseteq\trusupp^c$ with $|T|=\sps$ achieving the minimum, then $T\cap \trusupp=\emptyset$, and
    \begin{align*}
        \lambda_{\min}(\Sigma_{\trusupp\given T}) = \lambda_{\min}\Big(C - A_T\T(B_{TT}+A_TC^{-1}A_T\T)^{-1}A_T\Big) < \lambda_{\min}(B)\,.
    \end{align*}
    This completes the proof.
    
    To see how this generalizes the example in the main paper, since $A\ne 0$, choose $T$ such that $A_T\ne 0$. Therefore, 
    \begin{align*}
        \lambda_{\min}\Big(C - A_T\T(B_{TT}+A_TC^{-1}A_T\T)^{-1}A_T\Big) & = \min_{\|v\|=1} \Big[ v\T C v - v\T A_T\T(B_{TT}+A_TC^{-1}A_T\T)^{-1}A_T v \Big]\\
        & < \min_{\|v\|=1} v\T C v \\
        & = \lambda_{\min}(C) \le \lambda_{\min}(B)
    \end{align*}
    Thus, Lemma~\ref{lem:opt:eigen2} also generalizes to the case where $A=0$, which essentially requiring $\lambda_{\min}(C)<\lambda_{\min}(B)$ such that:
    \begin{align*}
        \min_{T\subseteq\trusupp^c,|T|=\sps}\lambda_{\min}(\Sigma_{\trusupp\given T}) &= \lambda_{\min}(C) < \lambda_{\min}(B) \le \min_{T}\lambda_{\min}(\Sigma_{T\setminus \trusupp \given\trusupp})\,.
        \qedhere
    \end{align*}
\end{proof}

\subsection{Proof of Example~\ref{exmp:gap}}\label{app:general:ex:gap}
We prove the following general version of the conclusion of Example~\ref{exmp:gap}. 
\begin{lemma}\label{lem:opt:eigen3}
    If $k^*\in \trusupp$ is a source node with $\sps$ children. Let $\var(X_{k^*})=\sigma_{k^*}^2$, and $T=\ch(k^*)$ be generated as
    \begin{align*}
        X_T = b X_{k^*} + B \mathbf{X}_e + \epsilon_T\,,
    \end{align*}
    where $\mathbf{X}_e$ are other ancestors of $T$, and $B\in\mathbb{R}^{s\times |\mathbf{X}_e|}, b\in\mathbb{R}^\sps$, $\cov(\mathbf{X}_e)=\Sigma_e$, $\cov(\epsilon_T) = \Sigma_T = \text{diag}\Big(\{\sigma_j^2\}_{j\in T}\Big)$. Then 
    \begin{align*}
        \lambda_{\min}(\Sigma_{\trusupp\setminus T\given T}) \le \var(X_{k^*}\given T) = \frac{\sigma^2_{k^*}}{1 + b\T (B\Sigma_eB\T + \Sigma_T )^{-1} b \sigma_{k^*}^2} \,.
    \end{align*}
    Moreover, if $\lambda_{\max}(B\Sigma_eB\T) \vee \max_{j\in T}\sigma_j^2\le M < \infty$, $\min_{j\in T}|b_j|\ge \underline{b}>0$ and $\sigma_{k^*}^2=\sigmam^2$, then
    \begin{align*}
        \lambda_{\min}(\Sigma_{\trusupp\setminus T\given T}) \le \frac{\sigmam^2}{1 + \sps \underline{b}^2 \sigmam^2 / (2M)} \,.
    \end{align*}
\end{lemma}
\begin{proof}[Proof of Lemma~\ref{lem:opt:eigen3}]
     We can compute 
     \begin{align*}
         \cov(X_T, X_{k^*}) & = b\sigma_{k^*}^2 \\
         \var(X_T) & = B\Sigma_e B\T + \Sigma_T + bb\T \sigma^2_{k^*} \,.
     \end{align*}
     Denote $A := B\Sigma_e B\T + \Sigma_T$, $\nu := b\T A^{-1} b$, then the conditional variance is
     \begin{align*}
         \var(X_{k^*}\given T) & = \sigma^2_{k^*} - \sigma^4_{k^*} b\T (A + bb\T \sigma^2_{k^*} )^{-1} b \\
         & = \sigma^2_{k^*} - \sigma^4_{k^*} b\T \bigg(A^{-1} - A^{-1}b \Big(1 / \sigma^2_{k^*} +  b\T A^{-1} b\Big)^{-1}b\T A^{-1} \bigg) b \\ 
         & = \sigma^2_{k^*} - \sigma^4_{k^*} \bigg(\nu - \nu^2 \Big[1/\sigma^2_{k^*} + \nu\Big]^{-1} \bigg) \\
         & = \frac{\sigma^2_{k^*}}{1 + \nu \sigma^2_{k^*}} \,.
     \end{align*}
     $\lambda_{\min}(\Sigma_{\trusupp\setminus T\given T}) \le \var(X_{k^*}\given T)$ is by the definition of minimum eigenvalue. If further parameters are suitably bounded, then
     \begin{align*}
         \nu = b\T A^{-1} b & \ge \|b\|^2 \lambda_{\min}(A^{-1}) \\ 
         & \ge \sps \underline{b}^2 / \lambda_{\max}(A) \\ 
         & \ge \sps \underline{b}^2 / \Big(\lambda_{\max}(\Sigma_T) + \lambda_{\max}(B\Sigma_eB\T \Big) \\ 
         & \ge \sps \underline{b}^2 / (2M)\,,
     \end{align*}
     which completes the proof.
\end{proof}

\subsection{Details of Example~\ref{ex:pathcancel}}\label{app:general:ex:pathcancel}

\begin{figure}[t]
    \centering
    \begin{tikzpicture}[scale=0.9, transform shape]
        \node[state] (x1) {$X_1$};
        \node (x2) [state, right = of x1] {$X_2$};
        \node (dot1) [right = of x2] {$\cdots$};
        \node (xsm1) [state, right = of dot1] {$X_{\sps-1}$};
        \node (xsp1) [state, fill=gray, fill opacity=0.5, text opacity=1, above = of x1] {$X_{\sps+1}$};
        \node (xsp2) [state, fill=gray, fill opacity=0.5, text opacity=1, above = of x2] {$X_{\sps+2}$};
        \node (dot2) [above = of dot1, yshift=.6cm] {$\cdots$};
        \node (x2sm1) [state, fill=gray, fill opacity=0.5, text opacity=1, above = of xsm1] {$X_{2\sps-1}$};

        \node (xs) [state, below = of x1, xshift=-1cm] {$X_\sps$};
        \node (x2s) [state, fill=gray, fill opacity=0.5, text opacity=1, above = of xsp1, xshift=-1cm] {$X_{2\sps}$};
        
        \node [state, fill=gray, fill opacity=0.1, text opacity=1, below = of dot1, xshift=-1cm, yshift=-2cm] (y) {$Y$};

        \path (x1) edge [line width=.5mm] (y);
        \path (x2) edge [line width=.5mm] (y);
        \path (xsm1) edge [line width=.5mm] (y);
        \path (xs) edge [line width=.5mm] (y);

        \path (x1) edge (xsp1);
        \path (x2) edge (xsp2);
        \path (xsm1) edge (x2sm1);

        \path (x1) edge (xs);
        \path (xsm1) edge (xs);
        \path (xsp1) edge (x2s);
        \path (x2sm1) edge (x2s);

        \draw [decorate, decoration = {brace, amplitude=10pt}, -] ([xshift=0.5cm] xsm1.north east) -- ([yshift=-2.2cm, xshift=0.5cm] xsm1.south east) node[midway, xshift=0.5cm]{\Large $\trusupp=\{X_1,X_2,\cdots,X_{\sps-1},X_{\sps}\}$};
        \draw [decorate, decoration = {brace, amplitude=10pt}, -] ([yshift=2.2cm, xshift=0.5cm] x2sm1.north east) -- ([xshift=0.5cm] x2sm1.south east) node[midway, xshift=0.5cm]{\Large $T=\{X_{\sps+1},X_{\sps+2},\cdots,X_{2\sps-1},X_{2\sps}\}$};
    \end{tikzpicture} 
    \caption{The DAG of the SEM in Example~\ref{ex:pathcancel}. The DAG has $\dd=2\sps$ nodes. The true parents (support) of $Y$ is $\trusupp=\{X_1,X_2,\ldots,X_{\sps}\}$. The alternative set of variables (shaded) is denoted as $T=\{X_{\sps+1},X_{\sps+2},\ldots,X_{2\sps}\}$. The parents of $X_\sps$ are from $\{X_1,X_2,\cdots,X_{\sps-1}\}$, and the parents of $X_{2\sps}$ are from $\{X_{\sps+1},X_{\sps+2},\cdots,X_{2\sps-1}\}$. The edges from $\trusupp$ to $Y$ are in bold.}
    \label{fig:app:ex:pathcancel}
\end{figure}

Consider the DAG in Figure~\ref{fig:app:ex:pathcancel}, the SEM is given by
\begin{align*}
    X_{j} & = \epsilon_j, \quad  X_{\sps+j} = bX_j + \epsilon_{\sps+j}, \quad \epsilon_j,\epsilon_{\sps+j}\sim \mathcal{N}(0,1) , \quad \forall j=1,2,\cdots,\sps-1  \\
    X_{\sps} & = \sum_{j\in\pa(\sps)} a X_j + \epsilon_\sps , \quad \epsilon_\sps\sim \mathcal{N}(0,1)\\
    X_{2\sps} & = \sum_{j\in\pa(\sps)} a X_{\sps+j} + \epsilon_{2\sps} ,\quad \epsilon_{2\sps}\sim \mathcal{N}(0,1)\\
    \pa(\sps) &\subseteq \{X_1,X_2,\cdots,X_{\sps-1}\}, \quad |\pa(\sps)|=k\in\{0,1,2,\cdots,\sps-1\} \,.
\end{align*}
We focus this example on an adversarial choice of $T$ that is most difficult to distinguish from $\trusupp$.
The true parents $\trusupp$ of $Y$ and the alternative set of variables $T$ are
\begin{align*}
    \trusupp=\{X_1,X_2,\ldots,X_{\sps}\}, \quad
    T=\{X_{\sps+1},X_{\sps+2},\ldots,X_{2\sps}\}.
\end{align*}
For simplicity, assume $|a|<|b|$.

The two key parameters we are interested in are 
\begin{itemize}
    \item The coefficient parameter $b$: This measures the strength of the dependence in the SEM, and also characterizes the asymmetry in the covariance.
    \item The path cancellation parameter $k$: This measures the amount of path cancellation, and also characterizes the graph density of the underlying SEM.
\end{itemize}
We start with comparing $\eigvalk(\Sigma)$ and $\eigvalb(\Sigma)$ in this SEM. To illustrate the idea that path cancellation is rare, we will also investigate the signals $\signalone$, $\signaltwo$ defined in Appendix~\ref{sec:disc:analysis} and show that path cancellation---and hence signal loss---only occurs when $a=-1$. Finally, we compare these signals in the presence of exact and near path cancellation. The main takeaway is that \klbss{} still outperforms \bss{}, even when there is path cancellation.

\paragraph{Comparison of eigenvalues}
Recall the definitions in~(\ref{eq:eigval:bss}-\ref{eq:eigval:klbss}). Since the hardest comparison is $\trusupp$ vs. $T$, it suffices to set the minimization argument $T$ in the eigenvalue definitions by $T=\{X_{\sps+1},X_{\sps+2},\ldots,X_{2\sps}\}$. 
For any vector $u,v\in \mathbb{R}^\sps$, we have
\begin{align*}
    X_{\trusupp}\T u &= u_\sps \epsilon_\sps + \sum_{j\not\in \pa(\sps)}u_j\epsilon_j + \sum_{k\in\pa(\sps)}(u_k+au_\sps)\epsilon_k \\
    X_{\trusupp}\T u + X_T\T v &= u_\sps\epsilon_\sps + v_\sps\epsilon_{2\sps} + \sum_{j\not\in\pa(\sps)}\Big\{(u_j+bv_j)\epsilon_j + v_j\epsilon_{\sps+j}\Big\} \\
    & \quad + \sum_{k\in\pa(\sps)}\Big\{\big[(u_k+au_\sps) + (v_k+av_\sps)b\big]\epsilon_k + (v_k+av_\sps)\epsilon_{\sps+k} \Big\}
\end{align*}
It is easy to see the adversarial choice of $u_k= -au_\sps$ and $v_k=-av_\sps$ for $k\in\pa(\sps)$ leads to cancellation.
Therefore, the eigenvalues can be computed as
\begin{align*}
    \eigvalb(\Sigma) &= \min_{u}\frac{\var(X_{\trusupp}\T u \given T)}{\|u\|^2} = \min_{u}\frac{u_\sps^2 + \sum_{j\not\in\pa(\sps)}u_j^2 / (1+b^2)}{u_\sps^2 + \sum_{j\not\in\pa(\sps)}u_j^2 + ka^2u_\sps^2} \asymp \frac{1 }{1 + b^2 + a^2\frac{k}{\sps-k}} \\
    \eigvalk(\Sigma) &= \min_{u,v}\frac{\var(X_{\trusupp}\T u + X_T\T v)}{\min_{|R|=\sps}\|(u,v)_R\|^2} = \min_{u,v}\frac{u_\sps^2 + v_\sps^2 + \sum_{j\not\in \pa(\sps)}v_j^2}{\min_{|R|=\sps}\|(u,v)_R\|^2} \asymp \frac{(\sps-k)}{(\sps-k)  + a^2k} = \frac{1}{1 + a^2 \frac{k}{\sps-k}}  \,.
\end{align*}
Taking $a$ as constant, $\eigvalb(\Sigma)$ depends on both $b^2$ and $k/(\sps-k)$, the latter has a range of $[0,\sps-1]$, while $\eigvalk(\Sigma)$ only depends on $k/(\sps-k)$. Thus, although path cancellation affects both $\eigvalb(\Sigma)$ and $\eigvalk(\Sigma)$,  \klbss{} still improves over \bss{} by avoiding the dependence on the parameter $b$.

\paragraph{Comparison of signals}
Since the eigenvalues consider the worst-case path cancellation, and only provide lower bounds of the signals in the analysis (cf.~\eqref{eq:opt:signal_eigen_lb}), we now directly compare the signals and show exact path cancellation is rare given a fixed $\beta$ vector.
Recall the definitions of signals used for the analysis in Appendix~\ref{sec:disc:analysis}, in particular, we compare $\signalone(\trusupp,T)$ (\bss{}) with $\signalone(\trusupp,T)+\signaltwo(\trusupp,T)$ (\klbss{}) in Definition~\ref{defn:main:twocase:signal}.

Fix the regression vector $\beta$ such that $\beta_{\trusupp} = \beta_0\mathbf{1}_\sps$, then for \bss{}:
\begin{align*}
    \signalone(\trusupp,T) & = \var(X_{\trusupp}\T \beta_{\trusupp}\given T) = \beta_\sps^2 + \sum_{j\not\in\pa(\sps)} \frac{\beta_j^2}{1+b^2} + \sum_{k\in\pa(\sps)} \frac{(\beta_\sps+a\beta_k)^2}{1+b^2} \\
    & = \beta_0^2\bigg(\frac{k(1+a)^2 + (s-k)}{1+b^2} + 1\bigg) \,.
\end{align*}
Similarly, we have for \klbss{}:
\begin{align*}
    \signalone(\trusupp,T) + \signaltwo(\trusupp,T) \ge  k(1+a)^2\beta_0^2 + (\sps-k)(\beta_0^2 + \betam^2) + (\beta_0^2 + \betam^2)\,.
\end{align*}
For ease of comparison, let $\beta_0 = \betam$, then we have
\begin{align*}
    \signalone(\trusupp,T) + \signaltwo(\trusupp,T) \ge  \beta_0^2\bigg(k(1+a)^2 + (\sps-k) + 1\bigg) \,.
\end{align*}
Unlike the eigenvalues, the coefficient parameter $a$ plays a role in controlling the cancellation of paths from $(X_1,\ldots,X_{\sps-1})$ directly to $Y$ and indirectly via $X_\sps$ to $Y$. 
When $a=-1$, we get exact cancellation, and when $a\to -1$, the term $(1+a)^2k$ diminishes, reducing the overall signals. Notice that the term $(1+a)^2k$ appears in both $\signalone(\trusupp,T)$ and $\signalone(\trusupp,T)+\signaltwo(\trusupp,T)$. Thus, the path cancellation affects both \bss{} and \klbss{}.
However, if $a$ is randomly sampled, then exact cancellation is unlikely to happen (i.e. with probability zero).

To see the effect of path cancellation concretely, we can compare the signals by examining their ratio:
\begin{align*}
    r(\trusupp,T) := \frac{\signalone(\trusupp,T)}{\signalone(\trusupp,T) +  \signaltwo(\trusupp,T) }\le \frac{\frac{1}{1+b^2} + \frac{1}{k(1+a)^2 + (\sps-k)}}{1 + \frac{1}{k(1+a)^2 + (\sps-k)}} \,.
\end{align*}
Smaller $r(\trusupp,T)$ indicates better relative performance of \klbss{} against \bss{}. We have the following cases of path cancellation, depending the behaviour of $a$:
\begin{itemize}
    \item When there is no path cancellation (i.e. $(1+a)^2\gtrsim1$): We always have $r(\trusupp,T)\to\frac{1}{1+b^2}$ as $\sps\to \infty$.
    \item When there is exact path cancellation ($a=-1$): As long as $k \ll \sps$, we still have $r(\trusupp,T)\to\frac{1}{1+b^2}$ as $\sps\to \infty$.
    \item When there is near path cancellation ($a\to-1$): As long as $\sps(1+a)^2\to \infty$ or $k \ll \sps$, we have $r(\trusupp,T)\to\frac{1}{1+b^2}$ as $\sps\to \infty$.
\end{itemize}
Thus, in the regime of growing sparsity $\sps$, \klbss{} improves upon \bss{} by a quadratic factor of the SEM coefficient $b^2$. 

To summarize this discussion: While path cancellation can potentially come into play, such cancellations rarely occur in practice. Moreover, even in the presence of path cancellation, the improvement persists under mild conditions on $k$.

\section{Proofs for \klbss{}}\label{app:main}
\subsection{Proof of Theorem~\ref{thm:main:ub:simplefull}}\label{app:main:ub:simplefull}
Before we prove the theorem, we firstly introduce the main device, which is an error probability bound for Algorithm~\ref{alg:main:twoccase2} to successfully distinguish the true support $\trusupp$ from any other candidate $T$:
\begin{lemma}\label{thm:main:twoccase2}
For any $(\beta,\Sigma,\sigma^2)\in\mclass(\betaspace,\Sigmaspace
,\sigma^2)$, let $\trusupp=\supp(\beta)$ and $|\trusupp|=\sps$.
Given $n$ i.i.d. samples from $P_{\beta,\Sigma,\sigma^2}$,
apply Algorithm~\ref{alg:main:twoccase2} to estimate support from $\trusupp$ and $T$ with $|\trusupp\setminus T|=r$ and output $\estsupp$.
Let $\signalone := \signalone(\trusupp,T)$ and $\signaltwo := \signaltwo(\trusupp,T)$.
If sample size satisfies $n \gtrsim \sps+\frac{r}{\signalone\vee \signaltwo}$, then we have for some positive constant $C_0$,
\begin{align*}
    \prob_{\beta,\Sigma,\sigma^2}(\estsupp=\trusupp) \ge 1- 9\exp\Big( -C_0(n-\sps) \min\Big( \signalone \vee \signaltwo , 1 \Big) +r \Big)\,.
\end{align*}
\end{lemma}
The proof is given in Appendix~\ref{app:main:twoccase2}. Then we are ready to prove the theorem.
\begin{proof}[Proof of Theorem~\ref{thm:main:ub:simplefull}]
    We apply Lemma~\ref{thm:main:twoccase2}, whose conditions are satisfied by the stated sample complexity,
    \begin{align*}
        \prob(\estsupp \ne \trusupp)  & \le \prob\bigg[ \bigcup_{r=1}^\sps\bigcup_{\substack{T\in\suppspace\setminus\{\trusupp\} \\ |\trusupp\setminus T|=r}}\big\{\mathcal{L}(T;(\trusupp,T)) - \mathcal{L}(\trusupp;(\trusupp,T))< 0\big\}\bigg] \\& \le \sum_{r=1}^\sps\sum_{\substack{T\in\suppspace\setminus\{\trusupp\} \\ |\trusupp\setminus T|=r}}\prob\bigg[ \mathcal{L}(T;(\trusupp,T)) - \mathcal{L}(\trusupp;(\trusupp,T))< 0\bigg] \\
        & \le  \sum_{r=1}^\sps\sum_{\substack{T\in\suppspace\setminus\{\trusupp\} \\ |\trusupp\setminus T|=r}} 9\exp\bigg( -C_0(n-\sps) \min\bigg( r\signal(\mclass) , 1 \bigg) +r \bigg)\\
        & \le \max_r \sps\times \binom{\sps}{r}\binom{\dd-\sps}{r}\times 9\exp\bigg( -C_0(n-\sps) \min\bigg( r\signal(\mclass) , 1 \bigg) +r \bigg) \,.
    \end{align*}
    We apply Lemma~\ref{thm:main:twoccase2} and definition of $\signal(\mclass)$ \eqref{eq:main:signal} for the third inequality, which relies on the scaling factor $|\trusupp\setminus T|=r$.
    Since $\sps\le \dd/2$, $\binom{\sps}{r} \le \binom{\dd-\sps}{r}$ and 
    \begin{align*}
        \log 9\sps \le \max_r \log 9\binom{\sps}{r} \le \max_r 2\log \binom{\sps}{r}\,,
    \end{align*}
    when $\sps$ is large enough. 
    Therefore, the error probability
    \begin{align*}
        \prob(\estsupp \ne \trusupp)  & \le\max_r \exp\bigg(\log 9\sps + \log \binom{\sps}{r} + \log \binom{\dd-\sps}{r}+r -C_0(n-\sps) \min\big( r\signal(\mclass) , 1 \big)\bigg)\\
        & \le\max_r \exp\bigg(5\log \binom{\dd-\sps}{r} -C_0(n-\sps) \min\big(r\signal(\mclass) , 1 \big)\bigg)\,.
    \end{align*}
    Setting the RHS to be smaller than $\delta$ for all $r$, we have desired sample complexity. 
\end{proof}

\subsection{Proof of Lemma~\ref{thm:main:twoccase2}}\label{app:main:twoccase2}
Before proving the lemma, we first introduce a technical lemma on sample covariance matrix estimation. The proof follows from standard arguments, and so we only sketch its proof here:
\begin{lemma}\label{lem:samcovmat}
If $U\in\mathbb{R}^{n\times r}$ with each entry $U_{ij}\overset{iid}{\sim}\mathcal{N}(0,1)$, $\Pi\in\mathbb{R}^{n\times n}$ is an idempotent matrix and $\Tr(\Pi)=n-k$ with $k< n$, then for $t>0$,
\begin{align*}
    \prob\bigg(\|\frac{U\T \Pi U}{n-k} - I_r\|_{\textup{op}} \ge t\bigg) \le \exp\bigg(-C(n-k)\min(t,t^2)+ r\bigg)\,,
\end{align*}
which implies $1-t\le \lambda_{\min}(\frac{U\T\Pi U}{n-k})\le \lambda_{\max}(\frac{U\T \Pi U}{n-k}) \le 1+t$ with high probability.
\end{lemma}
\begin{proof}[Proof sketch]
We need only to recognize 
\begin{align*}
    \|\frac{U\T \Pi U}{n-k} - I_r\|_{\textup{op}} & = \max_{v\in\mathbb{R}^r, \|v\|=1} |v\T (\frac{U\T \Pi U}{n-k}-I_r)v |\\
    & = \max_{v\in\mathbb{R}^r, \|v\|=1} |\frac{(vU)\T \Pi (Uv)}{n-k}-1| \\
    & = \max_{v\in\mathbb{R}^r, \|v\|=1} |\frac{\chi^2_{n-k}}{n-k}-1| \,.
\end{align*}
Then the proof follows the standard analysis for sample covariance matrix using covering and packing number arguments on the $\max_{v\in\mathbb{R}^r, \|v\|=1}$. 
\end{proof}
We also need the following lemma on the norm of projected Gaussian noise:
\begin{lemma}\label{lem:main:projnoise}
If $U\in\mathbb{R}^{n\times r}$ with each entry $U_{ij}\overset{iid}{\sim}\mathcal{N}(0,1)$, $\Pi\in\mathbb{R}^{n\times n}$ is an idempotent matrix and $\Tr(\Pi)=n-k$ with $k\le n$, $\xi\in\mathbb{R}^n$ with each entry $\xi\sim\mathcal{N}(0,\sigma^2)$, $U\indep \xi$, then with $t\in(0,1)$ and the constant $C$ in Lemma~\ref{lem:samcovmat}, assuming $n\ge k+\frac{8r(1+t)}{\delta'}$,
\begin{align*}
    \prob\bigg(\frac{\|U\T\Pi \xi\|^2}{\sigma^2(n-k)^2} \ge \delta'\bigg) \le \exp\bigg(-(n-k)\frac{\delta'}{4(1+t)} + \frac{r}{4}\bigg) + \exp(-C(n-k)t^2+r) \,.
\end{align*}
\end{lemma}
\begin{proof}
Given $U$, $U\T \Pi\xi \sim \mathcal{N}(0,\sigma^2U\T \Pi U)$, we can rewrite with $\nu\sim\mathcal{N}(0,I_r)$,
\begin{align*}
    \frac{\|U\T\Pi \xi\|^2}{\sigma^2(n-k)^2} &= \frac{1}{n-k}\nu\T \frac{U\T\Pi U}{n-k}\nu  \\
    & \le \|\frac{U\T \Pi U}{n-k}\|_{\textup{op}}\frac{\|\nu\|^2}{n-k} \le \frac{(1+t)\chi_r^2}{n-k} \,.
\end{align*}
For the second inequality, we invoke Lemma~\ref{lem:samcovmat} and corresponding error probability. Given that, we have
\begin{align*}
    \prob\bigg(\frac{\|U\T\Pi \xi\|^2}{\sigma^2(n-k)^2}\ge \delta'\bigg) & \le \prob\bigg(\chi_r^2 \ge \frac{(n-k)\delta'}{1+t}\bigg)\\
    & = \prob\bigg(\frac{\chi_r^2}{r}-1 \ge \frac{(n-k)\delta'}{r(1+t)}-1\bigg) \\
    & \le \exp\bigg(-(n-k)\frac{\delta'}{4(1+t)} + \frac{r}{4}\bigg)\,,
\end{align*}
providing $n\ge k+ \frac{8r(1+t)}{\delta'}$.
\end{proof}
Now we are ready to prove Lemma~\ref{thm:main:twoccase2}.
\begin{proof}[Proof of Lemma~\ref{thm:main:twoccase2}]
We start with some notations. Write $\trusupp=S'\cup W$ and $T=T'\cup W$ with $W=\trusupp \cap T$. Denote $\ptregcoef:=\ptregcoef(\trusupp,T)$.
Let $\epsilon_{0}\sim\mathcal{N}(0,\sigma^2\signalone)$ be the part of $Y$ that cannot be explained by $T$, i.e. $X_{S'}=\Sigma_{S'T}\Sigma_{TT}^{-1}X_T+\epsilon_{S'\given T}$, $\epsilon_{0}=\beta_{S'}\T \epsilon_{S'\given T}$. Let $\epsilon':=\epsilon_{0}+\epsilon\sim\mathcal{N}(0,\sigma^2(1+\signalone))$. Furthermore, write $\widetilde{\epsilon}=\Pi_W^\perp \epsilon$, $\widetilde{\epsilon}'=\Pi_W^\perp \epsilon'$. Therefore, we can write
\begin{align*}
    \widetilde{Y}:= \Pi_W^\perp Y & = \Pi_W^\perp X_{S'}\T\beta_{S'} + \Pi_W^\perp\epsilon = \widetilde{X}_{S'}\T\beta_{S'} + \widetilde{\epsilon}\\
    & = \Pi_W^\perp X_{T'}\T\ptregcoef + \Pi_W^\perp\epsilon' = \widetilde{X}_{T'}\T\ptregcoef + \widetilde{\epsilon}'\,. 
\end{align*}
Denote the OLS vector $\widehat{\gamma}$ for $S'$ and $T'$ to be $\widehat{\beta}$ and $\widehat{\alpha}$, then we have
\begin{align*}
    \widehat{\beta} & = \beta_{S'} + (\widetilde{X}_{S'}\T \widetilde{X}_{S'})^{-1}\widetilde{X}_{S'}\T \widetilde{\epsilon} \\
    \widehat{\alpha} & = \ptregcoef + (\widetilde{X}_{T'}\T \widetilde{X}_{T'})^{-1}\widetilde{X}_{T'}\T \widetilde{\epsilon}'\,.
\end{align*}
Let 
\begin{align*}
    \widehat{\beta}^* & = \argmin_{\widetilde{\beta}\in\betaspace_{S'}}(\widehat{\beta}-\widetilde{\beta})\T \frac{\widetilde{X}_{S'}\T \widetilde{X}_{S'}}{n-(\sps-r)}(\widehat{\beta}-\widetilde{\beta}) \\
    \widehat{\alpha}^* & = \argmin_{{\alpha}\in\betaspace_{T'}}(\widehat{\alpha}-{\alpha})\T \frac{\widetilde{X}_{T'}\T \widetilde{X}_{T'}}{n-(\sps-r)}(\widehat{\alpha}-{\alpha}) \\
    \alpha^* & = \argmin_{{\alpha}\in\betaspace_{T'}}(\ptregcoef-{\alpha})\T \Sigma_{T'\given W}(\ptregcoef-{\alpha}) \,,
\end{align*}
then we have the scores
\begin{align*}
    \mathcal{L}(\trusupp;(\trusupp,T)) & = \frac{\|\Pi_{\trusupp}^\perp \epsilon\|^2}{n-\sps} + \widehat{\signal}_2(\trusupp) \\
    \mathcal{L}(T;(\trusupp,T)) & = \frac{\|\Pi_{T}^\perp \epsilon'\|^2}{n-\sps} + \widehat{\signal}_2(T)\,,
\end{align*}
where we denote
\begin{align}
\begin{aligned}\label{eq:main:twocase2:Deltas}
    \widehat{\signal}_2(\trusupp) & = (\beta_{S'}-\widehat{\beta}^*)\T \frac{\widetilde{X}_{S'}\T\widetilde{X}_{S'}}{n-(\sps-r)}(\beta_{S'}-\widehat{\beta}^*) + \frac{\|\Pi_{\widetilde{S}'}\widetilde{\epsilon}\|^2}{n-(\sps-r)} + 2\langle\beta_{S'}-\widehat{\beta}^*,\frac{\widetilde{X}_{S'}\T\widetilde{\epsilon}}{n-(\sps-r)} \rangle\\
    \widehat{\signal}_2(T) & = (\ptregcoef-\widehat{\alpha}^*)\T \frac{\widetilde{X}_{T'}\T\widetilde{X}_{T'}}{n-(\sps-r)}(\ptregcoef-\widehat{\alpha}^*) + \frac{\|\Pi_{\widetilde{T}'}\widetilde{\epsilon}'\|^2}{n-(\sps-r)} + 2\langle\ptregcoef-\widehat{\alpha}^*,\frac{\widetilde{X}_{T'}\T\widetilde{\epsilon}'}{n-(\sps-r)} \rangle \,,
\end{aligned}
\end{align}
with $\Pi_{\widetilde{R}} = \widetilde{X}_R(\widetilde{X}_R\T \widetilde{X}_R)^{-1}\widetilde{X}_R\T$ for $R=S'$ or $T'$. Note we have $\widetilde{X}_R=\Pi_W^\perp X_R = \Pi_W^{\perp}\epsilon_{R\given W}=\Pi_W^{\perp}U_{R\given W}\Sigma_{R\given W}^{1/2}$ where $U_{R\given W}\in\mathbb{R}^{n\times r}$ and $U_{(R\given W),ij}\overset{iid}{\sim}\mathcal{N}(0,1)$. Thus given $W$, $\Tr(\Pi_{W}^\perp)=n-(\sps-r)$, invoking Lemma~\ref{lem:samcovmat}, we have for $t\in(0,1)$, with probability greater than $1- 2\exp(-C(n-(\sps-r))t^2 + r)$,
\begin{align*}
    \bigg\|\frac{\Sigma_{R\given W}^{-1/2}\widetilde{X}_R\T \widetilde{X}_R\Sigma_{R\given W}^{-1/2}}{n-(\sps-r)} - I_r\bigg\|_{\textup{op}} \le t \ \ \ \ \forall R=S',T'\,.
\end{align*}
Then the proof is based on two additional lemmas:
\begin{lemma}\label{lem:main:twoccase2:lem1}
    Providing $n\ge \sps + \frac{64r}{\signalone\vee \signaltwo}$,
    \begin{align*}
        \prob\bigg( \frac{\|\Pi_T^\perp \epsilon'\|^2-\|\Pi_{\trusupp}^\perp \epsilon\|^2}{\sigma^2(n-\sps)} \ge \frac{2}{3}\signalone - \frac{1}{4}\signalone\vee\signaltwo\bigg) \ge 1 - 5\exp\bigg(-(n-\sps)\frac{\min(\signalone\vee\signaltwo,1)}{64^2}\bigg) \,.
    \end{align*}
\end{lemma}
\begin{lemma}\label{lem:main:twoccase2:lem2}
    Providing $n\ge (\sps-r) + \frac{2C'r}{\min(\signalone\vee \signaltwo,1)}$, for some constant $C'$,
    \begin{align*}
        \prob\bigg( \frac{\widehat{\signal}_2(T) - \widehat{\signal}_2(\trusupp)}{\sigma^2} \ge \frac{2}{3}\signaltwo - \frac{1}{4}\signalone\vee\signaltwo\bigg) \ge 1 - 4\exp\bigg(-C'\Big(n-(\sps-r)\Big)\min(\signalone\vee\signaltwo,1)+r\bigg) \,.
    \end{align*}
\end{lemma}
Combining Lemma~\ref{lem:main:twoccase2:lem1} and~\ref{lem:main:twoccase2:lem2}, it suffices to have $n\gtrsim \sps+\frac{r}{\signalone\vee\signaltwo}$ to ensure the conditions are satisfied, i.e. $n\ge \sps + \frac{\max(64,2C')r}{\min(\signalone\vee\signaltwo,1)}$. Let $C_0=\min(C',1/64^2)$, with probability at least
\begin{align*}
    & 1 - 5\exp\bigg(-(n-\sps)\frac{\min(\signalone\vee\signaltwo,1)}{1024}\bigg) - 4\exp\bigg(-C'\Big(n-(\sps-r)\Big)\min(\signalone\vee\signaltwo,1)+r\bigg) \\
    \ge & 1 - 9\exp\bigg(-C_0(n-\sps))\min(\signalone\vee\signaltwo,1) + r\bigg) \,,
\end{align*}
we have 
\begin{align*}
    \mathcal{L}(T;(\trusupp,T)) - \mathcal{L}(\trusupp;(\trusupp,T)) & \ge \frac{2}{3}(\signalone+\signaltwo) - \frac{1}{2}\signalone\vee\signaltwo \\
    & \ge \frac{1}{6}\signalone\vee\signaltwo > 0 \,,
\end{align*}
which implies successful recovery $\estsupp=\trusupp$, and completes the proof.
\end{proof}
We proceed to show the proofs for Lemma~\ref{lem:main:twoccase2:lem1} and~\ref{lem:main:twoccase2:lem2}.
\begin{proof}[Proof of Lemma~\ref{lem:main:twoccase2:lem1}]
    Note that
    \begin{align}\label{eq:residvarexample}
    \begin{aligned}
        \frac{\|\Pi_T^\perp \epsilon'\|^2-\|\Pi_{\trusupp}^\perp \epsilon\|^2}{\sigma^2(n-\sps)} & = \frac{(\|\Pi_T^\perp \epsilon\|^2-\|\Pi_{\trusupp}^\perp \epsilon\|^2) + \|\Pi_{T}^\perp \epsilon_0\|^2 + 2\langle \Pi_T^\perp \epsilon_{0},\Pi_T^\perp \epsilon\rangle}{\sigma^2(n-\sps)}\\
        & =: A_1+A_2+A_3 \,.
    \end{aligned}
    \end{align}
    We will choose $C_{11},C_{12},C_{13}\in(0,1)$, 
    then we have for $A_1$,
    \begin{align*}
        \prob(A_1 \le -C_{11}\signalone\vee\signaltwo )& = \prob\bigg(\frac{\|\Pi_T^\perp \epsilon\|^2-\|\Pi_{T\cap \trusupp}^\perp \epsilon\|^2}{\sigma^2(n-\sps)} - \frac{\|\Pi_{\trusupp}^\perp \epsilon\|^2-\|\Pi_{T\cap \trusupp}^\perp \epsilon\|^2}{\sigma^2(n-\sps)}\le -C_{11}\signalone\vee\signaltwo\bigg) \\
        & \le 2\prob\bigg(|\frac{\chi^2_r}{r} -1|\ge \frac{C_{11}(n-\sps)}{2r}\signalone\vee\signaltwo\bigg) \\
        & \le 2\exp(-(n-\sps)\frac{C_{11}}{8}\signalone\vee\signaltwo) \,.
    \end{align*}
    Given $n\ge \sps + \frac{8}{C_{11}}\times\frac{r}{\signalone\vee\signaltwo}$. For $A_2$, since $C_{12}\in(0,1)$,
    \begin{align*}
        \prob(A_2\le C_{12}\signalone) & = \prob(\frac{\chi^2_{n-\sps}}{n-\sps}-1\le C_{12}-1)\\
        & \le \exp(-(n-\sps)\times \frac{(1-C_{12})^2}{16}) \,.
    \end{align*}
    For $A_3$, let $U_\epsilon=\epsilon/\sigma$, $U_{\epsilon_{0}}=\epsilon_{0}/\sqrt{\signalone\sigma^2}$,
    \begin{align*}
        \prob(A_3\le -C_{13}\signalone\vee\signaltwo) & = \prob\bigg(\frac{\|\Pi_T^\perp \frac{U_\epsilon+U_{\epsilon_0}}{\sqrt{2}}\|^2 - \|\Pi_T^\perp \frac{U_\epsilon-U_{\epsilon_0}}{\sqrt{2}}\|^2}{n-\sps} \le -C_{13}\frac{\signalone\vee\signaltwo}{\sqrt{\signalone}}\bigg) \\
        & \le 2\prob(|\frac{\chi^2_{n-\sps}}{n-\sps}-1|\ge \frac{C_{13}}{2}\frac{\signalone\vee\signaltwo}{\sqrt{\signalone}})\\
        & \le 2\prob(|\frac{\chi^2_{n-\sps}}{n-\sps}-1|\ge \frac{C_{13}}{2}\sqrt{\signalone\vee\signaltwo}) \\
        & \le 2\exp(-(n-\sps)\min(\sqrt{\signalone\vee\signaltwo},\signalone\vee\signaltwo)\times \frac{C_{13}^2}{64})\,.
    \end{align*}
    Let $C_{11}=\frac{1}{8},C_{12}=\frac{2}{3},C_{13}=\frac{1}{8}$, we conclude
    \begin{align*}
         & \quad \prob\bigg( \frac{\|\Pi_T^\perp \epsilon'\|^2-\|\Pi_{\trusupp}^\perp \epsilon\|^2}{\sigma^2(n-\sps)} \le \frac{2}{3}\signalone - \frac{1}{4}\signalone\vee\signaltwo\bigg) \\
         & \le \prob(A_1 \le -C_{11}\signalone\vee\signaltwo )+\prob(A_2 \le C_{12}\signalone)\\
         & + \prob(A_3 \le -C_{13}\signalone\vee\signaltwo )\\
         & \le 2\exp(-(n-\sps)\frac{C_{11}}{8}\signalone\vee\signaltwo)\\
         & + \exp(-(n-\sps)\times \frac{(1-C_{12})^2}{64 })\\
         & + 2\exp(-(n-\sps)\min(\sqrt{\signalone\vee\signaltwo},\signalone\vee\signaltwo)\times \frac{C_{13}^2}{16}) \\
         &\le 5\exp(-(n-\sps)\frac{\min(\signalone\vee\signaltwo,1)}{64^2}) \,.
    \end{align*}
\end{proof}
\begin{proof}[Proof of Lemma~\ref{lem:main:twoccase2:lem2}]
    Since $\widehat{\beta}^*$ is the minimizer, we have
    \begin{align*}
        \widehat{\signal}_2(\trusupp) &= (\widehat{\beta}-\widehat{\beta}^*)\T \frac{\widetilde{X}_{S'}\T\widetilde{X}_{S'}}{n-(\sps-r)}(\widehat{\beta}-\widehat{\beta}^*) \le (\widehat{\beta}-\beta_{S'})\T \frac{\widetilde{X}_{S'}\T\widetilde{X}_{S'}}{n-(\sps-r)}(\widehat{\beta}-\beta_{S'}) = \frac{\|\Pi_{\widetilde{S}'}\epsilon \|^2}{n-(\sps-r)} \,.
    \end{align*}
    On the other hand, for some $t\in(0,1)$ which will be specified later, with probability greater than $1- 2\exp(-C(n-(\sps-r))t^2 + r)$,
    \begin{align*}
        \widehat{\signal}_2(T) & = (\ptregcoef-\widehat{\alpha}^*)\T \frac{\widetilde{X}_{T'}\T\widetilde{X}_{T'}}{n-(\sps-r)}(\ptregcoef-\widehat{\alpha}^*) + \frac{\|\Pi_{\widetilde{T}'}\widetilde{\epsilon}'\|^2}{n-(\sps-r)} + 2\langle\ptregcoef-\widehat{\alpha}^*,\frac{\widetilde{X}_{T'}\T\widetilde{\epsilon}'}{n-(\sps-r)} \rangle \\
        & \ge (1-t)(\ptregcoef-\widehat{\alpha}^*)\T \Sigma_{T'\given W}(\ptregcoef-\widehat{\alpha}^*) + \frac{\|\Pi_{\widetilde{T}'}\widetilde{\epsilon}'\|^2}{n-(\sps-r)} \\
        & \qquad -2 \|\Sigma_{T'\given W}^{1/2}(\ptregcoef-\widehat{\alpha}^*)\|\|U_{T'\given W}\T\widetilde{\epsilon}'/(n-(\sps-r))\| \\
        & \ge (1-t)\signaltwo\sigma^2 + \frac{\|\Pi_{\widetilde{T}'}\widetilde{\epsilon}'\|^2}{n-(\sps-r)} -2 \|\Sigma_{T'\given W}^{1/2}(\ptregcoef-\widehat{\alpha}^*)\|\|U_{T'\given W}\T\widetilde{\epsilon}'/(n-(\sps-r))\|\,,
    \end{align*}
    where the last inequality is because $\signaltwo = \min_{{\alpha}\in\betaspace_{T'}}(\ptregcoef-{\alpha})\T \Sigma_{T'\given W}(\ptregcoef-{\alpha})$. Then
    \begin{align*}
        \frac{\widehat{\signal}_2(T) - \widehat{\signal}_2(\trusupp)}{\sigma^2} & \ge (1-t)\signaltwo + \bigg(\frac{\|\Pi_{\widetilde{T}'}\widetilde{\epsilon}'\|^2}{\sigma^2(n-(\sps-r))}-\frac{\|\Pi_{\widetilde{S}'}\widetilde{\epsilon}\|^2}{\sigma^2(n-(\sps-r))}\bigg) \\
        & \qquad  -2\frac{\|\Sigma_{T'\given W}^{1/2}(\ptregcoef-\widehat{\alpha}^*)\|\|U_{T'\given W}\T\widetilde{\epsilon}'\|}{\sigma^2(n-(\sps-r))} \\
        & :=  (1-t)\signaltwo + B_1 - B_2 \,.
    \end{align*}
    We will choose $C_{21},C_{22} \in(0,1)$. For $B_1$, since $\Tr(\Pi_{\widetilde{T}'})=\Tr(\Pi_{\widetilde{S}'})=r$, let $K,K'\sim\chi^2_{r}$, conditioned on $\trusupp$ and $T$,
    \begin{align*}
        \prob(B_1 \le - C_{21}\signalone\vee\signaltwo) & = \prob\bigg(\frac{(1+\signalone)K - K'}{n-(\sps-r)} \le - C_{21}\signalone\vee\signaltwo\bigg) \\
        & \le \prob\bigg(\frac{K - K'}{n-(\sps-r)} \le - C_{21}\signalone\vee\signaltwo\bigg) \\
        & \le 2\prob\bigg(|\frac{\chi^2_r}{r}-1| \le  \frac{C_{21}(n-(\sps-r))}{2r}\signalone\vee\signaltwo\bigg) \\
        & \le 2\exp(-(n-(\sps-r))\frac{C_{21}}{8}\signalone\vee\signaltwo) \,,
    \end{align*}
    given $n\ge (\sps-r) + \frac{8}{C_{21}}\times\frac{r}{\signalone\vee\signaltwo}$. For $B_2$, we invoke Lemma~\ref{lem:main:projnoise} for $\|U_{T'\given W}\T \widetilde{\epsilon}'\|$:
    \begin{align*}
        \prob\bigg(\frac{\|U_{T'\given W}\T \widetilde{\epsilon}'\|^2}{\sigma^2(n-(\sps-r))^2} \ge C_{22}\signalone\vee\signaltwo\bigg) & \le \exp\bigg(-(n-(\sps-r))\frac{C_{22}\signalone\vee\signaltwo}{4(1+t)(1+\signalone)}+ \frac{r}{4}\bigg)\\
        & + \exp(-C(n-(\sps-r))t^2+r) \,,
    \end{align*}
    providing $n\ge (\sps-r) + \frac{8r(1+t)(1+\signalone)}{C_{22}\signalone\vee\signaltwo}$. We further discuss $\frac{1+\signalone}{\signalone\vee\signaltwo}$ in cases:
    \begin{itemize}
        \item If $\signalone<1$: $\frac{1+\signalone}{\signalone\vee\signaltwo}< \frac{2}{\signalone\vee\signaltwo}$;
        \item If $\signalone\ge 1$: $\frac{1+\signalone}{\signalone\vee\signaltwo}\le \frac{2\signalone}{\signalone\vee\signaltwo}\le 2$.
    \end{itemize}
    Thus we only need $n\ge (\sps-r) + \frac{16r(1+t)}{C_{22}}(1\vee \frac{1}{\signalone\vee\signaltwo})$ to ensure
    \begin{align*}
        \prob\bigg(\frac{\|U_{T'\given W}\T \widetilde{\epsilon}'\|^2}{\sigma^2(n-\sps)^2} \ge C_{22}\signalone\vee\signaltwo\bigg) & \le \exp\bigg(-(n-(\sps-r))\frac{C_{22}\min(\signalone\vee\signaltwo,1)}{8(1+t)}+ \frac{r}{4}\bigg)\\
        & + \exp(-C(n-(\sps-r))t^2+r)         \,.
    \end{align*}
    Then for $\|\Sigma_{T'\given W}^{1/2}(\ptregcoef-\widehat{\alpha}^*)\|$, we start with the fact that $\widehat{\alpha}^*$ is the minimizer of the programming, we have
    \begin{align*}
        \frac{1}{\sigma^2}(\widehat{\alpha}-\widehat{\alpha}^*)\T \frac{\widetilde{X}_{T'}\T \widetilde{X}_{T'}}{n-(\sps-r)}(\widehat{\alpha}-\widehat{\alpha}^*) \le \frac{1}{\sigma^2}(\widehat{\alpha}-{\alpha}^*)\T \frac{\widetilde{X}_{T'}\T \widetilde{X}_{T'}}{n-(\sps-r)}(\widehat{\alpha}-{\alpha}^*)\,.
    \end{align*}
    Expand both sides,
    \begin{align*}
        & \frac{1}{\sigma^2}(\ptregcoef-\widehat{\alpha}^*)\T \frac{\widetilde{X}_{T'}\T \widetilde{X}_{T'}}{n-(\sps-r)}(\ptregcoef-\widehat{\alpha}^*) + \frac{2}{\sigma^2}\langle\Sigma_{T'\given W}^{1/2}(\ptregcoef-\widehat{\alpha}^*), U_{T'\given W}\T\widetilde{\epsilon}' / (n-(\sps-r))\rangle \\
        \le & \frac{1}{\sigma^2}(\ptregcoef-{\alpha}^*)\T \frac{\widetilde{X}_{T'}\T \widetilde{X}_{T'}}{n-(\sps-r)}(\ptregcoef-{\alpha}^*) + \frac{2}{\sigma^2}\langle\Sigma_{T'\given W}^{1/2}(\ptregcoef-{\alpha}^*), U_{T'\given W}\T\widetilde{\epsilon}' / (n-(\sps-r))\rangle \\
        \le & (1+t)\signaltwo + \frac{2}{\sigma^2}\|\Sigma_{T'\given W}^{1/2}(\ptregcoef-{\alpha}^*)\|\|U_{T'\given W}\T\widetilde{\epsilon}' / (n-(\sps-r))\|\\
        \le & (1+t)\signaltwo + 2\sqrt{\signaltwo}\sqrt{C_{22}\signalone\vee\signaltwo}\,.
    \end{align*}
    Moreover, 
    \begin{align*}
        \frac{1-t}{\sigma^2}\|\Sigma_{T'\given W}^{1/2}(\ptregcoef-\widehat{\alpha}^*)\|^2 & \le \frac{1}{\sigma^2}(\ptregcoef-\widehat{\alpha}^*)\T \frac{\widetilde{X}_{T'}\T \widetilde{X}_{T'}}{n-(\sps-r)}(\ptregcoef-\widehat{\alpha}^*) \\
        & \le (1+t)\signaltwo + 2\sqrt{\signaltwo}\sqrt{C_{22}\signalone\vee\signaltwo} \\
        & + \frac{2}{\sigma^2}\|\Sigma_{T'\given W}^{1/2}(\ptregcoef-\widehat{\alpha}^*)\|\|U_{T'\given W}\T\widetilde{\epsilon}' / (n-(\sps-r))\| \,.
    \end{align*}
    Denote our target $x:=\|\Sigma_{T'\given W}^{1/2}(\ptregcoef-\widehat{\alpha}^*)/\sigma\|$, then
    \begin{align*}
        x^2 \le \frac{1+t}{1-t}\signaltwo + \frac{2}{1-t}\sqrt{\signaltwo}\sqrt{C_{22}\signalone\vee\signaltwo} + \frac{2}{1-t}\sqrt{C_{22}\signalone\vee\signaltwo}x \,.
    \end{align*}
    After rearrangement, since both $t\AND C_{22}<1$,
    \begin{align*}
        \bigg(x - \frac{1}{1-t}\sqrt{C_{22}\signalone\vee\signaltwo}\bigg)^2 & \le \frac{1+t}{1-t}\signaltwo + \frac{2}{1-t}C_{22}^{1/2}\signalone\vee\signaltwo + \frac{C_{22}}{(1-t)^2}\signalone\vee\signaltwo \\
        & \le \frac{3\times 2}{(1-t)^2}\signalone\vee\signaltwo\\
        \implies x& \le (\frac{\sqrt{6}  }{1-t} + \frac{1}{1-t}\sqrt{C_{22}})\sqrt{\signalone\vee\signaltwo}\\
        & \le \frac{4}{1-t}\sqrt{\signalone\vee\signaltwo}\,.
    \end{align*}
    Therefore,
    \begin{align*}
        B_2 \le 2\times \sqrt{C_{22}\signalone\vee\signaltwo} \times \frac{4}{1-t}\sqrt{\signalone\vee\signaltwo} = \frac{8}{1-t}\sqrt{C_{22}}\signalone\vee\signaltwo\,,
    \end{align*}
    with probability at least
    \begin{align*}
        1 -\exp\bigg(-(n-(\sps-r))\frac{C_{22}\min(\signalone\vee\signaltwo,1)}{8(1+t)}+\frac{r}{4}\bigg)
        - \exp(-C(n-(\sps-r))t^2+r)\,.
    \end{align*}
    Furthermore, we have
    \begin{align*}
         \frac{\widehat{\signal}_2(T) - \widehat{\signal}_2(\trusupp)}{\sigma^2} & \ge (1-t)\signaltwo - (C_{21}+ \frac{8\sqrt{C_{22}}}{1-t})\signalone\vee\signaltwo\,,
    \end{align*}
    with probability at least
    \begin{align*}
         1 & - 2\exp(-(n-(\sps-r))\frac{C_{21}}{8}\signalone\vee\signaltwo) \\
        & - \exp\bigg(-(n-(\sps-r))\frac{C_{22}\min(\signalone\vee\signaltwo,1)}{8(1+t)}+\frac{r}{4}\bigg)\\
        &- \exp(-C(n-(\sps-r))t^2+r)\,.
    \end{align*}
    Let $t = \frac{1}{3},C_{21}=\frac{1}{8},C_{22}=\frac{1}{96^2}$, we get 
    \begin{align*}
        & \prob\bigg( \frac{\widehat{\signal}_2(T) - \widehat{\signal}_2(\trusupp)}{\sigma^2} \ge \frac{2}{3}\signaltwo - \frac{1}{4}\signalone\vee\signaltwo\bigg)   \\
        \ge&  1-4\exp\bigg(-(n-(\sps-r))\min(\signalone\vee\signaltwo)\times \min(\frac{C}{9}, \frac{1}{96^2\times 32/3}) + r\bigg)\\
         =&  1-4\exp\bigg(-(n-(\sps-r))C'\min(\signalone\vee\signaltwo) + r\bigg)\,,
    \end{align*}
    where $C'=\min(\frac{C}{9}, \frac{1}{96^2\times 32/3})$.
\end{proof}

\subsection{Proof of Theorem~\ref{thm:main:ub:simplefull:unknown}}\label{app:main:ub:simplefull:unknown}
\begin{proof}[Proof of Theorem~\ref{thm:main:ub:simplefull:unknown}]
The proof is based on the following error probability bound of Algorithm~\ref{alg:main:twoccase2:unknown}, which is proved in Appendix~\ref{app:main:twoccase2:unknown}.
\begin{lemma}\label{thm:main:twoccase2:unknown}
    For any $(\beta,\Sigma,\sigma^2)\in\mclass(\betaspace,\Sigmaspace,\sigma^2)$, let $\trusupp=\supp(\beta)$ and $|\trusupp|\le \ubsps$.
    Given $n$ i.i.d. samples from $P_{\beta,\Sigma,\sigma^2}$ with $|\trusupp|\le \ubsps$,  apply Algorithm~\ref{alg:main:twoccase2:unknown} on $(\trusupp,T)$ with output $\estsupp$. 
    Let $\ell':=\max\{|T|-|\trusupp|,0\}$, use the shorthand notation $\signalone := \signalone(\trusupp,T)$, $\signaltwo := \signaltwo(\trusupp,T)$, and $\mclass:= \mclass(\betaspace,\Sigmaspace,\sigma^2)$,
    if sample size $n \gtrsim \ubsps+\frac{1}{\signalb(\mclass)}$, then we have for some constant $C_0$,
    \begin{align*}
    & \prob_{\beta,\Sigma,\sigma^2}\bigg(\frac{\mathcal{L}(T;(\trusupp,T)) - \mathcal{L}(\trusupp;(\trusupp,T))}{\sigma^2}  \ge \frac{3}{4}(\signalone+\signaltwo) - \frac{1}{4}\Big(\signalone\vee\signaltwo + \ell'\signalb(\mclass)\Big)\bigg) \\
    \ge & 1 - 8\exp\bigg(-C_0(n-\ubsps)\min\Big(1, \signalone\vee\signaltwo + \ell' \signalb(\mclass)\Big) + |\trusupp\setminus T| + |T\setminus \trusupp|\bigg)\,.
    \end{align*}
\end{lemma}
Use the shorthand notation $\signalb:= \signalb(\mclass)$.
Denote the event that $\trusupp$ beats an alternative $T$ with $|T|=j$:
\begin{align*}
    \mathcal{E}(T,j) = \bigg\{ \mathcal{L}(\trusupp;(\trusupp,T)) + \frac{\sps}{4}\signalb\sigma^2 \le \mathcal{L}(T;(\trusupp,T))  + \frac{j}{4}\signalb\sigma^2\bigg\}\,,
\end{align*}
then the estimator succeeds with 
\begin{align*}
    \prob(\estsupp=\trusupp) = \prob\bigg( 
    \bigcap_{j\in[\ubsps]}\bigcap_{T\in\supps_{\dd,j}\setminus \{\trusupp\}} \mathcal{E}(T,j) \bigg) \,.
\end{align*}
Therefore, let $\ell:=|j-\sps|$,
\begin{align*}
    \prob(\estsupp\ne \trusupp) & = \prob\bigg( \bigcup_{j\in[\ubsps]}\bigcup_{T\in\supps_{\dd,j}\setminus \{\trusupp\}} \overline{\mathcal{E}(T,j)} \bigg) \\
    & \le \sum_{T\in\supps_{\dd,\sps}\setminus\{\trusupp\}}\prob(\overline{\mathcal{E}(T,\sps)}) \\
    & + \sum_{\ell=1}^{\sps}\sum_{T\in\supps_{\dd,\sps-\ell}}\prob(\overline{\mathcal{E}(T,\sps-\ell)})  + \sum_{\ell=1}^{\ubsps-\sps}\sum_{T\in\supps_{\dd,\sps+\ell}}\prob(\overline{\mathcal{E}(T,\sps+\ell)}) \,.
\end{align*}
The first term is controlled by Theorem~\ref{thm:main:ub:simplefull}, now let's look at remaining two. Let $k:= |T\cap \trusupp|$,
\begin{align*}
    A_1+A_2 :=&  \sum_{\ell=1}^{\sps} \sum_{k=0}^{\sps-\ell} \sum_{\overset{T\in\supps_{\dd,\sps-\ell}}{|T\cap \trusupp| = k} }\prob(\overline{\mathcal{E}(T,\sps-\ell)}) +   \sum_{\ell=1}^{\ubsps-\sps}\sum_{k=0}^{\sps}\sum_{\overset{T\in\supps_{\dd,\sps+\ell}}{|T\cap \trusupp| = k}}\prob(\overline{\mathcal{E}(T,\sps+\ell)}) \,.
\end{align*}
The cardinality of the innermost sums of $A_1 \AND A_2$ are bounded by $\binom{\dd-\sps}{\sps-k}^2$ and $\binom{\dd-\sps}{\sps-k + \ell}^2$ respectively. Now we analyze the error probability respectively using Lemma~\ref{thm:main:twoccase2:unknown}.

For $|T|=\sps-\ell$, i.e. $|\trusupp|>|T|$, and $|T\cap \trusupp|=k$, we have $|\trusupp\setminus T|=\sps-k \ge \ell$, $|T\setminus \trusupp| =\sps-\ell - k$, $\ell':=\max\{j-\sps,0\} = 0$.
Note that the event
\begin{align*}
     & \frac{\mathcal{L}(T;(\trusupp,T)) - \mathcal{L}(\trusupp;(\trusupp,T))}{\sigma^2}  \ge \frac{3}{4}(\signalone+\signaltwo) - \frac{1}{4}(\signalone \vee \signaltwo + \ell'\signalb) \\
     \implies &\frac{ \mathcal{L}(T;(\trusupp,T)) - \mathcal{L}(\trusupp;(\trusupp,T))}{\sigma^2} - \frac{1}{4}\ell\signalb 
     \ge \frac{1}{2}\signalone\vee \signaltwo - \frac{1}{4}\ell\signalb \\
     \ge &\frac{1}{4}\signalone\vee \signaltwo  + \frac{1}{4}(\sps-k-\ell)\signalb > 0 \\
     \implies & \mathcal{E}(T,\sps-\ell)\,,
\end{align*}
by definitions of $\signalb$.
Therefore,
\begin{align*}
    \prob(\overline{\mathcal{E}(T,\sps-\ell)}) & \le 8\exp\bigg(-C_0(n-\ubsps)\min(1, \signalone\vee\signaltwo ) + |\trusupp\setminus T| + |T\setminus \trusupp|\bigg) \\
    & \le 8\exp\bigg(-C_0(n-\ubsps)\min(1, (\sps-k)\signalb ) + 2(\sps-k)\bigg) \,.
\end{align*}
For $|T|=\sps+\ell$, i.e. $|\trusupp|<|T|$, and $|T\cap \trusupp|=k$, we have $|\trusupp\setminus T|=\sps-k $, $|T\setminus \trusupp| =\sps+\ell - k$, $\ell'=\ell$.
The event
\begin{align*}
     & \frac{\mathcal{L}(T;(\trusupp,T)) - \mathcal{L}(\trusupp;(\trusupp,T))}{\sigma^2}  \ge \frac{3}{4}(\signalone+\signaltwo) - \frac{1}{4}(\signalone \vee \signaltwo + \ell\signalb) \\
     \implies & \frac{\mathcal{L}(T;(\trusupp,T)) - \mathcal{L}(\trusupp;(\trusupp,T))}{\sigma^2} + \frac{1}{4}\ell\signalb 
     \ge \frac{1}{2}\signalone\vee \signaltwo - \frac{1}{4}\ell\signalb + \frac{1}{4}\ell\signalb \\
     = &\frac{1}{2}\signalone\vee \signaltwo  > 0 \\
     \implies & \mathcal{E}(T,\sps+\ell)\,,
\end{align*}
Therefore,
\begin{align*}
    \prob(\overline{\mathcal{E}(T,\sps+\ell)}) & \le 8\exp\bigg(-C_0(n-\ubsps)\min(1, \signalone\vee\signaltwo +\ell\signalb ) + |\trusupp\setminus T| + |T\setminus \trusupp|\bigg) \\
    & \le 8\exp\bigg(-C_0(n-\ubsps)\min(1, (\sps-k+\ell)\signalb ) + 2(\sps-k+\ell)\bigg) \,.
\end{align*}
Thus, for $A_1$, let $t:=\sps-k \in [\sps]$,
\begin{align*}
    A_1 & = \sum_{\ell=1}^{\sps} \sum_{k=0}^{\sps-\ell} \sum_{\overset{T\in\supps_{\dd,\sps-\ell}}{|T\cap \trusupp| = k} }\prob(\overline{\mathcal{E}(T,\sps-\ell)}) \\
    & \le \sps \ubsps\max_{\overset{1\le \ell\le \sps}{0\le k \le \sps-\ell}} 8\exp\bigg(-(n-\ubsps)C_0\min\Big((\sps-k)\signalb,1\Big) + 2(\sps-k) + 2\log \binom{\dd-\sps}{\sps-k}\bigg) \\
    & \le \sps \ubsps\max_{t\in[\sps]} 8\exp\bigg(-(n-\ubsps)C_0\min\Big(t\signalb,1\Big)+ 4\log \binom{\dd-\sps}{t}\bigg) \,.
\end{align*}
For $A_2$, which is positive only when $\sps<\ubsps$, let $t := \sps-k+\ell \in [\ubsps]$,
\begin{align*}
    A_2 &= \sum_{\ell=1}^{\ubsps-\sps}\sum_{k=0}^{\sps}\sum_{\overset{T\in\supps_{\dd,\sps+\ell}}{|T\cap \trusupp| = k}}\prob(\overline{\mathcal{E}(T,\sps+\ell)}) \\
    & \le (\ubsps-\sps)\ubsps\max_{\overset{1\le \ell\le \ubsps-\sps}{0\le k \le \sps}} 8\exp\bigg(-(n-\ubsps)C_0\min\Big((\sps-k+\ell)\signalb,1\Big) + 2(\sps-k+\ell) + 2\log \binom{\dd-\sps}{\sps-k+\ell}\bigg) \\
    &\le (\ubsps-\sps)\ubsps\max_{t\in [\ubsps]} 8\exp\bigg(-(n-\ubsps)C_0\min\Big(t\signalb,1\Big) + 4\log \binom{\dd-\sps}{t}\bigg) \,.
\end{align*}
Therefore,
\begin{align*}
    A_1+A_2 & \le 8\ubsps^2 \max_{t\in [\ubsps]} \exp\bigg(-(n-\ubsps)C_0\min\Big(t\signalb,1\Big)+ 4\log \binom{\dd-\sps}{t}\bigg) \\
    & = \max_{t\in [\ubsps]} \exp\bigg(-(n-\ubsps)C_0\min\Big(t\signalb,1\Big) + 4\log \binom{\dd-\sps}{t} + \log(8\ubsps^2)\bigg) \,.
\end{align*}
Since for large enough $\ubsps$,
\begin{align*}
    \log (8\ubsps^2) &= \log 8 + 2\log \ubsps \\
    & \le \log 8 + 2\max_{t\in[\ubsps]}\log \binom{\ubsps}{t} \\
    & \le 3\max_{t\in[\ubsps]}\log \binom{\ubsps}{t} \\
    & \le 3\max_{t\in[\ubsps]}\log \binom{\dd-\sps}{t} \,,
\end{align*}
we have
\begin{align*}
    A_1+A_2 & \le \max_{t\in [\ubsps]} \exp\bigg(-(n-\ubsps)C_0\min\Big(t\signalb,1\Big) + 7 \log \binom{\dd-\sps}{t}\bigg) \,.
\end{align*}
Combined with Theorem~\ref{thm:main:ub:simplefull}, we have following error probability,
\begin{align*}
    \prob(\estsupp\ne \trusupp) & \le 2\max_{t\in [\ubsps]} \exp\bigg(-(n-\ubsps)C_0\min\Big(t\signalb,1\Big) + 7\log \binom{\dd-\sps}{t}\bigg) \\
    & \le 2\max_{t\in [\ubsps]} \exp\bigg(-(n-\ubsps)C_0\min\Big(t\signalb,1\Big) +7 \log \binom{\dd}{t}\bigg)\,.
\end{align*}
Setting the RHS to be smaller than $\delta$ leads to desired sample complexity.
\end{proof}

\subsection{Proof of Lemma~\ref{thm:main:twoccase2:unknown}}\label{app:main:twoccase2:unknown}
\begin{proof}[Proof of Lemma~\ref{thm:main:twoccase2:unknown}]
    The proof is similar with the one for Lemma~\ref{thm:main:twoccase2}.
    We adopt the same notation as the proof for Lemma~\ref{thm:main:twoccase2} in Appendix~\ref{app:main:twoccase2}. 
    Let $\trusupp=S'\cup W$, $T=T'\cup W$, $\beta_{\trusupp} = (\beta_{\trusupp\setminus T},\beta_{\trusupp\cap T})=(\beta_{S'},\beta_W)$. Thus, $\trusupp\setminus T=S'$, $T\setminus \trusupp= T'$, $\trusupp\cap T=W$. 
    Furthermore, denote $|T|=j$, then $\ell'=\max\{(j-\sps),0\}$. 
    Let $X_{S'}=\Sigma_{S'T}\Sigma_{TT}^{-1}X_T+\epsilon_{S'\given T}$, $\epsilon_0=\beta_{S'}\T \epsilon_{S'\given T}\sim\mathcal{N}(0,\sigma^2\signalone)$, $\epsilon'=\epsilon_0+\epsilon\sim\mathcal{N}(0,\sigma^2(1+\signalone))$, and $(\widetilde{X}_{S'},\widetilde{X}_{T'},\widetilde{Y},\widetilde{\epsilon},\widetilde{\epsilon}_0,\widetilde{\epsilon}') = \Pi^\perp_W(X_{S'},X_{T'},Y,\epsilon,\epsilon_0,\epsilon')$. Recall that
    \begin{align*}
    \mathcal{L}(\trusupp;(\trusupp,T)) & = \frac{\|\Pi_{\trusupp}^\perp \epsilon\|^2}{n-\sps} + \widehat{\signal}_2(\trusupp) \\
    \mathcal{L}(T;(\trusupp,T)) & = \frac{\|\Pi_{T}^\perp \epsilon'\|^2}{n-j} + \widehat{\signal}_2(T)\,,
    \end{align*}
    where $\widehat{\signal}_2(\trusupp),\widehat{\signal}_2(T)$ are defined as in~\eqref{eq:main:twocase2:Deltas} with denominator replaced by $n-|W|$.
    Invoking Lemma~\ref{lem:samcovmat}, we have the same conclusion for either $R=S'$ or $T'$, with probability greater than $1- 2\exp(-C(n-|W|)t^2 + |S'|+|T'|)$,
    \begin{align*}
        \bigg\|\frac{\Sigma_{R\given W}^{-1/2}\widetilde{X}_R\T \widetilde{X}_R\Sigma_{R\given W}^{-1/2}}{n-|W|} - I_{|R|}\bigg\|_{\textup{op}} \le t \ \ \ \ \forall R=S',T'\,.
    \end{align*}
    Write $\signalb:=\signalb(\mclass)$. Then the proof is based on two lemma:
\begin{lemma}\label{lem:main:twoccase2:unknown:lem1}
        Providing $n\ge \ubsps + \frac{192}{\signalb}$,
        \begin{align*}
            & \prob\bigg( \frac{\|\Pi_T^\perp \epsilon'\|^2}{\sigma^2(n-j)} - \frac{\|\Pi_{\trusupp}^\perp \epsilon\|^2}{\sigma^2(n-\sps)} \ge \frac{3}{4}\signalone - \frac{1}{8}(\signalone\vee\signaltwo +\ell'\signalb)\bigg) \\
            \ge &  1 - 5\exp\bigg(-(n-\ubsps)\frac{\min(\signalone\vee\signaltwo + \ell'\signalb,1)}{64^2} + |T'|\bigg) \,.
        \end{align*}
    \end{lemma}
\begin{lemma}\label{lem:main:twoccase2:unknown:lem2}
        Providing $n\ge |W| + \frac{2C'\max\{|S'|,|T'|\}}{\min(\signalone\vee \signaltwo + \ell'\signalb,1)}$, for some constant $C'$,
        \begin{align*}
            & \prob\bigg( \frac{\widehat{\signal}_2(T) - \widehat{\signal}_2(\trusupp)}{\sigma^2} \ge \frac{3}{4}\signaltwo - \frac{1}{8}(\signalone\vee\signaltwo +\ell'\signalb)\bigg) \\
            \ge&  1 - 3\exp\bigg(-C'(n-|S'|)\min(\signalone\vee\signaltwo + \ell'\signalb,1)+|S'|+|T'|\bigg) \,.
        \end{align*}
    \end{lemma}
Combining Lemma~\ref{lem:main:twoccase2:unknown:lem1} and~\ref{lem:main:twoccase2:unknown:lem2}, it suffices to have $n\gtrsim \ubsps+\frac{1}{\signalb}$ to ensure the conditions are satisfied, because 
    \begin{align*}
        \ubsps + \frac{1}{\signalb} \gtrsim |W| + |T'| + \frac{|S'|+\ell'}{(|S'|+\ell')\signalb} \gtrsim |W| + |T'| + \frac{|S'|+\ell'}{\signalone\vee\signaltwo + \ell'\signalb} \gtrsim |W| + \frac{|T'|}{\min(\signalone\vee\signaltwo+\ell'\signalb,1)}\,.
    \end{align*}
    Note that $\ell'+|S'| \ge \max\{|T'|,|S'|\}$ by definition of $\ell'$ and equality holds when $|T|\ge |\trusupp|$. Let $C_0=\min(C',1/64^2)$, with probability at least
    \begin{align*}
        & 1 - 5\exp\bigg(-(n-\ubsps)\frac{\min(\signalone\vee\signaltwo+\ell'\signalb,1)}{1024} +|T'|\bigg)  \\
        & \ \ - 3\exp\bigg(-C'(n-|W|)\min(\signalone\vee\signaltwo+\ell'\signalb,1)+|S'|+|T'|\bigg) \\
        \ge & 1 - 8\exp\bigg(-C_0(n-\ubsps))\min(\signalone\vee\signaltwo+\ell'\signalb,1) + |S'|+|T'|\bigg) \,,
    \end{align*}
    we have 
    \begin{align*}
        \frac{\mathcal{L}(T;(\trusupp,T)) - \mathcal{L}(\trusupp;(\trusupp,T))}{\sigma^2} & =\frac{\|\Pi_T^\perp \epsilon'\|^2}{\sigma^2(n-j)} - \frac{\|\Pi_{\trusupp}^\perp \epsilon\|^2}{\sigma^2(n-\sps)} + \frac{\widehat{\signal}_2(T) - \widehat{\signal}_2(\trusupp)}{\sigma^2}\\
        & \ge \frac{3}{4}(\signalone+\signaltwo) - \frac{1}{4}(\signalone\vee\signaltwo+\ell'\signalb) \,,
    \end{align*}
    which completes the proof.
\end{proof}
We proceed to show the proofs for Lemma~\ref{lem:main:twoccase2:unknown:lem1} and~\ref{lem:main:twoccase2:unknown:lem2}.
\begin{proof}[Proof of Lemma~\ref{lem:main:twoccase2:unknown:lem1}]
    Start with the same decomposition:
    \begin{align*}
        \frac{\|\Pi_T^\perp \epsilon'\|^2}{\sigma^2(n-j)} - \frac{\|\Pi_{\trusupp}^\perp \epsilon\|^2}{\sigma^2(n-\sps)} & = \frac{\|\Pi_T^\perp (\epsilon+\epsilon_0)\|^2}{\sigma^2(n-j)} - \frac{\|\Pi_{\trusupp}^\perp \epsilon\|^2}{\sigma^2(n-\sps)}  \\
        & = \underbrace{\frac{\|\Pi^\perp_{T}\epsilon_0\|^2}{\sigma^2(n-j)}}_{:=A_1} + \underbrace{\frac{2\langle \Pi^\perp_T\epsilon,\Pi^\perp_T\epsilon_0\rangle}{\sigma^2(n-j)}}_{:=A_2} + \underbrace{\frac{\|(\Pi^\perp_T-\Pi^\perp_{\trusupp})\epsilon\|^2}{\sigma^2(n-j)}}_{:=A_3} + \underbrace{\frac{-(\sps-j)}{n-j}\frac{\|\Pi^\perp_{\trusupp}\epsilon\|^2}{\sigma^2(n-\sps)}}_{:=A_4}.
    \end{align*}
    For $A_1$,
    \begin{align*}
        \prob(A_1 \ge \frac{3}{4}\signalone) & = \prob( \frac{\chi^2_{n-j}}{n-j} \ge \frac{3}{4}) \\
        & = \prob( \frac{\chi^2_{n-j}}{n-j}- 1 \ge -\frac{1}{4}) \\
        & \ge 1 - \exp(- \frac{1}{16^2}(n-j)) \\
        & \ge 1 - \exp(- \frac{1}{16^2}(n-\ubsps)) \,.
    \end{align*}
    For $A_2$,
    \begin{align*}
        \prob\bigg(A_2 \ge - \frac{1}{24}(\signalone\vee \signaltwo + \ell'\signalb) \bigg)& = \prob\bigg(\frac{2\langle \Pi^\perp_T\epsilon,\Pi^\perp_T\epsilon_0\rangle}{\sigma^2(n-j)} \ge - \frac{1}{24}(\signalone\vee \signaltwo + \ell'\signalb)\bigg) \\
        & \ge 1 - 2\prob\bigg(|\frac{\chi^2_{n-j}}{n-j} - 1| \le \frac{1}{48}\sqrt{\signalone\vee \signaltwo + \ell'\signalb}\bigg) \\
        & \ge 1 - 2\exp\bigg(-(n-j) \frac{\min(\signalone\vee \signaltwo + \ell'\signalb,\sqrt{\signalone\vee \signaltwo + \ell'\signalb})}{48^2\times 16}\bigg) \,.
        \end{align*}
    The first inequality is because 
    \begin{align*}
        \bigg(\frac{\signalone\vee \signaltwo + \ell'\signalb}{\sqrt{\signalone}}\bigg)^2 & = \frac{(\signalone\vee \signaltwo )^2+ (\ell'\signalb)^2 + 2(\signalone\vee \signaltwo)(\ell'\signalb)}{\signalone} \\
        & \ge \signalone\vee \signaltwo + 2(\ell'\signalb)\\ 
        & \ge \signalone\vee \signaltwo + \ell'\signalb\,.
    \end{align*}
    For $A_3$, 
    \begin{align*}
         \frac{\|(\Pi^\perp_T-\Pi^\perp_{\trusupp})\epsilon\|^2}{\sigma^2(n-j)}  & =  \frac{\|(\Pi_{\trusupp}-\Pi_{\trusupp\cap T})\epsilon\|^2}{\sigma^2(n-j)} -  \frac{\|(\Pi_T-\Pi_{\trusupp\cap T})\epsilon\|^2}{\sigma^2(n-j)} \\
         & \ge -  \frac{\|(\Pi_T-\Pi_{\trusupp\cap T})\epsilon\|^2}{\sigma^2(n-j)} \\
         & \sim - \frac{ \chi^2_{|T\setminus \trusupp|}}{n-j} = - \frac{ \chi^2_{|T'|}}{n-j} \,.
    \end{align*}
    If $|T'|=0$, then $A_3\ge 0$, otherwise, 
    \begin{align*}
        \prob\bigg(A_3 \ge  - \frac{1}{24}(\signalone\vee \signaltwo + \ell'\signalb)\bigg) & = \prob\bigg( \frac{\chi^2_{|T'|}}{|T'|} - 1\le   \frac{1}{24}(\signalone\vee \signaltwo + \ell'\signalb)\times \frac{n-j}{|T'|} - 1\bigg) \\
        & \ge 1 - \exp\bigg(- (n-j)\frac{\signalone\vee \signaltwo + \ell'\signalb}{96} + |T'|\bigg)\,,
    \end{align*}
    given $n-j \ge \frac{182|T'|}{\signalone\vee \signaltwo + \ell'\signalb}$, which is ensured by
    \begin{align*}
        n-\ubsps \ge \frac{192}{\signalb} \ge \frac{192|T'|}{(|S'|+\ell')\signalb} \ge\frac{192|T'|}{\signalone\vee \signaltwo + \ell'\signalb} \,.
    \end{align*}
    For $A_4$, if $|\trusupp|\le |T|$, i.e. $\sps-j\le 0$, then $A_4\ge 0$, otherwise,
    \begin{align*}
        \prob\bigg(A_4 \ge - \frac{1}{24}(\signalone\vee \signaltwo + \ell'\signalb)\bigg) & = \prob\bigg( - \frac{\sps-j}{n-j}\frac{\|\Pi^\perp_{\trusupp}\epsilon\|^2}{\sigma^2(n-\sps)} \ge- \frac{1}{24}(\signalone\vee \signaltwo + \ell'\signalb) \bigg) \\
        & = \prob\bigg(\frac{\chi^2_{n-\sps}}{n-\sps}-1 \le  \frac{1}{24}(\signalone\vee \signaltwo + \ell'\signalb) \times \frac{n-j}{\sps-j} - 1\bigg) \\
        & \ge 1 - \exp\bigg(-(n-\sps)\bigg[\frac{\signalone\vee \signaltwo + \ell'\signalb}{96}\times \frac{n-j}{\sps-j} - \frac{1}{4}\bigg]\bigg) \\
        & \ge 1 - \exp\bigg(-(n-\sps)\frac{\signalone\vee \signaltwo + \ell'\signalb}{96}\bigg)\,,
    \end{align*}
    where the first inequality requires
    \begin{align*}
        \frac{\signalone\vee \signaltwo + \ell'\signalb}{96}\frac{n-j}{\sps-j} \ge 2 \Leftrightarrow n-j \ge \frac{192(s-j) }{\signalone\vee \signaltwo + \ell'\signalb}\,,
    \end{align*}
    which is ensured by $n - \ubsps \ge \frac{192}{\signalb}$.
    And the second inequality requires
    \begin{align*}
        \frac{\signalone\vee \signaltwo + \ell'\signalb}{96}\times \frac{n-j}{\sps-j} - \frac{1}{4} \ge \frac{\signalone\vee \signaltwo + \ell'\signalb}{96} \Leftrightarrow n-\sps \ge \frac{24(\sps-j)}{\signalone\vee \signaltwo + \ell'\signalb}\,,
    \end{align*}
    which is ensured by 
    \begin{align*}
        n - \ubsps \ge \frac{192}{\signalb}= \frac{192 \ell'}{\signalb\ell'} \ge \frac{24\ell'}{\signalone\vee \signaltwo + \ell'\signalb} = \frac{24(\sps-j)}{\signalone\vee \signaltwo + \ell'\signalb}\,.
    \end{align*}
    Combined these results, we have
    \begin{align*}
        & \prob\bigg( \frac{\|\Pi_T^\perp \epsilon'\|^2}{\sigma^2(n-j)} - \frac{\|\Pi_{\trusupp}^\perp \epsilon\|^2}{\sigma^2(n-\sps)} \ge \frac{3}{4}\signalone - \frac{1}{8}(\signalone\vee\signaltwo +\ell'\signalb)\bigg) \\
        \ge &  1 - 5\exp\bigg(-(n-\ubsps)\frac{\min(\signalone\vee\signaltwo + \ell'\signalb,1)}{188^2} + |T'|\bigg) \,.
    \end{align*}
\end{proof}
\begin{proof}[Proof of Lemma~\ref{lem:main:twoccase2:unknown:lem2}]
    With probability at least $1- 2\exp(-C(n-|W|)t^2 + |S'|+|T'|)$, we have
    \begin{align*}
            \frac{\widehat{\signal}_2(T) - \widehat{\signal}_2(\trusupp)}{\sigma^2} & \ge (1-t)\signaltwo + \bigg(\frac{\|\Pi_{\widetilde{T}'}\widetilde{\epsilon}'\|^2}{\sigma^2(n-|W|)}-\frac{\|\Pi_{\widetilde{S}'}\widetilde{\epsilon}\|^2}{\sigma^2(n-|W|)}\bigg) \\
            & \qquad  -2\frac{\|\Sigma_{T'\given W}^{1/2}(\ptregcoef-\widehat{\alpha}^*)\|\|U_{T'\given W}\T\widetilde{\epsilon}'\|}{\sigma^2(n-|W|)} \\
            & :=  (1-t)\signaltwo + B_1 - B_2 \,.
    \end{align*}
    For $B_1$, if $|S'|=0$, then $B_1\ge 0$, otherwise,
    \begin{align*}
        \prob\bigg(B_1 \le - \frac{1}{16}(\signalone\vee\signaltwo+\ell'\signalb)\bigg)) & \le \prob\bigg(-\frac{\|\Pi_{\widetilde{S}'}\widetilde{\epsilon}\|^2}{n-|W|} \le - \frac{1}{16}(\signalone\vee\signaltwo+\ell'\signalb)\bigg) \\
        & = \prob\bigg(\frac{\chi^2_{|S'|}}{|S'|}-1 \ge  \frac{n-|W|}{|S'|}\frac{1}{16}(\signalone\vee\signaltwo+\ell'\signalb)-1\bigg) \\
        & \le \exp(-(n-|W|)\frac{\signalone\vee\signaltwo + \ell'\signalb}{64} + |S'|) \,,
    \end{align*}
    given $n\ge |W| + \frac{128|S'|}{\signalone\vee\signaltwo+\ell'\signalb}$. For $B_2$, we again invoke Lemma~\ref{lem:main:projnoise} for $\|U_{T'\given W}\T \widetilde{\epsilon}'\|$:
    \begin{align*}
        \prob\bigg(\frac{\|U_{T'\given W}\T \widetilde{\epsilon}'\|^2}{\sigma^2(n-|W|)^2} \ge C_{22}(\signalone\vee\signaltwo + \ell'\signalb)\bigg) & \le \exp\bigg(-(n-|W|)\frac{C_{22}(\signalone\vee\signaltwo + \ell'\signalb)}{4(1+t)(1+\signalone)}+ \frac{|T'|}{4}\bigg)\\
        & + \exp(-C(n-|W|)t^2+|T'|) \,,
    \end{align*}
    for some $C_{22}\in(0,1)$ providing $n\ge |W| + \frac{8|T'|(1+t)(1+\signalone)}{C_{22}(\signalone\vee\signaltwo + \ell'\signalb)}$. We further discuss $\frac{1+\signalone}{\signalone\vee\signaltwo + \ell'\signalb}$ in cases:
    \begin{itemize}
        \item If $\signalone<1$: $\frac{1+\signalone}{\signalone\vee\signaltwo + \ell'\signalb}< \frac{2}{\signalone\vee\signalt2+ \ell'\signalb}$;
        \item If $\signalone\ge 1$: $\frac{1+\signalone}{\signalone\vee\signaltwo+ \ell'\signalb}\le \frac{2\signalone}{\signalone\vee\signaltwo+ \ell'\signalb}\le 2$.
    \end{itemize}
    Thus we only need $n\ge |W| + \frac{16|T'|(1+t)}{C_{22}}(1\vee \frac{1}{\signalone\vee\signaltwo+ \ell'\signalb})$ to ensure
    \begin{align*}
        &\quad \prob\bigg(\frac{\|U_{T'\given W}\T \widetilde{\epsilon}'\|^2}{\sigma^2(n-|W|)^2} \ge C_{22}(\signalone\vee\signaltwo + \ell'\signalb)\bigg) \\
        & \le \exp\bigg(-(n-|W|)\frac{C_{22}\min(\signalone\vee\signaltwo + \ell'\signalb,1)}{8(1+t)}+ \frac{|T'|}{4}\bigg)\\
        & \quad + \exp(-C(n-|W|)t^2+|T'|) \,. 
    \end{align*}
    As shown in the proof of Lemma~\ref{lem:main:twoccase2:lem2}, the event above implies 
    \begin{align*}
        B_2 \le  \frac{8}{1-t}\sqrt{C_{22}}(\signalone\vee\signaltwo+\ell'\signalb)\,,
    \end{align*}
    with probability at least
    \begin{align*}
        1 - \exp\bigg(-(n-|W|)\frac{C_{22}\min(\signalone\vee\signaltwo + \ell'\signalb,1)}{8(1+t)}+ \frac{|T'|}{4}\bigg) + \exp(-C(n-|W|)t^2+|T'|)\,.
    \end{align*}
    Furthermore, let $t = \frac{1}{4},C_{22}=\frac{1}{16^4}$, we have
    \begin{align*}
         \frac{\widehat{\signal}_2(T) - \widehat{\signal}_2(\trusupp)}{\sigma^2} & \ge (1-t)\signaltwo - (\frac{1}{16}+ \frac{8\sqrt{C_{22}}}{1-t})(\signalone\vee\signaltwo+\ell'\signalb)\\
         & \ge \frac{3}{4}\signaltwo - \frac{1}{8}(\signalone\vee\signaltwo+\ell'\signalb)\,,
    \end{align*}
    with probability at least
    \begin{align*}
         1 & - \exp\bigg(-(n-|W|)\frac{\signalone\vee\signaltwo + \ell'\signalb}{64} + |S'|\bigg) \\
        & - \exp\bigg(-(n-|W|)\frac{\min(\signalone\vee\signaltwo + \ell'\signalb,1)}{16^4\times 8 \times 5/4}+ \frac{|T'|}{4}\bigg)\\
        & - \exp(-C(n-|W|)/16+|T'|) \\
        \ge 1 & - 3\exp\bigg(-C'(n-|W|)\min(\signalone\vee\signaltwo+\ell'\signalb,1) + |S'|+|T'| \bigg)\,,
    \end{align*}
    where $C'=\min(\frac{C}{16}, \frac{1}{16^4\times 10})$.
\end{proof}

\subsection{Proof of Lemma~\ref{lem:main:signallb}}\label{app:main:signallb}
\begin{proof}[Proof of Lemma~\ref{lem:main:signallb}]
    Let $\signaltwo(S,T,\alpha_{T\setminus S}):= \Big(\ptregcoef(S,T)- \alpha_{T\setminus S}\Big)\T\Sigma_{T\setminus S\given S\cap T}\Big(\ptregcoef(S,T)- \alpha_{T\setminus S}\Big)/\sigma^2$. Therefore,
    \begin{align*}
        &\quad \var\Big[X_{S\setminus T}\T\beta_{S\setminus T} - X_{T\setminus S}\T\alpha_{T\setminus S}\given S\cap T\Big]  \\
        & = \beta_{S\setminus T}\T \Sigma_{S\setminus T\given S\cap T}\beta_{S\setminus T} + \alpha_{T\setminus S}\T \Sigma_{T\setminus S\given S\cap T}\alpha_{T\setminus S} \\
        & \quad -2\beta_{S\setminus T}\T \Sigma_{(S\setminus T)(T\setminus S)\given S\cap T}\alpha_{T\setminus S} \\
        & = \beta_{S\setminus T}\T \Sigma_{S\setminus T\given T}\beta_{S\setminus T}\\
        & \quad + \beta_{S\setminus T}\T \Sigma_{(S\setminus T)(T\setminus S)\given S\cap T}\Sigma^{-1}_{T\setminus S\given S\cap T}\Sigma_{(T\setminus S)(S\setminus T)\given S\cap T}\beta_{S\setminus T}\\
        & \quad + \alpha_{T\setminus S}\T \Sigma_{T\setminus S\given S\cap T}\alpha_{T\setminus S} - 2\beta_{S\setminus T}\T \Sigma_{(S\setminus T)(T\setminus S)\given S\cap T}\alpha_{T\setminus S} \\
        & = \signalone(S,T)\sigma^2 + \signaltwo(S,T,\alpha_{T\setminus S})\sigma^2\,.
    \end{align*}
Note that 
\begin{align*}
    \frac{1}{2}\Big(\signalone(S,T) + \signaltwo(S,T) \Big)\le \signalone(S,T)\vee \signaltwo(S,T) \le \signalone(S,T) + \signaltwo(S,T)\,,
\end{align*}
with
\begin{align*}
    \signalone(S,T) + \signaltwo(S,T) 
    & =\signalone(S,T) + \min_{\alpha_{T\setminus S}\in\betaspace_{T\setminus S}}\signaltwo(S,T,\alpha_{T\setminus S})  \\
    & = \frac{1}{\sigma^2}\min_{\alpha_{T\setminus S}\in\betaspace_{T\setminus S}}\var\Big[X_{S\setminus T}\T\beta_{S\setminus T} - X_{T\setminus S}\T\alpha_{T\setminus S}\given S\cap T\Big]\,.
\end{align*}
\end{proof}

\section{Neighbourhood selection and support recovery}\label{app:pre:eqvmb}
For completeness, we recall the connection between neighbourhood selection in SEM and support recovery in linear models here. 
Recall the definition of the Markov boundary of $Z_k$ with respect to a set $A\subseteq V$ is the smallest subset $S\subseteq A$ such that $Z_{k}\indep Z_{A\setminus S} \given Z_S$ (i.e. the smallest Markov blanket of $Z_{k}$ in $A$), denoted by $S(k;A)$.
The equivalence of the two problems is given in Proposition~\ref{prop:pre:eqvmb} below. 
\begin{proposition}\label{prop:pre:eqvmb}
If $(Z_k,Z_A)\sim \mathcal{N}(\mathbf{0},\Gamma)$ and $\Gamma$ is a positive definite covariance matrix, then for any subset $S\subseteq A$, the following are equivalent:
\begin{enumerate}
    \item $S = S(k;A)$, where $S(k;A)$ is the Markov boundary defined above;
    \item We have a linear model $Z_k=\beta\T Z_A + \epsilon$ with $\epsilon\indep Z_A$, $\E[\epsilon]=0$, and $\supp(\beta)=S$.
\end{enumerate}
Moreover, suppose $P(Z)$ is an SEM by~\eqref{eq:pre:sem} with $b_{jk}\ne0,\forall j\in\pa(k)$. If $A=\nd(k)$, then $S(k;A)=\pa(k)=\supp(\beta)$. 
\end{proposition}
\noindent
Therefore, Markov boundary learning reduces to a regression problem between $Z_{k}$ and $Z_A$
More specifically, the goal is to recover the support set (i.e. nonzero entries) for this regression problem, i.e. support recovery. Thus, in our setting we have:
\begin{align*}
    \text{neighbourhood selection} 
    \iff 
    \text{Markov boundary learning} 
    \iff
    \text{support recovery in regression}.
\end{align*}
\vspace{-3em}

\begin{proof}[Proof of Proposition~\ref{prop:pre:eqvmb}]
The Markov blanket of $k$ with respect to subset $A\subseteq Z_{-k}$ is a subset $T\subseteq A$ such that $Z_k\indep (A\setminus T)\given T$. The Markov boundary is the smallest Markov blanket.
In the following proof, we will use regression notation by taking $Y:=Z_k, X:=Z_A$. 

\textit{1)$\Rightarrow$ 2):}
Denote $A\setminus S = S^c$. We can write 
\begin{align*}
    Y = \Gamma_{YS}\Gamma_S^{-1}X_S + [Y - \Gamma_{YS}\Gamma_S^{-1}X_S] := \beta_S\T X_S + \epsilon = \beta\T X + \epsilon \,,
\end{align*}
where $\beta$ is a vector such that $\beta_S:=\Gamma_{YS}\Gamma_S^{-1}$, $\beta_j=0$ if $j\in S^c$. Moreover, $\E[\epsilon] = 0 -\beta_S\T 0 = 0 $. Now we want to show $\epsilon\indep S$ and $\epsilon\indep S^c$. By Gaussianity, it suffices to deal with covariance. For the first one,
\begin{align*}
    \E[X_{S}\epsilon] & = \E [X_SY] + \E[X_SX_S\T \beta_S] \\
    & = \Gamma_{SY} + \Gamma_S\beta_S \\
    & = \Gamma_{SY} + \Gamma_S\Gamma_S^{-1}\Gamma_{SY} \\
    & = 0 \,.
\end{align*} 
For the second one, since $Y\indep S^c\given S$, i.e.
\begin{align*}
    0 & = \E_S\cov(Y,S^c\given S) \\
    & = \E(X_{S^c}Y) - \E_S[\E(X_{S^c}\given S)\E(Y\given S)] \\
    & = \E(X_{S^c}Y) - \E_S[\E(X_{S^c}\given S)X_S\T \beta_S]\\
    & = \E(X_{S^c}Y) - \E(X_{S^c}X_S\T \beta_S] \\
    & =\E[X_{S^c}\epsilon] \,.
\end{align*}
The second equality is by definition of conditional covariance; the third equality is by independence between $\epsilon$ and $S$; the fourth equality is by tower rule; the last equality is by definition of $\epsilon$. Finally, we want to show $\beta_j\ne 0$, $\forall j\in S$. By way of contradiction, suppose $\beta_j=0$, denote $R = S\setminus\{j\}$, then $Y = \beta_R\T X_R+\epsilon$. Notice that
\begin{align*}
    &\E_R\cov(Y, S^c\cup\{j\} \given R)
    \\& = \bigg(\E[YX_j] -  \E_R[\E(Y\given R)\E(X_j\given R)], \E[YX_{S^c}] -\E_R[\E(Y\given R)\E(X_{S^c}\given R)] \bigg)\T  \\
    & = \bigg(\E[X_jX_R\T]\beta_R - \E_R[X_R\T\beta_R \E(X_j\given R)], \E[X_{S^c}X_R\T]\beta_R  - \E_R[X_R\T\beta_R \E(X_{S^c}\given R)]\bigg)\T\\
    & =(0,0)\T \,.
\end{align*}
Therefore, $Y\indep S^c\cup\{j\}\given R$, i.e. $R$ is a Markov blanket of $Y$ in $X$, which contradicts the minimality of $S$.
This completes the first half of the proof.

\textit{2)$\Rightarrow$ 1):}
We can compactly write $Y=\beta_S\T X_S + \epsilon$.
For $S$ to be $S(Y;X)$, we need to check whether $Y\indep S^c\given S$ and $S$ is the minimal subset satisfies it. For the first one,
\begin{align*}
    \E_S\cov(Y,S^c\given S)  & = \E[X_{S^c}Y] - \E_S[\E(X_{S^c}\given S)\E(Y\given S)] \\
    & = \E [X_{S^c}X_S\T \beta_S] - \E_S[\E(X_{S^c}\given S)X_S\T \beta_S] \\
    & = \E [X_{S^c}X_S\T \beta_S]-\E [X_{S^c}X_S\T \beta_S]\\
    &=0\,.
\end{align*}
The second equality is by definition of linear model. Now for the second one, by way of contradiction, suppose there is a set $T\subseteq X$ with $|T|<|S|$ such that $Y\indep T^c\given T$. Let $S = (S\cap T)\cup(S\setminus T) :=S_1\cup S_2$. Thus $Y\indep S_2\given T$, i.e.
\begin{align*}
    0& = \E_T\cov(Y,X_{S_2}\given T) \\
    & = \E[YX_{S_2}] - \E_T[\E(Y\given T)\E(X_{S_2}\given T)] \\
    & = \E[X_{S_2}X_{S_1}\T]\beta_{S_1} + \E[X_{S_2}X_{S_2}\T] \beta_{S_2} - \E_T[(X_{S_1}\T \beta_{S_1} + \E(X_{S_2}\T\given T)\beta_{S_2}) \E(X_{S_2}\given T)] \\
    & = \Gamma_{S_2}\beta_{S_2} - \E_T[\E(X_{S_2}\given T)\E(X_{S_2}\T\given T)]\beta_{S_2} \\
    & = (\Gamma_{S_2} - \Gamma_{S_2T}\Gamma^{-1}_T\Gamma_{S_2T})\beta_{S_2}\,,
\end{align*}
Then $\beta_{S_2}\T (\Gamma_{S_2} - \Gamma_{S_2T}\Gamma^{-1}_T\Gamma_{S_2T})\beta_{S_2}=0$, since $\beta_{S_2}\ne 0$, which contradicts the assumption that $\Gamma$ is positive definite, because the principal submatrix and Schur complement of a  positive definite matrix are still positive definite.

Finally, if $P(Z)$ is an SEM generated by a DAG $G$, and $A=\nd(k)$ is the nondescendants, by Markov property, we have $X_k\indep X_{A\setminus\pa(k)} \given \pa(k)$. Thus $\pa(k)$ is a Markov blanket. Because $b_{jk}\ne 0$ for all $j\in\pa(k)$, then $\pa(k)$ is the Markov boundary.
\end{proof}

\section{Concentration of $\chi^2$ random variable}\label{app:chisq}
We start by introducing tail probability bound for centralized $\chi^2$ distribution from \citet{laurent2000adaptive}.
\begin{lemma}\label{lem:chisq:orig}
If $Z\sim \chi^2_m$ with degree $m$, then for any $t\ge 0$,
\begin{align*}
    & \prob\bigg[\frac{Z-m}{m} \ge 2(\sqrt{t}+t)\bigg] \le \exp(-mt) \\
    & \prob\bigg[\frac{Z-m}{m} \le -2\sqrt{t}\bigg] \le \exp(-mt) \,.
\end{align*}
\end{lemma}
One consequence of Lemma~\ref{lem:chisq:orig} is the concentration of $\chi^2$ distribution around its mean.
\begin{lemma}\label{lem:chisq}
If $Z\sim \chi^2_m$ with degree $m$, then for any $t\ge 0$,
\begin{align*}
    \prob\bigg[\frac{|Z-m|}{m} \ge 4t\bigg] \le \exp(-m\min(t,t^2)) \,.
\end{align*}
\end{lemma}
\begin{proof}
If $t \ge 1$, then $2(\sqrt{t} + t)\le 4t$, $-4t \le -2t \le -2\sqrt{t}$, thus
\begin{align*}
    & \prob\bigg[\frac{Z-m}{m} \ge 4t\bigg] \le \prob\bigg[\frac{Z-m}{m} \ge 2(\sqrt{t}+t)\bigg] \le \exp(-mt) \\
    & \prob\bigg[\frac{Z-m}{m} \le -4t\bigg] \le \prob\bigg[\frac{Z-m}{m} \le -2\sqrt{t}\bigg] \le \exp(-mt) \,.
\end{align*}
If $t\in[0,1)$, let $h=t^2 \in [0,1)$, then $2(\sqrt{h}+h) \le 4\sqrt{h}$, $-4\sqrt{h}\le -2\sqrt{h}$, thus
\begin{align*}
    & \prob\bigg[\frac{Z-m}{m} \ge 4t\bigg] = \prob\bigg[\frac{Z-m}{m} \ge 4\sqrt{h} \bigg]\le \prob\bigg[\frac{Z-m}{m} \ge 2(\sqrt{h}+h)\bigg] \le \exp(-mh) = \exp(-mt^2) \\
    & \prob\bigg[\frac{Z-m}{m} \le -4t\bigg]=\prob\bigg[\frac{Z-m}{m} \le -4\sqrt{h}\bigg] \le \prob\bigg[\frac{Z-m}{m} \le -2\sqrt{h}\bigg] \le \exp(-mh) = \exp(-mt^2) \,.
\end{align*}
\end{proof}

\section{Lower bound techniques}\label{app:fano}

For completeness, we state some known lemmas that are used in proving our lower bounds. We start with Fano's inequality:
\begin{lemma}[\citet{yu1997assouad}, Lemma~3]\label{lem:fano}
For a model family $\mclass$ contains $M$ many distributions indexed by $j=1,2,\ldots,M$ such that 
\begin{align*}
\alpha &= \max_{P_j\ne P_k\in\mclass}\mathbf{KL}(P_j\| P_k) \\ 
s& =\min_{P_j\ne P_k\in\mclass}\mathbf{dist}(\theta(P_j),\theta(P_k)) \,,   
\end{align*}
where $\theta$ is a functional of its distribution argument. Then for any estimator $\widehat{\theta}$ for $\theta(P)$,
\begin{align*}
    \inf_{\widehat{\theta}}\sup_{P\in\mclass} \E_P \mathbf{dist}(\theta(P), \widehat{\theta}) \ge \frac{s}{2}\bigg(1 - \frac{\alpha + \log 2 }{\log M}\bigg) \,.   
\end{align*}
\end{lemma}
Set $\theta(P_j)=j$ to be the index, $\mathbf{dist}(\cdot,\cdot) = \mathbf{1}\{\cdot \ne \cdot\}$, consider $P_j$ to be a product measure of degree $n$ for any $P_j\in \mclass$, i.e. $n$ i.i.d. samples. One consequence of Lemma~\ref{lem:fano} is as follows:
\begin{corollary}[Fano's inequality]\label{coro:fano}
For a model family $\mclass$ contains $M$ many distributions indexed by $j=1,2,\ldots,M$ such that 
$\alpha= \max_{P_j\ne P_k\in\mclass}\mathbf{KL}(P_j\| P_k)$. 
If the sample size is bounded as
\begin{align*}
    n \le \frac{(1-2\delta)\log M}{\alpha} \,,
\end{align*}
then for any estimator $\widehat{\theta}$ for the model index:
\begin{align*}
    \inf_{\widehat{\theta}}\sup_{j\in[M]} P_j(\widehat{\theta}\ne j) \ge \delta - \frac{\log 2}{\log M} \,.
\end{align*}
\end{corollary}
\noindent
We also use Le Cam's two point method without proof. See, e.g. 
\citet{tsybakov2009introduction}, Theorem~2.2.
\begin{lemma}\label{lem:lecam}
    For a model family $\mclass$ contains $M$ distributions indexed by $j=1,2,\ldots,M$, for any $\ell,k\in [M]$ and for any estimator $\widehat{\theta}$ for the model index:
    \begin{align*}
        \inf_{\widehat{\theta}}\sup_{j\in[M]} P_j(\widehat{\theta}\ne j) \ge\inf_{\widehat{\theta}}\sup_{j\in\{\ell,k\}} P_j(\widehat{\theta}\ne j) \ge \frac{1}{2} - \frac{1}{2}\sqrt{\frac{\mathbf{KL}(P_\ell\|P_k)}{2}} \,.
    \end{align*}
\end{lemma}

\section{Full experiments and details (Section~\ref{sec:expt})}\label{app:expt}
Here we provide full details of our experiments along with additional experiments to compare \klbss{} and Vanilla \klbss{} (Appendix~\ref{app:expt:fxv}). Finally, Appendix~\ref{app:expt:eval} summarizes the results from all the simulation setups.

\subsection{Simulation setup}\label{app:expt:setup}

For graph types, we generate:
\begin{itemize}
    \item \textit{Erd\"os-R\'enyi (ER)}. Graphs whose edges are selected from all possible $\binom{\dd}{2}$ edges independently with specified expected number of edges;
    \item \textit{Scale-Free network (SF)}. Graphs simulated according to the Barabasi-Albert model;
    \item \textit{Bipartite graph}. Generated as follows:    \begin{enumerate}
        \item Randomly divide $[\dd]$ into $V_1$ and $V_2$;
        \item Let $\widetilde{\sps} = \min\{\sps,|V_1|\}$;
        \item For each $j\in V_2$, randomly sample the number of parents $|\pa(j)|$ from $[\widetilde{\sps}]$;
        \item Randomly sample $|\pa(j)|$ many nodes from $V_1$ to be $\pa(j)$;
        \item Randomly permute the nodes.
    \end{enumerate}
    \item \textit{Complete graph}. Graphs with all possible $\binom{\dd}{2}$ edges. Nodes are randomly permuted.
\end{itemize}
We generated graphs from ER and SF with $\{\dd,2\dd,4\dd\}$ edges each, which are denoted as XX-$k$ where $\text{XX}\in\{\text{ER},\text{SF}\}$ denotes graph type and $k$ denotes the average number of edges (i.e. expected total number of edges is $k\dd$). 

Given the DAG $G$, the data $(X,Y)$ is then generated by
\begin{align*}
    X_k&=\sum_{j\in\pa_G(k)}b_{jk} X_j + \epsilon_{k}\\
    Y&=\sum_{\ell\in\trusupp} \beta_\ell X_\ell+\epsilon\,,
\end{align*}
with $\trusupp$ randomly sampled from $[\dd]$, $\beta_\ell=\Rad_\ell\times \betam$, $b_{jk}\sim \Rad_{jk} \times \Unif(b_{\min},b_{\max})$ where $\Rad_\ell,\Rad_{jk}$ are independent Rademacher random variables and $b_{\max}=5,\betam=b_{\min}=0.1$. No effort is made to enforce our assumptions, to avoid path cancellation, etc.

For the noise distributions, we consider different distributions centered at zero: \{Gaussian, t, Uniform, Laplace\}. Scale the random variable such that $\var(\epsilon_k)=\sigma^2_k$ and $\var(\epsilon)=\sigma^2$. Let $\sigma=1$, $\sigma_k\sim \Unif(\sigma_{\min},\sigma_{\max})$ with $\sigmam=0.5,\sigma_{\max}=1$.
We also consider ``mixed'' noise distributions where for each of $\epsilon_k$ and $\epsilon$, we randomly choose one distribution from the four above and sample from it.

Finally, we generated random datasets with sample size $n\in\{1000, 2000, \ldots , 8000\}$ for number of nodes $\dd\in\{8,9,10\}$ and sparsity level $\sps\in\{2,3,4\}$. 
For $\dd=50$, we consider $\sps\in\{10,15,20\}$, and set $b_{\max}=2$ instead of $5$ to avoid numerical issues.
Especially, we set $b_{\max}=1$ for Complete graph. We implement both \klbss{} and \bss{} using the MIP detailed in Section~\ref{sec:prac:mip} with $M=10$. We use \texttt{Gurobi} with tolerance parameter set to be \texttt{MIPGap=1e-9}.
For the Lasso, we use \texttt{sklearn} with \texttt{n\_alphas=500}. 

For unknown sparsity, we consider three additive penalties $\tau$ given by:
\begin{itemize}
    \item \textit{BIC}: $\tau = \frac{\log n}{n}$;
    \item \textit{EBIC}: $ \tau = \frac{\log \dd}{n}$;
    \item \textit{Delta}: $\tau = \frac{\sigmam^2\betam^2}{4}$.
\end{itemize}
The last one (\textit{Delta}) is the theoretical choice given in Appendix~\ref{sec:main:unknown} when the model satisfies the optimality condition, which is realized as $\sigmam^2\betam^2/4 = 0.5^2\times 0.1^2 / 4$ in our setup.
For unknown $\betam$, we consider candidate choices for $\widetilde{\beta}_{\min}$: $\{\betam^\ell\}_{\ell=1}^L = 10^{-2.4,-2.2,-2,\ldots,0.2,0.4,0.6}$ with $K=5$-fold CV.

We run the experiments for the same data generating process in the low dimensions ($\dd\in\{8,9,10\}$) with $\ubsps=4$; and high dimensions ($\dd=50$) with $\ubsps=25$.

For evaluation, we consider four metrics based on the estimated support $\estsupp$: 
\begin{itemize}
    \item \textit{Recovery probability}: $\mathbbm{1}\{\estsupp=\trusupp\}$;
    \item \textit{Hamming distance}: $|\trusupp\setminus \estsupp| + |\estsupp\setminus\trusupp|$;
    \item \textit{False discovery rate}: $\frac{|\estsupp\setminus\trusupp|}{|\estsupp|}$;
    \item \textit{True positive rate}: $\frac{|\estsupp\cap\trusupp|}{|\trusupp|}$.
\end{itemize}
The results are reported by average over $N=200$ replications.

\subsection{Real data application}\label{app:expt:real}
\subsubsection{Selection of genes}
We start by removing all the genes with variances smaller than 0.01; then select the top 25 genes based on their marginal variances
and group the remaining genes according to their variances in ascending order into $\dd-25$ bins.
Then for each replication, we randomly sample one gene from each bin to form the $X$ (of dimension $\dd$). 
We consider $\dd=50,60,70,80,90$ with $\sps=10$. Let $\trusupp=(2, 4, 6,\ldots, 18, 20)$. We randomly shuffle the rows of $X$, center every gene, and let
\begin{align*}
    Y = X\beta + \epsilon
\end{align*}
where $\beta_j = \betam\times\Rad_j$, $\Rad_j$ are i.i.d. Rademacher random variables, $\epsilon_i\sim\mathcal{N}(0,1)$, and $\betam=0.1$.
Apply \klbss{} and \bss{} on this data and evaluate the performance for $n=200,200,\ldots,800$. The results are reported as average of $N=200$ replications.

\subsubsection{Prediction performance}
We pick the gene with the largest variance as $Y$, then the remaining genes as candidate $X$'s to explain $Y$.
For each replication, we randomly choose $\dd=50$ genes, and randomly split the dataset in training set $\mathcal{D}_0$ and test set $\mathcal{D}_1$. Apply \klbss{} and \bss{} with $\sps=10$ on the training set to estimate the support $\estsupp$ and coefficients $\widehat{\beta}_{\estsupp}$. Finally, evaluate the estimate using prediction error on the test set:
\begin{align*}
    \frac{1}{n_1}\sum_{i\in \mathcal{D}_1}(Y_i - X_{i\estsupp}\T\widehat{\beta}_{\estsupp})^2 \,.
\end{align*}
For both \bss{} and \klbss{}, we use the MIP implementation for this exercise. In particular, we apply the method in Section~\ref{sec:prac:unknown} for the choice of $\betam$ for \klbss{}.
We run for $N=100$ replications.

\subsection{Details in Section~\ref{sec:expt:param}}\label{app:expt:param}
\subsubsection{Effect of unknown sparsity}\label{app:expt:param:unknown}
We investigate the effect of different choices of $\ubsps$ on the performance of \klbss{} when true sparsity is unknown.
We take one setup from Appendix~\ref{app:expt:setup}: $\dd=7,\sps=3$ and SF-2 graph, and the additive penalties in Section~\ref{sec:prac:unknown}. 
We run experiments for \klbss{} with all possible valid choices $\ubsps=3,4,5,6,7$ to show the robustness.

\subsubsection{CV for choice of $\betam$}\label{app:expt:param:betam}
We valid the usage of CV and study the misspecification of $\betam$. Consider the same experiment setup as previous exercise: $\dd=7,\sps=3$ and SF-2 graph. We consider candidate choices for $\widetilde{\beta}_{\min}$: $\{\betam^\ell\}_{\ell=1}^L = 10^{-2.4,-2.2,-2,\ldots,0.2,0.4,0.6}$. Apply $K=5$-fold CV for each estimators.
We also include the performance of \klbss{} input with each of $\betam^\ell$ to see the effect of misspecification.

\subsubsection{Time complexity}\label{app:expt:param:timecomp}
We investigate the time complexity of \klbss{} when MIP is applied. One important parameter of MIP is the MIP gap, which is essentially a tolerance of the precision of solution. It also serves as a trade-off between the time complexity and the recovery performance of the solution. 
We record the time used in solving the programming to MIP gap smaller than $0.01$.
Consider ER-2 graphs with Gaussian noise, $n=5000$, $\sps=10$ and $\dd\in\{20,30,\ldots,100,200,500,1000\}$, and the result is averaged over 50 replications.

\subsection{Structure learning}\label{app:expt:dag}
We perform experiments to compare the performance of \bss{} and \klbss{} on structure learning. We take the same experiment setup described in Appendix~\ref{app:expt:setup} with ER/SF-2 graphs to generate $X$. We set $\sigma_k\equiv 2$ for all $k$.
The learning produce is as follows: 
Initialize $\widehat{G}$ as an empty graph.
Given $G$ and data $X$, we either take one valid topological ordering of $G$, or apply EqVar algorithm \citep{chen2019causal} for ordering estimation, denoted as $\pi$ (a permutation of $[\dd]$). For each $k=2,3\ldots,\dd$, we apply \bss{}/\klbss{} with unknown sparsity and $\ubsps=\text{deg}(G)+1$ and BIC penalty to conduct support recovery for $X_{\pi_k}$ from $X_{\pi_{[1:k-1]}}$, which is used as estimate for parents of $\pi_k$ in $\widehat{G}$. Finally, we evaluate the performances by structural Hamming distance (SHD) between $G$ and $\widehat{G}$.

\subsection{Additional metrics}\label{app:expt:add}
For comparison, at the end of this appendix we have included results for other metrics for the setups in Figure~\ref{fig:main1} as described in Appendix~\ref{app:expt:setup}.

\begin{itemize}
    \item Figure~\ref{fig:main1:hd}: Hamming distance results for ER-4, SF-4 and Complete graph types, known and unknown sparsity, $(\dd,\sps,\ubsps) = (10,3,4) \AND (50,10,25)$;
    \item Figure~\ref{fig:main1:fdr}: FDR results for ER-4, SF-4 and Complete graph types, known and unknown sparsity, $(\dd,\sps,\ubsps) = (10,3,4) \AND (50,10,25)$;
    \item Figure~\ref{fig:main1:tpr}: TPR results for ER-4, SF-4 and Complete graph types, known and unknown sparsity, $(\dd,\sps,\ubsps) = (10,3,4) \AND (50,10,25)$.
\end{itemize}

\subsection{Overall comparison}\label{app:expt:eval}
\begin{figure}[t]
    \centering
    \includegraphics[width=1.\linewidth]{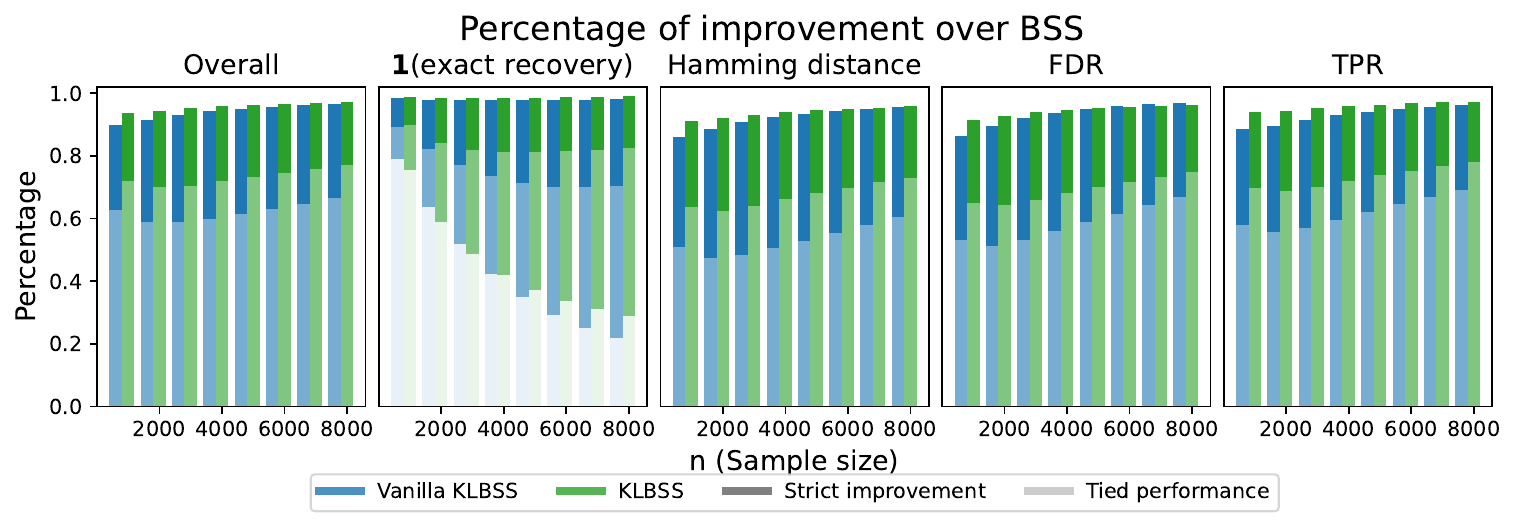}
    \caption{Detailed comparison of \klbss{} vs. \bss{}. Overall evaluation by percentage of improvement over \bss{} on a per-dataset basis (i.e. not averaged for all replications in each setting). The solid bars indicate percentage of strict improvement. The transparent bars above the solid bars indicate the percentage of tied performance. For exact support recovery metric (the second panel from left), if representing the outcomes of \klbss{} and \bss{} as a tuple, we count $(1,0)$ as strict improvement, $(1,1)$ as tied performance, and the most transparent bars indicate percentage of $(0,0)$.}
    \label{fig:eval2}
\end{figure}
To visualize the improvement over \bss{} and summarize results from our comprehensive experiments in a succinct way, we use the percentage of runs where \klbss{} outputs a strictly better result (across four different metrics) over \bss{} across all replications (simulated datasets) in all experiment setups for evaluation. For each sample size $n$ and each replication, we compute the metric (successful recovery indicator, Hamming distance, FDR and TPR) and calculate the percentages of \klbss{} (and Vanilla \klbss{}) giving tied and better metric against \bss{} averaged over all datasets, whose total number is 200 (number of replications per setup) $\times$ 8 (number of graph types) $\times$ 5 (number of noise distributions) $\times$ 12 (number of pairs of $(\dd,\sps)$, 9 for \klbss{} since not implemented for $\dd=50$) $\times$ 4 (known and unknown sparsity with BIC/EBIC/Delta penalty) $=384,000$. 
The result is shown in Figure~\ref{fig:eval2}, where we compute the overall percentage and also separately for each metric.
The percentages of strict improvement and tied performance are plotted via solid and transparent bars, respectively.
In particular, for exact support recovery metric, if the outcomes are denoted as a tuple $(\mathbbm{1}\{\estsupp^{\klbss{}}=\trusupp\}, \mathbbm{1}\{\estsupp^{\bss{}}=\trusupp\})$, then we count $(1,0)$ as strict improvement, $(1,1)$ as tied performance, and the most transparent bars (bottom) indicate the percentage of $(0,0)$.
We can see \klbss{} strictly improves \bss{} on around 20\%-30\% of the simulated datasets, and gives equivalent performance as \bss{} for nearly all of the rest datasets.

\subsection{Performance in SEM with growing degree}\label{app:expt:growdeg}
We compare \klbss{} and \bss{} in the context of SEMs with growing degree to empirically verify the theoretical discussion of Example~\ref{exmp:gap} in Section~\ref{sec:analysis:sem}. We construct the graph by embedding the growing degree structure in Example~\ref{exmp:gap} into a ER-4 graph with $\dd=50$. Specifically, we start by adding edges $X_1\to X_k$ for $k=2,3,\ldots,\sps+1$, then randomly select $4\dd-\sps$ edges from the remaining possible edges. Finally, randomly permute the nodes then obtain a graph with growing degree. The SEM is generated as in Appendix~\ref{app:expt:setup} with $\trusupp$ randomly sampled from $[\dd]$. We consider ``mixed'' noise distribution and $\sps=10,15,20$, and unknown sparsity with $\ubsps=25$ and BIC. The results (including CV) are shown in Figure~\ref{fig:growdeg}, which demonstrate the improvement of \klbss{} over \bss{} in the setting of growing degree SEMs.

\begin{figure}[h]
    \centering
    \includegraphics[width=.8\linewidth]{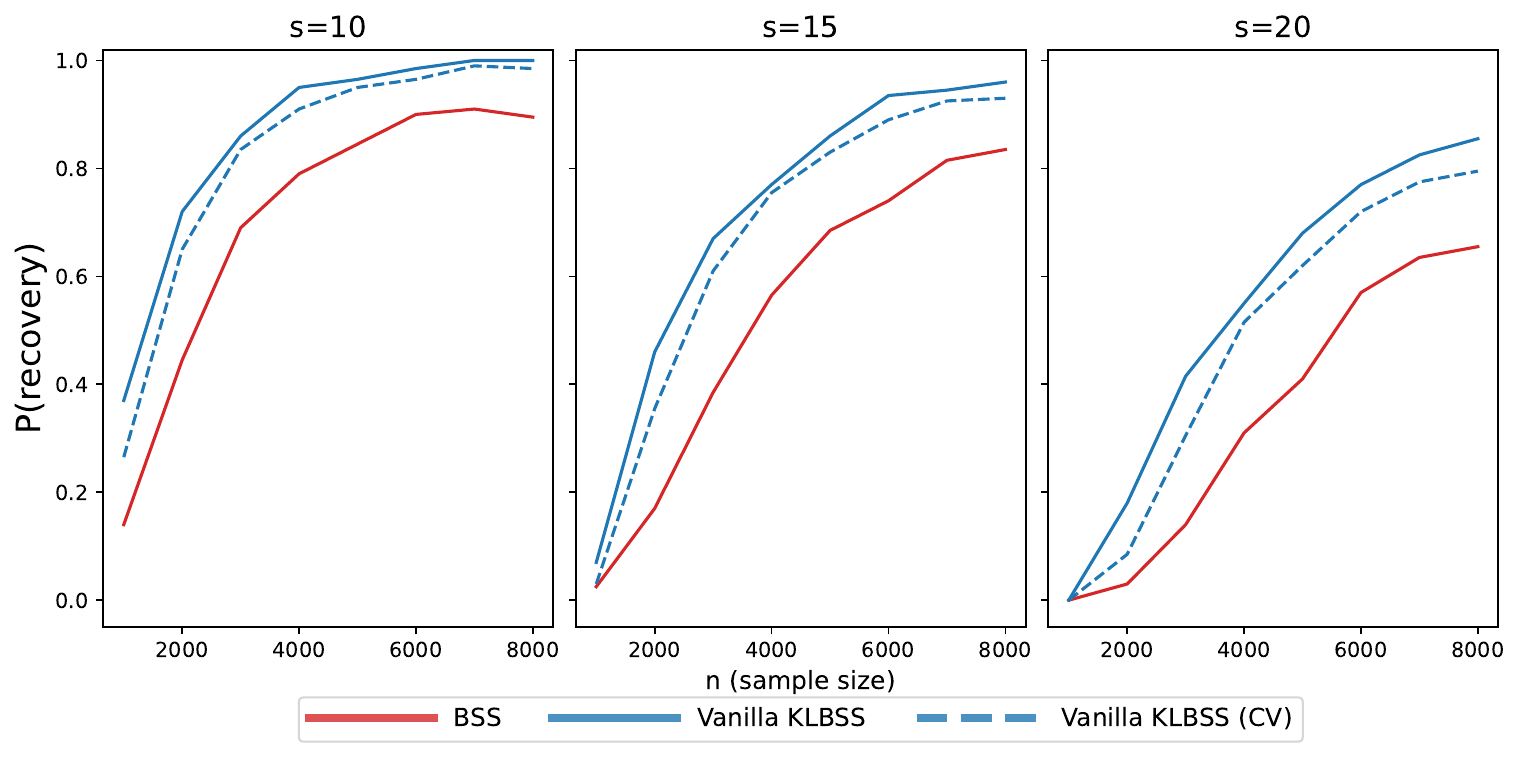}
    \caption{Comparison between \klbss{} and \bss{} in SEMs with growing degree as in Example~\ref{exmp:gap}. Both methods are implemented with unknown sparsity and $\ubsps=25$ and BIC. CV results are presented in dashed lines. The improvement of \klbss{} persists in growing degree SEMs.}
    \label{fig:growdeg}
\end{figure}

\subsection{Vanilla \klbss{} and \klbss{}}\label{app:expt:fxv}
In this appendix, we use random SF graphs to showcase a setting where \klbss{} has advantage over Vanilla \klbss{}. 
Vanilla \klbss{} requires a factor of $\sps$ in the sample complexity (Theorem~\ref{thm:main:ub:vanilla}), which comes from the covariance matrix estimation. In the cases where (certain submatrices of) $\Sigma=\cov(X)$ are hard to estimate, \klbss{} will give better performance. 
Here we consider to increase the upper limit of the linear coefficients in generating $X$, and remove the Rademacher random multiplier before them (Appendix~\ref{app:expt:setup}), which means the linear coefficients are always positive and makes $X$'s highly asymmetric in their (co)variances, and the maximum eigenvalue of $\Sigma$ is large.

Concretely, consider the following experiment setup: $\dd=7,\sps=3$ and an SF-2 graph. The only difference is $b_{jk}\sim \Unif(b_{\min},b_{\max})$ are all positive with $b_{\min}=0.1,b_{\max}=15$, and $\sigma_{\max}=2$.
The results are shown in Figure~\ref{fig:fxv}, from which we can see the slightly better performance of \klbss{} compared to Vanilla \klbss{}. This is an example of a hard instance that encapsulates the \emph{minimax} (i.e. worst-case) behaviour in neighbourhood selection, but is not indicative of the \emph{average} (similarly, pointwise) behaviour, where Vanilla \klbss{} is better.

\begin{figure}[h]
    \centering
    \includegraphics[width=.7\linewidth]{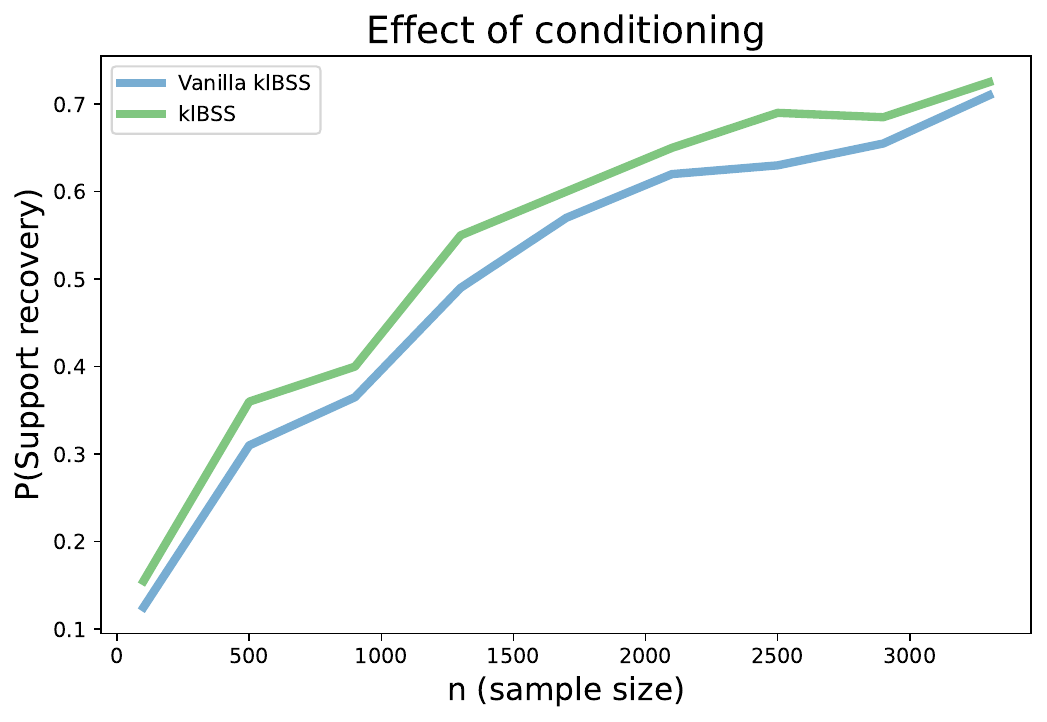}
    \caption{Comparison between \klbss{} and Vanilla \klbss{} on covariance matrix with large maximum eigenvalue. \klbss{} gives better performance when the covariance (sub)matrix is relatively harder to estimate.}
    \label{fig:fxv}
\end{figure}

\begin{figure}[t]
    \centering
    \includegraphics[width=1.\linewidth]{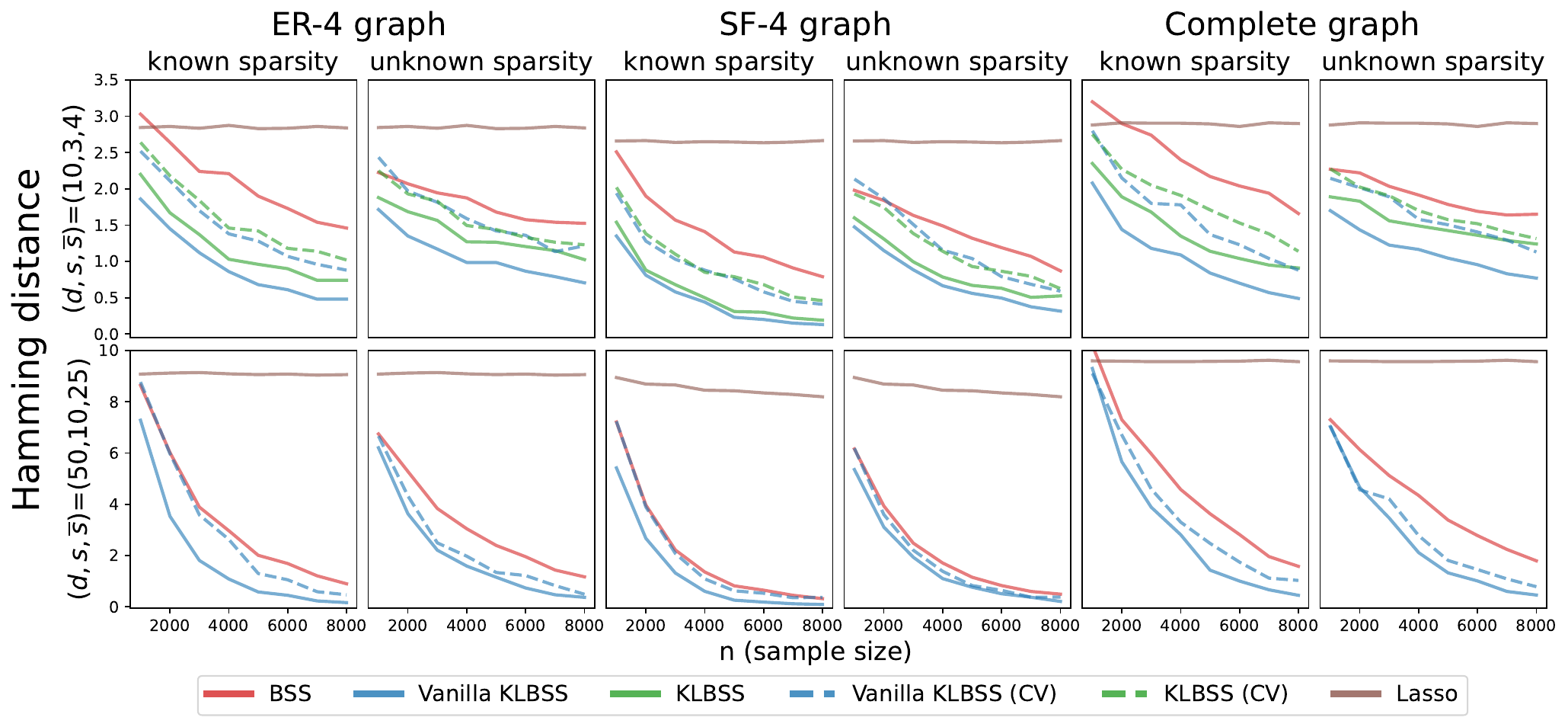}
    \caption{Experiment results of Hamming distance for the same setups with Figure~\ref{fig:main1}.}
    \label{fig:main1:hd}
\end{figure}

\begin{figure}[t]
    \centering
    \includegraphics[width=1.\linewidth]{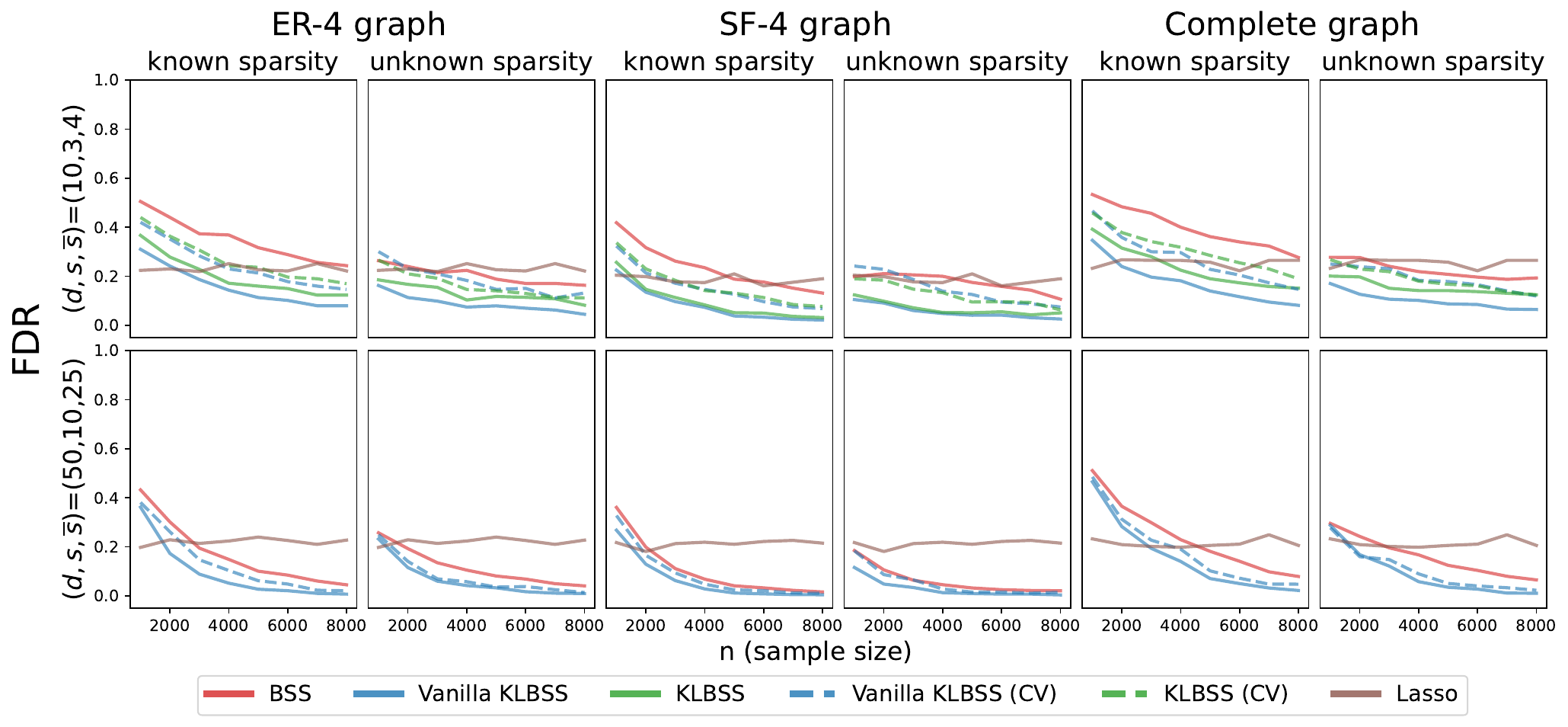}
    \caption{Experiment results of FDR for the same setups with Figure~\ref{fig:main1}.}
    \label{fig:main1:fdr}
\end{figure}

\begin{figure}[t]
    \centering
    \includegraphics[width=1.\linewidth]{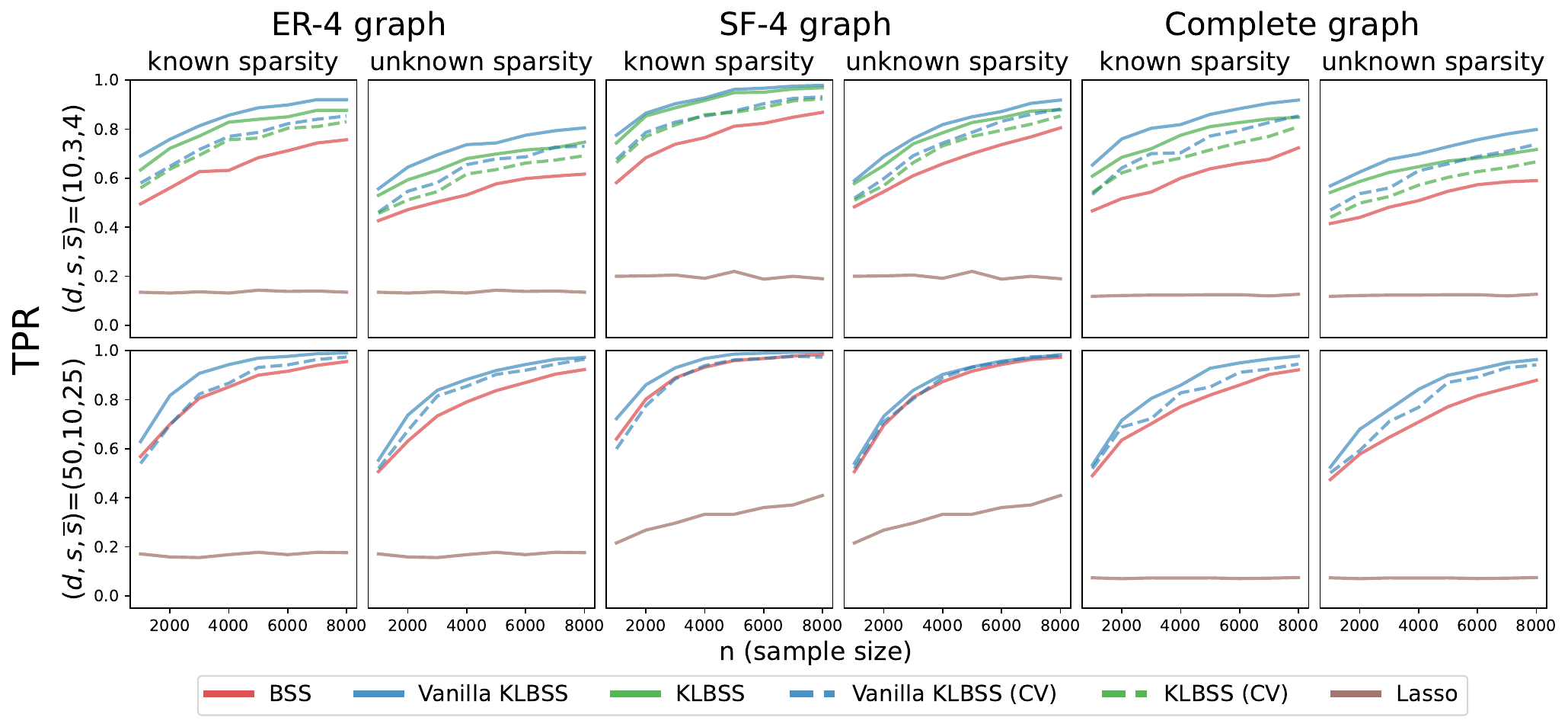}
    \caption{Experiment results of TPR for the same setups with Figure~\ref{fig:main1}.}
    \label{fig:main1:tpr}
\end{figure}

\end{document}